\documentclass[a4paper,11pt]{article}

\pdfoutput=1

\usepackage{a4wide}
\usepackage{amsmath, amsthm, amssymb}
\usepackage{fontenc}
\usepackage[latin1]{inputenc}
\usepackage[pdftex]{graphicx}
\graphicspath{{img/PET/}{img/SpectralCT/}}
\DeclareGraphicsExtensions{.pdf,.jpeg,.png}
\usepackage{pstricks}
\usepackage{algorithmic, algorithm}

\usepackage{tikz}
\usepackage{multirow}

\usepackage{url}

\usepackage{caption}
\usepackage{subcaption}

\usepackage{enumerate}

\newtheorem{theorem}{Theorem}[section]

\newtheorem{proposition}[theorem]{Proposition}
\newtheorem{lemma}[theorem]{Lemma}
\newtheorem{remark}[theorem]{Remark}
\newtheorem{example}[theorem]{Example}
\newtheorem{corollary}[theorem]{Corollary}

\newcommand{\R}{\mathbb{R}}

\def\tT{{\mbox{\tiny{T}}}}
\def\argmin{\mathop{\rm argmin}}
\def\argmax{\mathop{\rm argmax}}
\def\prox{{\rm prox}}
\def\dom{{\rm dom}}
\def\sup{\mathop{\rm sup}}

\def\lb{\langle}
\def\rb{\rangle}

\title{First Order Algorithms in Variational Image Processing}
\author{
M. Burger\thanks{University of M\"unster, Department of Mathematics and Computer Science, 48149 M\"unster, Germany}, A. Sawatzky\footnotemark[1], \
and
G. Steidl\thanks{University of Mannheim, Dept. of Mathematics and Computer Science, A5, 68131 Mannheim, Germany}
}
\begin{document}

\maketitle

\section{Introduction} \label{sec:intro}
Variational methods in imaging are nowadays developing towards a quite universal and flexible tool, 
allowing for highly successful approaches on tasks like denoising, deblurring, inpainting, segmentation,
super-resolution, disparity, and optical flow estimation. 
The overall structure of such approaches is of the form
$$ {\cal D}(Ku) + \alpha  {\cal R} (u) \rightarrow \min_u, $$
where the functional ${\cal D}$ is a data fidelity term also depending on some input data $f$ and measuring the deviation of $Ku$ 
from such  and ${\cal R}$ is a regularization functional. Moreover $K$ is a (often linear) forward operator modeling the dependence 
of data on an underlying image, and $ \alpha$ is a positive regularization parameter.
While ${\cal D}$ is often smooth and (strictly) convex, 
the current practice almost exclusively uses nonsmooth regularization functionals. 
The majority of successful techniques is using nonsmooth and convex functionals like the total variation and generalizations 
thereof, cf. \cite{bredies2010total,burger2013guide,CCCNP2010}, 
or $\ell_1$-norms of coefficients arising from scalar products with some frame system, 
cf. \cite{fornasier2010theoretical} and references therein. 

The efficient solution of such variational problems in imaging demands for appropriate algorithms. Taking into account 
the specific structure as a sum of two very different terms to be minimized, splitting algorithms 
are a quite canonical choice. Consequently this field has revived the interest in techniques 
like operator splittings or augmented Lagrangians.
In this chapter we shall provide an overview of methods currently developed and recent results as well as some computational studies 
providing a comparison of different methods and also illustrating their success in applications.  

We start with a very general viewpoint in the first sections, discussing basic notations, properties of proximal maps, 
firmly non-expansive and averaging operators, which form the basis of further convergence arguments. 
Then we proceed to a discussion of several state-of-the art algorithms and their (theoretical) convergence properties.
 After a section discussing issues related to the use of analogous iterative schemes for ill-posed problems, 
we present some practical convergence studies in numerical examples related to PET and spectral CT reconstruction.

\section{Notation} \label{sec:notation}
In the following we summarize the notations and definitions that will be used throughout the presented chapter:
\begin{itemize}
\item $x_+ := \max\{x,0\}$, $x \in \R^d$, whereby the maximum operation has to be interpreted componentwise.
\item
$\iota_C$ is the indicator function of a set $C \subseteq \R^d$ given by
$$
\iota_C(x) := 
\left\{
\begin{array}{rl}
0&{\rm if} \; x \in C,\\
+\infty&{\rm otherwise}.
\end{array}
\right.
$$
\item
$\Gamma_0(\mathbb R^d)$ is a set of proper, convex, and lower semi-continuous functions mapping from $\R^d$ into the extended
real numbers $\R \cup \{+\infty\}$.
\item 
${\rm dom} f := \{x \in \mathbb R^d: f(x) < + \infty\}$ denotes the {\it effective domain}\index{effective domain} of $f$.
\item
$\partial f(x_0) := \{p \in \mathbb R^d: f(x) - f(x_0) \ge \langle p, x-x_0 \rangle \; \forall x \in \mathbb R^d\}$ 
denotes the {\it subdifferential}\index{subdifferential} of $f\in \Gamma_0(\R^d)$  
at  $x_0 \in {\rm dom} f$ and is the set consisting of
the
{\it subgradients}\index{subgradients} of $f$ at $x_0$.
If $f \in  \Gamma_0(\R^d)$ is differentiable at $x_0$, then
$
\partial f(x_0) = \{ \nabla f(x_0) \}.
$
Conversely, if  $\partial f(x_0)$ contains only one element then $f$ is differentiable at $x_0$ 
and this element is just the gradient of $f$ at $x_0$.
By {\it Fermat's rule},
$\hat {x}$ is a global minimizer of $f \in \Gamma_0(\R^d)$ if and only if
\[
0 \in \partial f(\hat{x}).
\]
\item
$f^*(p) := \sup_{x \in \mathbb R^d} \{ \langle p,x\rangle - f(x) \}$ is the (Fenchel) {\it conjugate}\index{conjugate@(Fenchel) conjugate} of $f$. 
For proper $f$, we have $f^* = f$  if and only if $f(x) = \frac12 \|x\|_2^2$.
If $f \in \Gamma_0(\R^d)$ is {\it positively homogeneous}, i.e., $f(\alpha x) = \alpha f(x)$ for all $\alpha > 0$, 
then 
$$f^* (x^*) = \iota_{C_f} (x^*), \quad C_f := \{x^* \in \R^d: \langle x^*,x\rangle \le f(x) \; \forall x \in \R^d\}.$$
In particular, the  conjugate functions of  $\ell_p$-norms, $p \in [1,+\infty]$,  are given by
\begin{equation} \label{dual_norms}
\|\cdot\|_{p}^{*} (x^*) = \iota_{B_q(1)}(x^*)
\end{equation}
where $\frac{1}{p} + \frac{1}{q} = 1$ and as usual $p=1$ corresponds to $q=\infty$ and conversely,
and 
$B_q(\lambda) := \{ x \in \mathbb R^d: \|x\|_q \le \lambda\}$
denotes the ball of radius $\lambda >0$ with respect to the $\ell_q$-norm.
\end{itemize}

\section{Proximal Operator} \label{sec:prox_operator}
%
The algorithms proposed in this chapter to solve various problems
in digital image analysis and restoration have in common that they basically reduce to
the evaluation of a series of proximal problems.
Therefore we start with these kind of problems. For a comprehensive overview
on proximal algorithms we refer to \cite{PB2013}.

\subsection{Definition and Basic Properties} \label{subsec:prox_operator_basics}
For $f\in \Gamma_0(\mathbb R^d)$ and $\lambda >0$,
the {\it proximal operator} \index{proximation mapping} 
$\prox_{\lambda f}: \mathbb R^d \rightarrow \mathbb R^d$ of $\lambda f$ is defined by
\begin{equation} \label{def:prox}
\prox_{\lambda f} (x) := \argmin\limits_{y \in \mathbb R^d} 
\left\{ \frac{1}{2\lambda} \|x-y\|_2^2 +  f(y)\right\}.
\end{equation}
It compromises between minimizing $f$ and being near to $x$, where  $\lambda$ is the trade-off parameter between these terms.
The {\it Moreau envelope} \index{Moreau envelope} or {\it Moreau-Yoshida regularization} \index{Moreau-Yoshida regularization} 
$\; {}^\lambda \!f: \mathbb R^d \rightarrow \mathbb R$ 
is given by 
$$
{}^\lambda \! f (x) := \min\limits_{y \in \mathbb R^d} 
\left\{ \frac{1}{2\lambda} \|x-y\|_2^2 + f(y) \right\}. 
$$
A straightforward calculation shows that
${}^\lambda \! f = (f^* + \frac12 \|\cdot \|^2_2)^*$.
The following theorem ensures that the minimizer in \eqref{def:prox} exists, is unique and 
can be characterized by a variational inequality. 
The Moreau envelope can be considered as a smooth approximation of $f$.
For the proof we refer to \cite{Au03}.
%
\begin{theorem} \label{prox_1}
Let $f\in  \Gamma_0(\mathbb R^d)$. Then,
\begin{itemize}
\item[{\rm i)}]
For any $x \in \mathbb R^d$, there exists a unique minimizer
$\hat x = {\rm prox}_{\lambda f}(x)$
of \eqref{def:prox}.
\item[{\rm ii)}] The variational inequality \index{variational inequality}
\begin{equation} \label{basic_problem_1}
\frac{1}{\lambda} \langle x- \hat x, y- \hat x  \rangle + f(\hat x) - f(y) \le 0
\qquad \forall y \in \mathbb R^d.
\end{equation}
is necessary and sufficient for $\hat x$ to be the minimizer of \eqref{def:prox}.
\item[{\rm iii)}] $\hat x$ is a minimizer of $f$ if and only if it is a fixed point of ${\rm prox}_{\lambda f}$, i.e.,
$$
\hat x = {\rm prox}_{\lambda f} (\hat x).
$$
\item[{\rm iv)}]  ~The Moreau envelope ${}^\lambda \! f$ is continuously differentiable with gradient
\begin{equation} \label{grad_env}
\nabla \big( {}^\lambda f \big) (x) = \frac{1}{\lambda} \left( x - {\rm prox}_{\lambda f} (x) \right).
\end{equation}
\item[{\rm v)}] The set of minimizers of $f$ and ${}^\lambda f$ are the same. 
\end{itemize}
Rewriting iv) as
$
 {\rm prox}_{\lambda f} (x) = x - \lambda \nabla \big( {}^\lambda f \big) (x)
$
we can interpret ${\rm prox}_{\lambda f}(x)$ as a gradient descent step with step size $\lambda$ for minimizing
${}^\lambda f$.
\end{theorem}
%


\begin{example} \label{ex:prox} {\rm
Consider the univariate function $f(y) := |y|$
and
$$
{\rm prox}_{\lambda f} (x) 
= 
\argmin\limits_{y \in \mathbb R}  \left\{ \frac{1}{2 \lambda} (x-y)^2 + |y| \right\}.
$$
Then, a straightforward computation yields that
${\rm prox}_{\lambda f}$ is the {\it soft-shrinkage} function \index{soft-shrinkage} $S_\lambda$ with threshold $\lambda$ (see Fig. \ref{fig:soft_shrinkage})
defined by
$$
S_\lambda(x):= 
(x- \lambda)_+ - (-x-\lambda)_+ 
=
\left\{
\begin{array}{cl}
x - \lambda& {\rm for} \; x > \lambda,\\
0         & {\rm for} \; x \in [-\lambda,\lambda],\\
x + \lambda& {\rm for} \; x < -\lambda.
\end{array}
\right.
$$
Setting $\hat x := S_\lambda(x) = \prox_{\lambda f}(x)$, we get
$$
{}^\lambda \! f (x)
= |\hat x | + \frac{1}{2\lambda} (x- \hat x)^2
=
\left\{
\begin{array}{cl}
x - \frac{\lambda}{2} & {\rm for} \; x > \lambda,\\[0.5ex]
\frac{1}{2\lambda} x^2         & {\rm for} \; x \in [-\lambda,\lambda],\\[0.5ex]
- x - \frac{\lambda}{2}& {\rm for} \; x < -\lambda.
\end{array}
\right.
$$
This function ${}^\lambda \! f$ is known as {\it Huber function} (see Fig. \ref{fig:soft_shrinkage}).
}
\end{example}

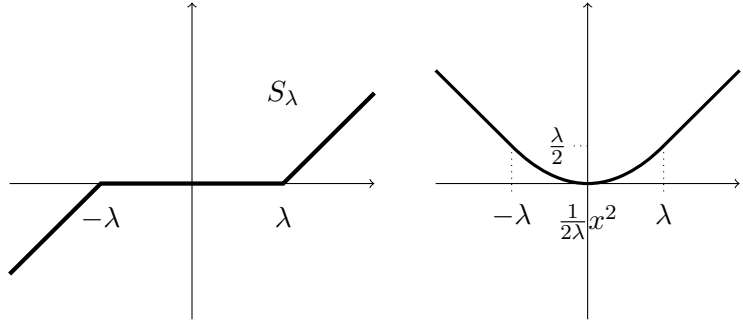
\begin{figure}[htbp]
\begin{center}
\begin{tikzpicture}[scale=0.6] 

\draw[->] (-4,0) -- (4,0);  
\draw[->] (0,-3) -- (0,4); 

\draw[ultra thick] (-4,-2) -- (-2,0) 
-- (-2,0) node[label=below:$-\lambda$] {} -- (2,0) node[label=below:$\lambda$] {} 
-- (2,0) -- (4,2);
\node at (2,2) {$S_{\lambda}$};
\end{tikzpicture}
\hspace{0.5cm}
\begin{tikzpicture}[scale=1]

\draw[->] (-2,0) -- (2,0); 
\draw[->] (0,-1.8) -- (0,2.4);
\draw[very thick] (-2,1.5) -- (-1,0.5)
parabola bend (0,0) (1,0.5) -- (2,1.5);

\node[label=left: $\frac{\lambda}{2}$] at (0,0.5) {} ;
\node[label=below: $-\lambda$] at (-1,0) {};
\node[label=below: $\lambda$] at (1,0) {} ;
\draw[dotted] (0,0.5) -- (-0.25,0.5); 
\draw[dotted] (-1,0.5) -- (-1,-0.15); 
\draw[dotted] (1,0.5) -- (1,-0.15); 

\node[label=below: $\frac{1}{2\lambda}x^2$] at (0,0) {} ;
\end{tikzpicture}
\end{center}
\caption{Left: Soft-shrinkage function $\prox_{\lambda f} = S_\lambda$ for $f(y) = |y|$.
Right: Moreau envelope ${}^\lambda f$.
\label{fig:soft_shrinkage}}
\end{figure}
\begin{theorem}[Moreau decomposition] \label{prox_2}
For $f\in  \Gamma_0(\mathbb R^d)$ the following decomposition holds: 
\begin{align*}
\prox_f (x) + \prox_{f^*} (x) = x,\\
{}^1 \!f (x) + {}^1 \!f^* (x) = \frac12 \|x\|_2^2.
\end{align*}
\end{theorem}
For a proof we refer to \cite[Theorem 31.5]{Ro97}.

\begin{remark}[Proximal operator and resolvent] \label{rem:resolvent} 
{\rm
The subdifferential operator is a set-valued function\index{set-valued function}
$\partial f : \mathbb R^d \rightarrow 2^{\mathbb R^d}$.
For $f \in \Gamma_0(\R^d)$, we have by Fermat's rule and subdifferential calculus that
$\hat x = \prox_{\lambda \partial f} (x) $ if and only if
\begin{align*}
0 &\in \hat x - x + \lambda \partial f(\hat x),\\
x &\in (I+\lambda \partial f) (\hat x),
\end{align*}
which implies by the uniqueness of the proximum that
$\hat x = (I+\lambda \partial f)^{-1} (x)$. 
In particular,  
$J_{\lambda \partial f} := (I+\lambda \partial f)^{-1}$ is a single-valued operator which is called the
{\it resolvent} of the set-valued operator $\lambda \partial f$.
In summary, the proximal operator of $\lambda f$ coincides with the  resolvent of $\lambda \partial f$, i.e.,
$$
\prox_{\lambda f} = J_{\lambda \partial f}.
$$
}
\end{remark}
The proximal operator \eqref{def:prox} and the proximal algorithms described in Section \ref{sec:prox_algs}
can be generalized by introducing a symmetric, positive definite matrix $Q \in \R^{d,d}$ as follows:
\begin{equation} \label{precond}
\prox_{Q,\lambda f} := \argmin\limits_{y \in \mathbb R^d}  \left\{ \frac{1}{2\lambda} \|x-y\|_Q^2 + f(y) \right\},
\end{equation}
where $ \|x\|_Q^2 := x^\tT Q x$, see, e.g., \cite{CPR2013,CV2013,ZBO09}.

\subsection{Special Proximal Operators} \label{subsec:prox_operator_specials}
Algorithms involving the solution of proximal problems are only efficient
if the corresponding proximal operators can be evaluated in an efficient way. 
In the following we collect frequently appearing proximal mappings in image processing.
For  epigraphical projections see \cite{BC11,CPPP2013,HPS2013}.

\subsubsection{Orthogonal Projections}
The proximal operator generalizes the orthogonal projection operator.
The orthogonal projection of $x \in \R^d$ onto a non-empty, closed, convex set C is given by
$$
\Pi_C(x) := \argmin_{y \in C} \|x-y\|_2
$$
and can be rewritten for any $\lambda > 0$ as
$$
\Pi_C(x) = \argmin_{y \in \R^d} \left\{ \frac{1}{2\lambda} \|x-y\|_2^2 + \iota_C (y) \right\} = \prox_{\lambda \iota_C} (x).
$$
Some special sets $C$ are considered next.
\paragraph{Affine set} $C := \{y \in \R^d: Ay = b\}$ with $A \in \R^{m,d}$, $b \in \R^m$. \\ In case of $\|x-y\|_2  \rightarrow \min_y$ subject to $Ay = b$ we substitute $z:= x-y$ which leads to
$$\|z\|_2 \rightarrow \min_z \quad \mbox{subject to} \quad Az = r := Ax-b.$$
This can be directly solved, see \cite{Bj96}, and leads after back-substitution to
$$
\Pi_C (x) = x - A^\dagger (Ax-b),
$$
where $A^\dagger$ denotes the Moore-Penrose inverse of $A$.
\paragraph{Halfspace} $C := \{y \in \R^d: a^\tT y \le b\}$  with $a \in \R^{d}$, $b \in \R$.\\
A straightforward computation gives
$$
\Pi_C (x) = x - \frac{(a^\tT x - b)_+}{\|a\|_2^2} a.
$$
\paragraph{Box and Nonnegative Orthant} $C := \{y \in \R^d: l \le y \le u\}$ with $l, u \in \R^{d}$.\\
The proximal operator can be applied componentwise and gives
$$
(\Pi_C (x))_k = \left\{
\begin{array}{ll}
l_k&{\rm if} \; x_k < l_k,\\
x_k&{\rm if} \;  l_k \le x_k \le u_k,\\
u_k&{\rm if} \; x_k > u_k.
\end{array}
\right.
$$
For $l = 0$ and $u = +\infty$ we get the orthogonal projection onto the non-negative orthant
$$
\Pi_C (x) = x_+.
$$
\paragraph{Probability Simplex} $C := \{y \in \R^d: {\bf 1}^\tT y = \sum_{k=1}^d y_k = 1, \; y \ge 0\}$.\\
Here we have
$$
\Pi_C (x) = (x- \mu {\bf 1})_+,
$$
where $\mu \in \R$ has to be determined such that $h(\mu) := {\bf 1}^\tT (x- \mu {\bf 1})_+ = 1$.
Now $\mu$ can be found, e.g., by bisection
with starting interval $[\max_k x_k - 1, \max_k x_k]$ or by a method similar to those described in Subsection \ref{norms}
for projections onto the $\ell_1$-ball.
Note that $h$ is a linear spline function with knots $x_1,\ldots,x_d$ so that $\mu$ is completely determined
if we know the neighbor values $x_k$ of $\mu$.

\subsubsection{Vector Norms} \label{norms}
We consider the proximal operator of $f = \| \cdot \|_p$, $p \in [1,+\infty]$.
By the Moreau decomposition in Theorem \ref{prox_2}, regarding $(\lambda f)^* = \lambda f^*(\cdot / \lambda)$
and by \eqref{dual_norms} we obtain
\begin{align*}
\prox_{\lambda f} (x) &= x - \prox_{\lambda f^*} \left(\frac{\cdot}{\lambda} \right)\\
                      &= x - \Pi_{B_q(\lambda)} (x), 
\end{align*}
where $\frac{1}{p} + \frac{1}{q} = 1$.
Thus the proximal operator can be simply computed by the projections onto the $\ell_q$-ball.
In particular, it follows for $p = 1,2,\infty$:
\paragraph{$p = 1, \, q=\infty$:} For $k=1,\ldots,d$,
\begin{align*}
\big( \Pi_{B_\infty(\lambda)}(x) \big)_k = 
\left\{
\begin{array}{rcl}
x_k& {\rm if} & |x_k| \le \lambda,\\
 \lambda\, {\rm sgn}(x_k) & {\rm if} & |x_k| > \lambda,
\end{array}
\right.
\quad {\rm and} \quad
 \prox_{\lambda \| \cdot \|_1 } (x)  = 
 S_\lambda (x) ,
\end{align*}
where $S_\lambda(x)$, $x \in \R^d$, denotes the componentwise soft-shrinkage with threshold $\lambda$.
\paragraph{$p=q=2:$}
$$\Pi_{B_{2,\lambda}}(x) =
\left\{
\begin{array}{ccc}
x & {\rm if} &\|x\|_{2}\leq\lambda,\\
\lambda\frac{x}{\|x\|_{2}} & {\rm if} &\|x\|_{2}>\lambda,
\end{array}
\right.
\quad {\rm and} \quad
\prox_{\lambda \| \cdot \|_2 } (x) =
\left\{
\begin{array}{ccc}
0 & {\rm if} &\|x\|_{2}\leq\lambda,\\
x(1-\frac{\lambda}{\|x\|_{2}}) & {\rm if} &\|x\|_{2}>\lambda.
\end{array}
\right.
$$
\paragraph{$p=\infty, \, q = 1:$}
\begin{align*}
\Pi_{B_{1,\lambda}}(x) &=
\left\{
\begin{array}{ccl}
x & {\rm if} &\|x\|_1\leq\lambda,\\
S_\mu (x)  & {\rm if} &\|x\|_1 > \lambda,
\end{array}
\right.
\intertext{and}
\prox_{\lambda \| \cdot \|_\infty} (x) &=
\left\{
\begin{array}{ccc}
0 & {\rm if} &\|x\|_{1}\leq\lambda,\\
x -  S_\mu (x)  & {\rm if} &\|x\|_1 > \lambda,
\end{array}
\right.
\end{align*}
with
$\mu := \frac{|x_{\pi(1)}| + \ldots + |x_{\pi(m)}| - \lambda}{m}$,
where
$|x_{\pi(1)}| \ge \ldots \ge |x_{\pi(d)}| \ge 0$
are the sorted absolute values of the components of $x$
and
$m \le d$ is the largest index such that $|x_{\pi(m)}|$ is positive and
$
\frac{|x_{\pi(1)}| + \ldots + |x_{\pi(m)}| - \lambda}{m} \le |x_{\pi(m)}|,
$
see also \cite{DFL08,DSSC08}.
Another method follows similar lines as the projection onto the probability simplex in the previous subsection.
\\

Further, grouped/mixed $\ell_2$-$\ell_p$-norms are defined for 
$x = (x_1,\ldots,x_n)^\tT \in \R^{dn}$ with $x_j := (x_{jk})_{k=1}^d \in \R^d$, $j=1,\ldots,n$
by
$$
\|x\|_{2,p} := \| \left( \|x_j\|_2 \right)_{j=1}^n \|_p.
$$
For the $\ell_2$-$\ell_1$-norm we see that
$$
\prox_{\lambda \|\cdot \|_{2,1}} (x) = \argmin_{y \in \R^{dn}} \left\{ \frac{1}{2\lambda} \|x-y\|_2^2 + \|y\|_{2,1}  \right\}
$$
can be computed separately for each $j$ which results by the above considerations for the $\ell_2$-norm for each $j$ in
$$
\prox_{\lambda \|\cdot \|_{2}} (x_j) = \left\{
\begin{array}{ccc}
0 & {\rm if} &\|x_j\|_{2}\leq\lambda,\\
x_j(1-\frac{\lambda}{\|x_j\|_{2}}) & {\rm if} &\|x_j\|_{2}>\lambda.
\end{array}
\right. 
$$
The procedure for evaluating $\prox_{\lambda \|\cdot \|_{2,1}}$ is sometimes called {\it coupled or grouped shrinkage}.
\\

Finally, we provide the following rule from \cite[Prop. 3.6]{CP2007}.
\begin{lemma}	\label{prop:pesquet}
 Let $f = g + \mu |\cdot|$, where $g \in \Gamma_0(\R)$ is differentiable at $0$ with $g'(0)=0$.
 Then $\prox_{\lambda f} = \prox_{\lambda g} \circ S_{\lambda \mu}$.
\end{lemma}
\begin{example} \label{ex:elastic_net}
{\rm
Consider the {\it elastic net} regularizer\index{elastic net} 
$f(x) := \frac{1}{2}\|x\|_2^2 + \mu \|x\|_1$, see \cite{ZH05}. 
Setting the  gradient  in the proximal operator of $g := \frac12 \| \cdot \|_2 ^2$ to zero we obtain 
$$
\prox_{\lambda g} (x) = \frac{1}{1+\lambda} x.
$$
The whole proximal operator of $f$ can be then evaluated componentwise
and we see by Lemma \ref{prop:pesquet} that
$$
\prox_{\lambda f} (x) = \prox_{\lambda g}  \left( S_{\lambda \mu} (x) \right) = \frac{1}{1 + \lambda} S_{\mu \lambda} (x).
$$
}
\end{example}

\subsubsection{Matrix Norms} \label{matrix_norms}
Next we deal with proximation problems involving matrix norms.
For $X\in\mathbb{R}^{m,n}$, we are looking for
\begin{equation} \label{matrix_prox}
\prox_{\lambda \| \cdot \|} (X)= \argmin_{Y\in\mathbb{R}^{m,n}} \left \{\frac{1}{2\lambda}\|X-Y\|_{{\cal F}}^{2}+\|Y\| \right\},
\end{equation}
where $\|\cdot\|_{{\cal F}}$ is the Frobenius norm and $\|\cdot\|$ is any unitarily invariant matrix norm, i.e.,
$\|X\|=\|UXV^{T}\|$ for all unitary matrices $U\in\mathbb{R}^{m,m},V\in\mathbb{R}^{n,n}$.
Von Neumann (1937) \cite{vN37} has characterized the unitarily invariant matrix norms as those matrix norms
which can be written in the form
$$\|X\| = g(\sigma(X)),$$ 
where $\sigma(X)$ is the vector of singular values of $X$ and $g$ is a
symmetric {\it gauge function}, see \cite{Wa92}. Recall that $g:\R^d \rightarrow \R_+$
is a symmetric gauge function if it is a positively homogeneous convex function which vanishes at the origin 
and fulfills
$$
g(x) = g(\epsilon_1 x_{k_1}, \ldots, \epsilon_k x_{k_d}) 
$$
for all $\epsilon_k \in \{-1,1\}$ and all permutations $k_1,\ldots, k_d$ of indices.
An analogous result was given by Davis \cite{Da57} for symmetric matrices,
where $V^\tT$ is replaced by $U^\tT$ and the singular values by the eigenvalues.

We are interested in the {\it Schatten-$p$ norms} for $p=1,2,\infty$ which are defined  for
$X \in \mathbb R^{m,n}$ and $t:=\min\{m,n\}$
by
\begin{eqnarray*}
\|X\|_{*}&:=&\sum\limits_{i=1}^{t}\sigma_{i}(X) =g_{*}(\sigma(X))=\|\sigma(X)\|_{1},
\qquad \qquad\textnormal{(Nuclear norm)}\\
\|X\|_{\mathcal{F}}&:=&(\sum\limits_{i=1}^{m}\sum\limits_{j=1}^{n}x_{ij}^{2})^{\frac{1}{2}}=
(\sum\limits_{i=1}^{t}\sigma_{i}(X)^{2})^{\frac{1}{2}}=g_{\mathcal{F}}(\sigma(X))=\|\sigma(X)\|_{2},
\quad \textnormal{(Frobenius norm)}\\
\|X\|_{2}&:=&\max\limits_{i=1,...,t}\sigma_{i}(X)=g_{2}(\sigma(X))=\|\sigma(X)\|_{\infty},
\qquad \qquad \textnormal{(Spectral norm)}.
\end{eqnarray*}

The following theorem shows that the solution of \eqref{matrix_prox} reduces to a proximal problem for the vector norm
of the singular values of $X$. Another proof for the special case of
the nuclear norm can be found in \cite{CCS10}.
\begin{theorem} \label{matrix-shrinkage}
Let $X =U \Sigma_{X} V^\tT$ be the singular value decomposition of $X$ and $\| \cdot\|$ a unitarily invariant matrix norm.
Then  $\prox_{\lambda \| \cdot \|} (X)$  in \eqref{matrix_prox}
is given by $\hat{X}= U \Sigma_{\hat X} V^\tT$, where the singular values $\sigma(\hat X)$ in $\Sigma_{\hat X}$
are determined by
\begin{equation} \label{soln}
\sigma(\hat X) := \prox_{\lambda g} (\sigma(X)) = \argmin_{\sigma\in\mathbb{R}^{t}}
\{\frac{1}{2}\|\sigma(X) - \sigma\|_{2}^{2}+\lambda g(\sigma)\}
\end{equation}
with  the symmetric gauge function $g$  corresponding to $\|\cdot\|$.
\end{theorem}
\begin{proof} 
By Fermat's rule we know that the solution $\hat{X}$ of \eqref{matrix_prox} is determined by
\begin{equation} \label{to_fulfill}
0\in \hat{X} - X+\lambda\partial\|\hat{X}\|
\end{equation}
and from \cite{Wa92} that
\begin{equation} \label{to_fulfill_1}
\partial\|X\|=\operatorname{conv}\{UDV^\tT: X = U \Sigma_X V^\tT, \, D={\rm diag}(d),\,
d\in\partial g(\sigma(X))\}.
\end{equation}
We now construct the unique solution $\hat{X}$ of \eqref{to_fulfill}.
Let $\hat \sigma$ be the unique solution of \eqref{soln}.
By Fermat's rule $\hat \sigma$ satisfies
$
0 \in \hat \sigma - \sigma(X) + \lambda\partial g(\hat \sigma)
$
and consequently there exists $d \in \partial g(\hat \sigma)$ such that
\begin{eqnarray*}
0 = U \big( {\rm diag} (\hat \sigma) - \Sigma_{X} + \lambda {\rm diag} (d) \big)  V_{F}^\tT 
&\Leftrightarrow& 0 = U \, {\rm diag} (\hat \sigma) \, V^\tT - X + \lambda U \, {\rm diag} (d) \,  V^\tT.
\end{eqnarray*}
By \eqref{to_fulfill_1} we see that $\hat X := U \, {\rm diag} (\hat \sigma) \, V^\tT$ is a solution of \eqref{to_fulfill}.
This completes the proof. 
\end{proof}

For the special matrix norms considered above, we obtain by the previous subsection
\begin{align*}
\| \cdot\|_*: & \quad
\sigma(\hat{X}) := \sigma(X) - \Pi_{B_{\infty,\lambda}}(\sigma(X)),\\
\| \cdot\|_{\mathcal F}:& \quad
\sigma(\hat{X}) :=  \sigma(X) - \Pi_{B_{2,\lambda}}(\sigma(X)),\\
\| \cdot\|_2:& \quad
\sigma(\hat{X}) := \sigma(X) - \Pi_{B_{1,\lambda}}(\sigma(X)).
\end{align*}

\section{Fixed Point Algorithms and Averaged Operators}
%
An operator $T: \mathbb R^d \rightarrow \mathbb R^d$ is 
{\it contractive}
\index{function! contractive}\index{contractive|see{function}}  if 
it is Lipschitz continuous \index{function!Lipschitz continuous}\index{Lipschitz continuous|see{function}} 
with Lipschitz constant $L < 1$, i.e., there exists a norm $\| \cdot \|$ on $\mathbb R ^d$ such that
\[
 \| Tx-Ty \|\leq L \| x-y \| \quad \forall x,y \in \mathbb R^d.
\]
In case $L = 1$, the operator is called {\it nonexpansive}\index{function! contractive}\index{contractive|see{function}}.
A function $T: \R^d \supset \Omega \rightarrow \mathbb R^d$ 
is {\it firmly nonexpansive}\index{firmly nonexpansive} if it fulfills for all $x,y \in \R^d$ 
one of the following equivalent conditions \cite{BC11}:
\begin{align} 
\|T x -T y\|_2^2 &\le \langle x-y,T x -T y\rangle, \nonumber\\
\|Tx-Ty\|_2^2 &\le \|x-y\|_2^2-\|(I-T)x-(I-T)y\|_2^2. \label{fne_1}
\end{align}
In particular we see that a firmly nonexpansive function is nonexpansive.
%
\begin{lemma} \label{fne_prox}
For $f \in \Gamma_0(\R^d)$, the proximal operator $\prox_{\lambda f}$ is firmly nonexpansive.
In particular the orthogonal projection onto convex sets is firmly nonexpansive.
\end{lemma}
%
\begin{proof} By Theorem \ref{prox_1}ii) we have that
$$
\langle x - \prox_{\lambda f} (x), z - \prox_{\lambda f} (x) \rangle \le 0 \quad \forall z \in \mathbb R^d.
$$
With $z := \prox_{\lambda f} (y)$ this gives
$$
\langle x - \prox_{\lambda f} (x), \prox_{\lambda f} (y) - \prox_{\lambda f} (x) \rangle \le 0
$$
and similarly 
$$
\langle y - \prox_{\lambda f} (y), \prox_{\lambda f} (x) - \prox_{\lambda f} (y) \rangle \le 0.
$$
Adding these inequalities we obtain
\begin{eqnarray*}
\langle x - \prox_{\lambda f} (x) + \prox_{\lambda f} (y) - y,\prox_{\lambda f} (y) - \prox_{\lambda f} (x) \rangle &\le& 0,\\
\|\prox_{\lambda f} (y) - \prox_{\lambda f} (x)\|_2^2 &\le& \langle y-x,\prox_{\lambda f} (y)- \prox_{\lambda f} (x) \rangle. 
\end{eqnarray*}
\end{proof}
The Banach fixed point theorem \index{Banach fixed point theorem} guarantees that a contraction has a unique fixed point and that the 
{\it Picard sequence} \index{Picard sequence} 
\begin{eqnarray} \label{picard}
x^{(r+1)} = T x^{(r)}
\end{eqnarray}
converges to this fixed point for every initial element $x^{(0)}$.
However, in many applications the contraction property is too restrictive 
in the sense that we often do not have a unique fixed point. 
Indeed, it is quite natural in many cases that the reached fixed point depends on the starting value $x^{(0)}$.
Note that if $T$ is continuous and $(T^r x^{(0)})_{r \in \mathbb N}$ is convergent, 
then it converges to a fixed point of $T$.
In the following, we denote by ${\rm Fix} (T)$ the {\it set of fixed points} \index{fixed points} of $T$.
Unfortunately, we do not have convergence of $(T^r x^{(0)})_{r \in \mathbb N}$  just for nonexpansive operators
as the following example shows.
%
\begin{example} \label{example_av}
{\rm
 In $\R^2$ we consider the reflection operator 
\[
R: =\left( \begin{array}{rr} 1 & 0 \\ 0 & -1 \end{array} \right). 
\]
Obviously, $R$ is nonexpansive and we only have convergence 
of $( R^r x^{(0)} )_{r \in \mathbb N}$ 
if $x^{(0)} \in {\rm Fix}(R) = {\rm span}\{(1,0)^\tT\}$. 
A possibility to obtain a 'better' operator is to average $R$, i.e., to build
$$T:=\alpha I + (1-\alpha) R = \left( \begin{array}{rr} 1 & 0 \\ 0 & 2\alpha - 1 \end{array} \right), \quad \alpha \in (0,1).$$ 
By
\begin{equation} \label{same_fixed_points}
Tx = x \Leftrightarrow \alpha x +(1-\alpha)R(x) = x\Leftrightarrow (1-\alpha)R(x)=(1-\alpha)x,
\end{equation}
we see that $R$ and $T$ have the same fixed points. 
Moreover, since $2\alpha - 1 \in (-1,1)$,
the sequence $(T^{r}x^{(0)})_{r \in \mathbb N}$ converges to $(x^{(0)}_1,0)^\tT$ for every $x^{(0)}=(x_1^{(0)},x_2^{(0)})^\tT \in \R^2$.
}
\end{example}
\vspace{0.2cm}

An operator $T: \mathbb R^d \rightarrow \mathbb R^d$ is called 
{\it averaged}\index{function!averaged}\index{averaged|see{function}} 
if there exists a nonexpansive mapping $R$ and a constant $\alpha\in (0,1)$ such that
\[
 T=\alpha I + (1-\alpha) R.
\]
Following \eqref{same_fixed_points} we see that
$$
{\rm Fix}(R) = {\rm Fix}(T).
$$
Historically, the concept of averaged mappings can be traced back to \cite{Kr55,Ma53,Sch57}, 
where the name 'averaged' was not used yet. 
Results on averaged operators can also be found, e.g., in \cite{BC11,By04,Co04}.
\begin{lemma}[Averaged, (Firmly) Nonexpansive and Contractive Operators] \label{lemma:averaged_1} {\color{white} space} 
\begin{itemize}
\item[{\rm i)}] Every averaged operator is nonexpansive.
\item[{\rm ii)}]
A contractive operator $T:\mathbb R^d\rightarrow \mathbb R^d$ 
with Lipschitz constant $L<1$ is averaged with respect to all parameters $\alpha \in (0,(1-L)/2]$.
\item[{\rm iii)}]
An operator is firmly nonexpansive if and only if it is averaged with $\alpha = \frac12$.
\end{itemize}
\end{lemma}
\begin{proof} i) Let $T = \alpha I + (1-\alpha) R$ be averaged. Then the first assertion follows by
$$
\|T(x) - T(y)\|_2 
\le 
\alpha \|x-y\|_2 + (1-\alpha) \| R(x) - R(y) \|_2
\le 
\|x-y\|_2.
$$
ii)
We define the operator 
$R:=\frac{1}{1-\alpha}(T-\alpha I)$. 
It holds for all $x,y\in \mathbb R^d$ 
that 
\begin{eqnarray*}
 \|Rx-Ry\|_2&=& \frac{1}{1-\alpha} \| (T-\alpha I)x - (T-\alpha I) y \|_2, \\
&\leq& \frac{1}{1-\alpha} \|Tx - Ty\|_2 + \frac{\alpha}{1-\alpha} \|x-y\|_2, \\
&\leq& \frac{L + \alpha}{1-\alpha} \|x - y\|_2,
\end{eqnarray*}
so $R$ is nonexpansive if  $\alpha \leq (1-L)/2$.
\\
iii)
With
$R: = 2T-I=T-(I-T)$ we obtain the following equalities
\begin{eqnarray*}
\| Rx-Ry\|_2^2 &=& \|Tx-Ty-((I-T)x-(I-T)y)\|_2^2\\
&=&-\|x-y\|_2^2+2\|Tx-Ty\|_2^2+2\|(I-T)x-(I-T)y\|_2^2
\end{eqnarray*}
and therefore after reordering
\begin{eqnarray*}
&&\|x-y\|_2^2-\|Tx-Ty\|_2^2-\|(I-T)x-(I-T)y\|_2^2\\
&=& \|Tx-Ty\|_2^2+\|(I-T)x-(I-T)y\|_2^2-\|Rx-Ry\|_2^2\\
&=& \frac12(\|x-y\|_2^2+\|Rx-Ry\|_2^2)-\|Rx-Ry\|_2^2\\
&=& \frac12(\|x-y\|_2^2-\|Rx-Ry\|_2^2).
\end{eqnarray*}
If $R$ is nonexpansive, then the last expression is $\ge 0$ and
consequently \eqref{fne_1} holds true so that $T$ is firmly nonexpansive. Conversely, if $T$ fulfills \eqref{fne_1}, then
$$
\frac12 \big(\|x-y\|_2^2-\|Rx-Ry\|_2^2 \big)\ge 0
$$
so that $R$ is nonexpansive. This completes the proof.
\end{proof}
By the following lemma averaged operators are closed under composition.
%
\begin{lemma}[Composition of Averaged Operators] \label{better_alpha} {\color{white} space} 
\begin{itemize}
\item[{\rm i)}]
 Suppose that $T:\mathbb R^d \rightarrow \mathbb R^d$ is averaged with respect to $\alpha \in (0,1)$. 
 Then, it is also averaged with respect to any other parameter $\tilde{\alpha} \in (0,\alpha]$.
 \item[{\rm ii)}]
 Let $T_1, T_2:\mathbb R^d \rightarrow \mathbb R^d$ be averaged operators. Then, $T_2\circ T_1$ is also averaged.
 \end{itemize}
\end{lemma}
%
\begin{proof} 
i) By assumption, $T=\alpha I + (1-\alpha) R$ with $R$ nonexpansive. We have
\[
 T = \tilde{\alpha} I + \big( (\alpha-\tilde{\alpha}) I + (1-\alpha) R \big)
= \tilde{\alpha} I + (1- \tilde \alpha) 
\underbrace{\left( \frac{\alpha-\tilde{\alpha}}{1-\tilde \alpha} I + \frac{1-\alpha}{1-\tilde \alpha} R\right)}_{\tilde R}
\]
and for all $x,y \in \mathbb R^d$ it holds that
\[
 \|\tilde{R}(x)  -\tilde{R}(y)\|_2 \leq \frac{\alpha-\tilde{\alpha}}{1-\tilde \alpha}\|x-y\|_2 +
\frac{1-\alpha}{1-\tilde \alpha}\|R(x)-R(y)\|_2
\leq \|x-y\|_2.
\]
So, $\tilde{R}$ is nonexpansive. 
\\
ii) By assumption there exist nonexpansive operators 
$R_1,R_2$ and $\alpha_1,\alpha_2 \in (0,1)$ such that 
\begin{eqnarray*}
T_{2}\left(T_{1} (x) \right)
&= & \alpha_{2} T_{1} (x) + \left(1-\alpha_{2}\right)R_{2} \left(T_{1} (x) \right)\\
&= & \alpha_{2} \left(\alpha_{1}x+\left(1-\alpha_{1}\right)R_{1}\left(x\right)\right)
+
\left(1-\alpha_{2}\right)R_{2}\left(T_{1}\left(x\right)\right)\\
&= & 
\underbrace{\alpha_{2}\alpha_{1}}_{:=\alpha}x + (\alpha_{2}-\underbrace{\alpha_{2}\alpha_{1}}_{=\alpha}) R_{1}\left(x\right)
+
\left(1-\alpha_{2}\right)R_{2}\left(T_{1}\left(x\right)\right)\\
&= & \alpha x+\left(1-\alpha\right)\underbrace{\left(\frac{\alpha_{2}-\alpha}{1-\alpha}R_{1}\left(x\right)
+
\frac{1-\alpha_{2}}{1-\alpha}R_{2}\left( T_{1}\left(x\right)\right)\right)}_{=:R}
\end{eqnarray*}
The concatenation of two nonexpansive operators is nonexpansive. Finally, the convex combination of two
nonexpansive operators is nonexpansive so that $R$ is indeed nonexpansive. 
\end{proof}
An operator $T:\mathbb R^d \rightarrow \mathbb R^d$ is called {\it asymptotically regular}\index{operator! asymptotically regular}\index{asymptotically regular|see{operator}} 
if it holds for all $ x \in \mathbb R^d$ that
\[
 \left(T^{r+1} x - T^{r} x \right) \rightarrow 0 \quad \mbox{for \quad $r\rightarrow +\infty$}. 
\]
Note that this property does not imply convergence, even boundedness cannot be guaranteed. 
As an example consider the partial sums of a harmonic sequence.
%
\begin{theorem}[Asymptotic Regularity of Averaged Operators] \label{av_asym_reg} 

Let $T: \mathbb R^d \rightarrow \mathbb R^d$ be an averaged operator 
with respect to the nonexpansive mapping $R$ and the parameter $\alpha \in (0,1)$. 
Assume that ${\rm Fix}(T)\neq \emptyset$. 
Then, $T$ is asymptotically regular.
\end{theorem}
%
\begin{proof} 
Let $\hat{x}\in {\rm Fix}(T)$ and  $x^{(r)}=T^r x^{(0)}$ for some starting element $x^{(0)}$. 
Since $T$ is nonexpansive, i.e., $\|x^{(r+1)}-\hat{x}\|_2 \leq \|x^{(r)}-\hat{x}\|_2$ we obtain
\begin{equation} \label{ass_0}
\lim_{r\rightarrow \infty} \|x^{(r)}- \hat{x}\|_2 = d \geq 0. 
\end{equation}
Using ${\rm Fix} (T) = {\rm Fix} (R)$ it follows 
\begin{equation} \label{ass}
\lim_{r \rightarrow \infty} \sup\|R(x^{(r)}) - \hat x \|_2 
= 
\lim_{r \rightarrow \infty} \sup \|R(x^{(r)}) - R(\hat x ) \|_2 
\le \lim_{r \rightarrow \infty}\|x^{(r)} -\hat x \|_2 = d. 
\end{equation}
Assume that $\|x^{(r+1)} - x^{(r)}\|_2 \not \rightarrow 0$ for $r \rightarrow \infty$.
Then, there exists a subsequence $( x^{(r_l)})_{l \in \mathbb N}$ such that
$$
\|x^{(r_l+1)} - x^{(r_l)}\|_2 \ge \varepsilon
$$
for some $\varepsilon >0$.
By \eqref{ass_0} the sequence $(x^{(r_l)})_{l \in \mathbb N}$ is bounded. 
Hence there exists a convergent subsequence $(x^{(r_{l_j})})$ such that
$$
\lim_{j \rightarrow \infty} x^{(r_{l_j})} = a,
$$
where $a \in S(\hat x,d) := \{ x \in \mathbb R^d: \|x- \hat x\|_2 = d\}$ by \eqref{ass_0}.
On the other hand, we have by the continuity of $R$ and \eqref{ass} that
$$
\lim_{j \rightarrow \infty} R(x^{(r_{l_j})}) = b, \quad b \in B(\hat x,d).
$$
Since
$
\varepsilon \le \|x^{(r_{l_j}+1)} - x^{(r_{l_j})} \|_2 = \|(\alpha - 1) x^{(r_{l_j})} + (1- \alpha) R(x^{(r_{l_j})}) \|_2
$
we conclude by taking the limit $j \rightarrow \infty$ that $a \not = b$.
By the continuity of $T$ and \eqref{ass_0} we obtain 
$$
\lim_{j \rightarrow \infty} T(x^{(r_{l_j})}) = c, \quad c \in S(\hat x,d).
$$
However, by the strict convexity of $\| \cdot\|_2^2$ this yields the contradiction
\begin{eqnarray*}
\|c - \hat x \|_2^2 
&=&
\lim_{j \rightarrow \infty} \| T(x^{(r_{l_j})}) - \hat x\|_2^2 
\; = \;
\lim_{j \rightarrow \infty} \| \alpha (x^{(r_{l_j})}  - \hat x) + (1-\alpha) (R(x^{(r_{l_j})})  - \hat x)\|_2^2
\\
&=&
\| \alpha (a - \hat x) + (1- \alpha) (b - \hat x) \|_2^2
< \alpha \|a - \hat x\|_2^2+ (1- \alpha) \|b - \hat x \|_2^2 \\
&\le& d^2.
\end{eqnarray*}
\end{proof}
The following theorem was first proved for operators on Hilbert spaces by Opial \cite[Theorem 1]{Op67} based on results in \cite{BP66},
where convergence must be replaced by weak  convergence in general Hilbert spaces.
A shorter proof can be found in the appendix of \cite{DDM04}. For finite dimensional spaces the proof simplifies as follows.
%
\begin{theorem}[Opial's Convergence Theorem] \label{conv_exact} 

 Let  $T:\mathbb R^d\rightarrow \mathbb R^d$ fulfill the following conditions:
  ${\rm Fix}(T)\neq \emptyset$,
  $T$ is nonexpansive and asymptotically regular.
 Then, for every $x^{(0)} \in \mathbb R^d$, the sequence of Picard iterates $(x^{(r)})_{r \in \mathbb N}$ 
generated by 
$x^{(r+1)} = T x^{(r)}$ 
converges to an element of ${\rm Fix}(T)$.
\end{theorem}
%
\begin{proof}
Since $T$ is nonexpansive, we have
for any 
$\hat{x}\in {\rm Fix}(T)$ 
and any 
$x^{(0)}\in \mathbb R^d$ 
that
$$
\|T^{r+1} x^{(0)} -\hat{x}\|_2 \leq \|T^{r} x^{(0)}-\hat{x}\|_2.
$$ 
Hence $(T^r x^{(0)})_{r \in \mathbb N}$ is bounded and
there exists a subsequence $( T^{r_l} x^{(0)} )_{l\in \mathbb N} = (x^{(r_l)})_{l\in \mathbb N}$ 
which converges to some $\tilde x$.
If we can show that $\tilde x \in {\rm Fix}(T)$ 
we are done because in this case
$$\|T^{r} x^{(0)}-\tilde x\|_2 \leq \|T^{r_l} x^{(0)} - \tilde x\|_2, \quad r\geq r_l$$ 
and thus the whole sequence converges to $\tilde x$. 

Since $T$ is asymptotically regular it follows that 
$$(T-I)(T^{r_l}x^{(0)}) = T^{r_l+1} x^{(0)} - T^{r_l} x^{(0)} \rightarrow 0$$ 
and since $(T^{r_l}x^{(0)})_{l \in \mathbb N}$ converges to $\tilde x$ and $T$ is continuous 
we get that $(T-I)(\tilde x)=0$, i.e., $\tilde x\in {\rm Fix}(T)$. 
\end{proof}

Combining the above Theorems \ref{av_asym_reg} and \ref{conv_exact} we obtain the following main result.
%
\begin{theorem}[Convergence of Averaged Operator Iterations] \label{thm:ernte} 
 Let  $T:\mathbb R^d \rightarrow \mathbb R^d$ be an averaged operator such that ${\rm Fix}(T)\neq \emptyset$. 
 Then, for every $x^{(0)}\in \mathbb R^d$, the  sequence $(T^r x^{(0)})_{r \in \mathbb N}$ converges to a fixed point of $T$.
\end{theorem}
%
\section{Proximal Algorithms} \label{sec:prox_algs}
%
\subsection{Proximal Point Algorithm} \label{subsec:prox_point}
By Theorem \ref{prox_1} iii) the minimizer of a function $f \in \Gamma_0(\R^d)$, which we suppose to exist,
is characterized by the
fixed point equation
$$
\hat x = \prox_{\lambda f} (\hat x).
$$
The corresponding Picard iteration gives rise to the following {\it proximal point algorithm}\index{proximal point algorithm} 
which dates back to \cite{Ma70,Ro76}. Since $\prox_{\lambda f}$ is firmly nonexpansive by Lemma \ref{fne_prox} and thus averaged, the algorithm converges
by Theorem \ref{thm:ernte} for any initial value $x^{(0)} \in \R^d$ to a minimizer of $f$ if there exits one.

\begin{algorithm}[H]
\caption{Proximal Point Algorithm (PPA) \label{a:PP}}
\begin{algorithmic}
\STATE \textbf{Initialization:} $x^{(0)} \in \R^d$, $\lambda >0$
\STATE \textbf{Iterations:} For $r = 0,1,\ldots$
\STATE
$	
x^{(r+1)} = \prox_{\lambda f} (x^{(r)}) = \argmin_{x \in \R^d} \left\{ \frac{1}{2\lambda} \| x^{(r)} - x\|_2^2 + f(x) \right\}
$
\end{algorithmic}
\end{algorithm}
The PPA can be generalized for the sum 
$\sum_{i=1}^n f_i$ of functions 
$f_i \in \Gamma_0(\R^d)$, $i=1,\ldots,n$.
Popular generalizations are the so-called cyclic PPA \cite{Ber11} 
and the parallel PPA \cite{CoPe11}.

\subsection{Proximal Gradient Algorithm} \label{sec:FBalg}
We are interested in minimizing functions of the form $f = g+h$, where
$g: \R^d \rightarrow \R$ is convex, differentiable with Lipschitz
continuous gradient and Lipschitz constant $L$, i.e.,
\begin{equation} \label{split}
\|\nabla g(x) -   \nabla g(y)\|_2 \le L \|x-y\|_2 \quad \forall x,y \in \R^d,
\end{equation}
and $h \in \Gamma_0(\R^d)$.
Note that the Lipschitz condition on $\nabla g$ implies
\begin{equation} \label{lip_grad_surr}
g(x) \le g(y) + \langle \nabla g(y), x-y \rangle + \frac{L}{2} \|x-y\|_2^2 \quad \forall x,y \in \R ^d,
\end{equation}
see, e.g., \cite{OR70}.
We want to solve
\begin{equation} \label{basic_split}
\argmin_{x \in \R^d} \{g(x) + h(x) \}.
\end{equation}
By Fermat's rule and subdifferential calculus we know that $\hat x$ is a minimizer of \eqref{basic_split} if and only if
\begin{align}
0                              &\in \nabla g(\hat x) + \partial h(\hat x), \nonumber\\
\hat x - \eta \nabla g(\hat x) & \in \hat x + \eta \partial h(\hat x), \nonumber\\
\hat x &= (I + \eta \partial h)^{-1} \left(\hat x - \eta \nabla g(\hat x) \right) 
       = \prox_{\eta h} \left(\hat x - \eta \nabla g(\hat x) \right). \label{path_fbs}
\end{align}
This is a fixed point equation for the minimizer $\hat x$ of $f$. 
The corresponding Picard iteration is known
as {\it proximal gradient algorithm}\index{proximal gradient algorithm}  
or as {\it proximal forward-backward splitting}\index{proximal forward-backward splitting}.
%
\begin{algorithm}[H]
\caption{Proximal Gradient Algorithm (FBS) \label{a:FBS}}
\begin{algorithmic}
\STATE \textbf{Initialization:} $x^{(0)} \in \R^d$, $\eta \in (0,2/L)$
\STATE \textbf{Iterations:} For $r = 0,1,\ldots$
\STATE
$
x^{(r+1)} = \prox_{\eta h} \left( x^{(r)} - \eta \nabla g( x^{(r)}) \right)
$
\end{algorithmic}
\end{algorithm}
In the special case when $h := \iota_C$ is the indicator function of a non-empty, closed, convex set $C \subset \R^d$, 
the above algorithm for finding
$$
\argmin_{x \in C} g(x)
$$
becomes the gradient descent re-projection algorithm.
\begin{algorithm}[H]
\caption{Gradient Descent Re-Projection Algorithm \label{a:grad_repro}}
\begin{algorithmic}
\STATE \textbf{Initialization:} $x^{(0)} \in \R^d$, $\eta \in (0,2/L)$
\STATE \textbf{Iterations:} For $r = 0,1,\ldots$
\STATE
$
x^{(r+1)} = \Pi_C \left( x^{(r)} - \eta \nabla g( x^{(r)}) \right)
$
\end{algorithmic}
\end{algorithm}
It is also possible to use flexible variables  $\eta_r \in (0,\frac{2}{L})$ in the proximal gradient algorithm.
For further details, modifications and extensions see also \cite[Chapter 12]{FP03_2}.
The convergence of the algorithm follows by the next theorem.
%
\begin{theorem}[Convergence of Proximal Gradient Algorithm] \label{thm:FBS} 
Let $g: \R^d \rightarrow \R$ be a convex, differentiable function on $\R^d$ with Lipschitz
continuous gradient and Lipschitz constant $L$ and  $h \in \Gamma_0(\R^d)$.
Suppose that a solution of \eqref{basic_split} exists. 
Then, for every initial point $x^{(0)}$ and $\eta \in (0,\frac{2}{L})$, 
the  sequence $\{ x^{(r)} \}_r$
generated by the proximal gradient algorithm
converges to a solution of \eqref{basic_split}. 
\end{theorem}
%
\begin{proof} 
We show that $\prox_{\eta h}(I-\eta \nabla g)$ is averaged. 
Then we are done by Theorem \ref{thm:ernte}.
By Lemma \ref{fne_prox} we know that $\prox_{\eta h}$ is firmly nonexpansive.
By the Baillon-Haddad Theorem \cite[Corollary 16.1]{BC11} the function $\frac{1}{L} \nabla g$ 
is also firmly nonexpansive, i.e., it is  averaged with parameter $\frac{1}{2}$. This means that 
there exists a nonexpansive mapping $R$ such that
$\frac{1}{L} \nabla g = \frac{1}{2} (I+R)$ which implies
\[
I-\eta \nabla g = I-\tfrac{\eta L}{2}(I+R) = (1-\tfrac{\eta L}{2})I + \tfrac{\eta L}{2}(-R).
\]
Thus, for $\eta \in (0,\frac{2}{L})$, the operator $I-\eta \nabla g$ is averaged.
Since the concatenation of two averaged operators is averaged again we obtain the assertion.
\end{proof}
Under the above conditions a linear convergence rate can be achieved in the sense that
$$f(x^{(r)}) - f(\hat x) = {\cal O} \left(1/r \right),$$
see, e.g., \cite{BT09,CR97}.
%
\begin{example} {\rm
For solving  
$$\argmin_{x \in \R^d} \big\{ \underbrace{\frac12 \|Kx - b\|_2^2}_{g} + \underbrace{\lambda \|x\|_1}_{h} \big\}$$
we compute $\nabla g (x) =  K^\tT(Kx -b)$ and use that the proximal operator of the $\ell_1$-norm is just the
componentwise soft-shrinkage. Then the proximal gradient algorithm becomes
$$
x^{(r+1)} = \prox_{\lambda \eta \|\cdot\|_1} \left(x^{(r)} - \eta K^\tT(K x^{(r)} - b) \right)
= S_{\eta \lambda} \left(x^{(r)} - \eta K^\tT(K x^{(r)} - b) \right).
$$ 
This algorithm is known as {\it iterative soft-thresholding algorithm} (ISTA)\index{iterative soft-thresholding algorithm (ISTA)}
and was developed and analyzed through various techniques by many researchers. For a general Hilbert space approach, see, e.g.,
\cite{DDM04}.
}
\end{example}
%
The FBS algorithm has been recently extended to the case of non-convex functions 
in \cite{AB08,ABS2011,BST2013,CPR2013,OCBP2013}.
The convergence analysis mainly rely on the assumption that the objective function $f = g + h$
satisfies the Kurdyka-Lojasiewicz inequality which is indeed fulfilled
for a wide class of functions as $\log-\exp$, semi-algebraic and subanalytic functions which are of interest in image processing.
%
\subsection{Accelerated Algorithms}
For large scale problems as those arising in image processing a major concern is to find efficient algorithms
solving the problem in a reasonable time.
While each FBS step has low computational complexity, it may suffer from slow linear convergence \cite{CR97}.
Using a simple extrapolation idea with appropriate parameters $\tau_r$, the convergence can often be accelerated:
\begin{align}
y^{(r)} &=  x^{(r)} + \tau_r \left( x^{(r)} - x^{(r-1)} \right), \nonumber\\
x^{(r+1)} &= \prox_{\eta h} \left( y^{(r)} - \eta \nabla g( y^{(r)}) \right). \label{xx}
\end{align}
By the next Theorem \ref{thm:fFBS} we will see that $\tau_r = \frac{r-1}{r+2}$
appears to be a good choice. 
Clearly, we can vary $\eta$ in each step again.
Choosing $\theta_r$ such that
$\tau_r = \frac{\theta_r(1-\theta_{r-1})}{\theta_{r-1}}$, e.g., $\theta_r = \frac{2}{r+2}$ 
for the above choice of $\tau_r$, the algorithm can be rewritten as follows:
\begin{algorithm}[H]
\caption{Fast Proximal Gradient Algorithm \label{a:FFBS}}
\begin{algorithmic}
\STATE \textbf{Initialization:} $x^{(0)} = z^{(0)} \in \R^d$, $\eta \in (0,1/L)$, $\theta_r = \frac{2}{r+2}$
\STATE \textbf{Iterations:} For $r = 0,1,\ldots$
\STATE
$
\begin{array}{lcl}
y^{(r)} &=& (1-\theta_r) x^{(r)} + \theta_r z^{(r)}\\
x^{(r+1)} &=& \prox_{\eta h} \left( y^{(r)} - \eta \nabla g( y^{(r)}) \right)\\
z^{(r+1)} &=& x^{(r)} + \frac{1}{\theta_r} \left( x^{(r+1)}  - x^{(r)} \right)
\end{array}
$
\end{algorithmic}
\end{algorithm}
By the following standard theorem the extrapolation modification of the FBS algorithm ensures a quadratic convergence rate
see also Nemirovsky and Yudin \cite{NY83}.

%
\begin{theorem} \label{thm:fFBS} 
Let $f = g+h$, where $g: \R^d \rightarrow \R$ is a convex, Lipschitz differentiable function with Lipschitz
constant $L$ and $h \in \Gamma_0(\R^d)$. Assume that $f$ has a minimizer $\hat x$.
Then the fast proximal gradient algorithm fulfills
$$f(x^{(r)}) - f(\hat x) = {\cal O} \left(1/r^2\right).$$
\end{theorem}
%
\begin{proof}
First we consider the progress in one step of the algorithm.
By the Lipschitz differentiability of $g$ in \eqref{lip_grad_surr} and since $\eta < \frac{1}{L}$ we know that
\begin{equation} \label{important_1}
g(x^{(r+1)}) \le g(y^{(r)}) + \langle \nabla g(y^{(r)}), x^{(r+1)}-y^{(r)} \rangle + \frac{1}{2\eta} \|x^{(r+1)}-y^{(r)}\|_2^2
\end{equation}
and by the variational characterization of the proximal operator in Theorem \ref{prox_1}ii) for all $u \in \R^d$ that 
\begin{align}
h(x^{(r+1)}) 
&\le 
h(u) + \frac{1}{\eta} \langle y^{(r)} - \eta \nabla g(y^{(r)}) - x^{(r+1)}, x^{(r+1)}-u \rangle \nonumber\\
&\le 
h(u) - \langle \nabla g(y^{(r)}), x^{(r+1)}-u \rangle  +  \frac{1}{\eta} \langle y^{(r)} - x^{(r+1)}, x^{(r+1)} - u\rangle.
\label{important_2}
\end{align}
Adding the main inequalities \eqref{important_1} and \eqref{important_2} and using the convexity of $g$ yields 
\begin{align*}
f(x^{(r+1)}) 
&\le 
f(u) \underbrace{- g(u) +  g(y^{(r)}) + \langle \nabla g(y^{(r)}), u - y^{(r)} \rangle}_{\le 0} \\
&+ \frac{1}{2\eta} \|x^{(r+1)}-y^{(r)}\|_2^2
+  \frac{1}{\eta} \langle y^{(r)} - x^{(r+1)}, x^{(r+1)} - u\rangle\\
&\le 
f(u)+ \frac{1}{2\eta} \|x^{(r+1)}-y^{(r)}\|_2^2
+  \frac{1}{\eta} \langle y^{(r)} - x^{(r+1)}, x^{(r+1)} - u\rangle.
\end{align*}
Combining these inequalities for $u := \hat x$  and $u := x^{(r)}$ with $\theta_r \in [0,1]$ gives
\begin{align*}
&\theta_r \left( f(x^{(r+1)}) - f(\hat x) \right) + (1-\theta_r) \left( f(x^{(r+1)}) - f(x^{(r)}) \right)\\
&= f(x^{(r+1)}) -  f(\hat x) + (1-\theta_r)\left(f(\hat x) -  f(x^{(r)}) \right)\\
&\le 
\frac{1}{2\eta} \| x^{(r+1)}-y^{(r)} \|_2^2 + 
\frac{1}{\eta} \langle y^{(r)}-x^{(r+1)}, x^{(r+1)} - \theta_r \hat x - (1-\theta_r) x^{(r)} \rangle\\
&=
\frac{1}{2\eta} \left( 
\| y^{(r)} -  \theta_r \hat x - (1-\theta_r) x^{(r)} \|_2^2  - \| x^{(r+1)} -  \theta_r \hat x - (1-\theta_r) x^{(r)}\|_2^2
\right)\\
&=
\frac{\theta_r^2}{2\eta} \left( \|z^{(r)} - \hat x\|_2^2 - \|z^{(r+1)} - \hat x\|_2^2 \right).
\end{align*}
Thus, we obtain for a single step
$$
\frac{\eta}{\theta_r^2} \left(f(x^{(r+1)} ) - f(\hat x) \right) + \frac12 \|z^{(r+1)} - \hat x\|_2^2
\le \frac{\eta(1-\theta_r)}{\theta_r^2} \left(f(x^{(r)} - f(\hat x) \right)+ \frac12 \|z^{(r)} - \hat x\|_2^2.
$$
Using the relation recursively on the right-hand side and regarding that 
$\frac{(1-\theta_r)}{\theta_r^2} \le \frac{1}{\theta_{r-1}^2}$
we obtain
$$
\frac{\eta}{\theta_r^2} \left(f(x^{(r+1)}) - f(\hat x) \right) \le  \frac{\eta(1-\theta_0)}{\theta_0^2} \left(f(x^{(0)}) - f(\hat x) \right)
+ \frac12 \|z^{(0)} - \hat x\|_2^2
=\frac12 \|x^{(0)} - \hat x\|_2^2.
$$
This yields the assertion
$$
f(x^{(r+1)} )- f(\hat x) \le \frac{2}{\eta (r+2)^2}\|x^{(0)} - \hat x\|_2^2.
$$
\end{proof}
There exist many variants and generalizations of the above algorithm as 
\begin{itemize}
\item[-]
{\it Nesterov's algorithms}\index{Nesterov algorithm} \cite{Ne04,Ne83}, see also  \cite{DHJJ09,Ts08};
this includes approximation algorithms for nonsmooth $g$ \cite{BBC2011,Ne05} as NESTA\index{NESTA},
\item[-] {\it fast iterative shrinkage algorithms} (FISTA)\index{FISTA} by Beck and Teboulle \cite{BT09},
\item[-] {\it variable metric strategies}\index{variable metric strategies} \cite{BGLS95,BQ99,CV2013,PLS08}, where based on \eqref{precond} step \eqref{xx} is replaced by
\begin{equation}
\label{eq:VarMetricStrategy}
x^{(r+1)} = \prox_{Q_r,\eta_r h} \left( y^{(r)} - \eta_r Q_r^{-1}\nabla g( y^{(r)}) \right)
\end{equation}
with symmetric, positive definite matrices $Q_r$.
\end{itemize}
Line search strategies can be incorporated \cite{GS2011,Gu92,Ne07}.
Finally we mention Barzilei-Borwein step size rules \cite{BB88} based on a Quasi-Newton approach and relatives, see \cite{FZZ08} for an overview
and the cyclic proximal gradient algorithm related to the cyclic Richardson algorithm \cite{SSM2013}.
%
\section{Primal-Dual Methods} \label{sec:p_d_methods}
%
\subsection{Basic Relations} \label{subsec:bas_pd}
The following minimization algorithms closely rely on the primal-dual formulation of problems.
We consider functions $f = g  + h(A\,\cdot)$, where $g \in \Gamma_0(\R^d)$, $h \in \Gamma_0(\R^m)$, 
and $A \in \R^{m,d}$, and ask for the solution of the primal problem
\begin{equation} \label{problem_pdhg}
(P) \quad \argmin_{x \in \mathbb R^d} \left\{ g(x) + h(Ax) \right\},
\end{equation}
that can be rewritten as
\begin{equation} \label{problem_pdhg_1}
(P) \quad \argmin_{x \in \mathbb R^d, y \in \mathbb R^m} \left\{ g(x) + h(y)  \quad {\rm s.t.} \quad A x = y \right\}.
\end{equation}
The {\it Lagrangian}\index{Lagrangian} of \eqref{problem_pdhg_1} is given by
\begin{equation} \label{Lagrangian}
L(x,y,p) := g(x) + h(y) + \langle p, Ax-y \rangle
\end{equation}
and the {\it augmented Lagrangian}\index{augmented Lagrangian} by
\begin{align}
L_\gamma(x,y,p) &:= g(x) + h(y) + \langle p, Ax - y \rangle + \frac{\gamma}{2} \|Ax - y\|_2^2, \quad \gamma > 0, \nonumber \\
&= g(x) + h(y) +  \frac{\gamma}{2} \|Ax - y + \frac{p}{\gamma}\|_2^2 - \frac{1}{2\gamma} \|p\|_2^2. \label{Lagrangian_augmented}
\end{align}
Based on the Lagrangian \eqref{Lagrangian}, the primal and dual problem can be written as
\begin{align}
(P) &\quad \argmin_{x\in \mathbb R^d ,y\in \mathbb R^m} \sup_{p \in \mathbb R^m} \left\{g(x) + h(y) + \langle p, Ax-y \rangle \right\},
\label{L_primal}\\
(D) &\quad \argmax_{p\in \mathbb R^m} \inf_{x\mathbb \in R^d,y \in \mathbb R^m} \left\{g(x) + h(y) + \langle p, Ax-y \rangle \right\}.
\label{L_dual}
\end{align}
Since 
$$\min_{y\in \mathbb R^m}  \{h(y) - \langle p,y \rangle\} = -\max_{y \in \mathbb R^m}  \{ \langle p,y\rangle -  h(y)\} = -h^*(p)$$
and in \eqref{problem_pdhg} further
$$
h(Ax) = \max_{p \in \mathbb R^m} \{\langle p,Ax \rangle - h^*(p) \},
$$
the primal and dual problem can be rewritten as
\begin{align*}
(P) &\quad \argmin_{x \in \mathbb R^d} \sup_{p \in \mathbb R^m} \left\{g(x) - h^*(p)   + \langle p, Ax \rangle \right\},\\
(D) &\quad \argmax_{p \in \mathbb R^m} \inf_{x \in \mathbb R^d} \left\{g(x)  - h^*(p) + \langle p, Ax \rangle \right\}.
\end{align*}
If the infimum exists, the dual problem can be seen as Fenchel dual problem
\begin{equation} \label{fenchel_D}
 (D) \quad \argmin_{p \in \mathbb R^m}  \left\{g^*(- A^\tT p)  + h^*(p)  \right\}.
\end{equation}

Recall that $((\hat x, \hat y), \hat p) \in \R^{dm,m}$ is a {\it saddle point}\index{saddle point} of the Lagrangian $L$ in \eqref{Lagrangian} if
$$
L((x,y),\hat p) \le L((\hat x, \hat y), \hat p) \le L((\hat x, \hat y), p) \qquad \forall (x,y) \in \R^{dm}, \, p \in \R^{m}.
$$
If $((\hat x, \hat y), \hat p) \in \R^{dm,m}$ is a saddle point of $L$, then $(\hat x, \hat y)$ is a solution of the primal
problem \eqref{L_primal} and $\hat p$ is a solution of the dual problem \eqref{L_dual}.
The converse is also true. However the existence of a solution of the primal problem $(\hat x, \hat y) \in \R^{dm}$
does only imply under additional qualification constraint that there exists $\hat p$ such that
$((\hat x, \hat y), \hat p) \in \R^{dm,m}$ is a saddle point of $L$.
%
\subsection{Alternating Direction Method of Multipliers} \label{subsec:admm}
%
Based on the Lagrangian formulation \eqref{L_primal} and \eqref{L_dual}, a first idea to solve the optimization problem
would be to alternate the minimization of the Lagrangian with respect to $(x,y)$ and to apply a gradient ascent
approach with respect respect to $p$. This is known as general Uzawa method \cite{AHU58}.
More precisely, noting that
for differentiable $\nu (p) := \inf_{x,y} L(x,y,p) = L(\tilde x, \tilde y, p)$ we have
$\nabla \nu (p) = A\tilde x - \tilde y$, the algorithm reads
\begin{align}
(x^{(r+1)}, y^{(r+1)}) &\in \argmin_{x \in \R^d,y \in \R^m} L(x,y,p^{(r)}), \label{sep}\\
p^{(r+1)} &= p^{(r)} + \gamma (A x^{(r+1)} - y^{(r+1)}), \qquad \gamma > 0. \nonumber
\end{align} 
Linear convergence can be proved under certain conditions (strict convexity of $f$) \cite{GL89}.
The assumptions on $f$ to ensure convergence of the algorithm can be relaxed by replacing
the Lagrangian by the augmented Lagrangian $L_\gamma$ \eqref{Lagrangian_augmented} with fixed parameter $\gamma$:
\begin{align}
(x^{(r+1)}, y^{(r+1)}) &\in \argmin_{x \in \R^d,y \in \R^m} L_\gamma(x,y,p^{(r)}), \label{ALM}\\
p^{(r+1)} &= p^{(r)} + \gamma (A x^{(r+1)} - y^{(r+1)}), \qquad \gamma > 0. \nonumber
\end{align}
This augmented Lagrangian  method is known as {\it method of multipliers}\index{method of multipliers} \cite{He69,Po72,Ro76}.
It can be shown \cite[Theorem 3.4.7]{BI00}, \cite{Be82} that the sequence $(p^{(r)})_r$ generated by the algorithm coincides with
the proximal point algorithm applied to $-\nu(p)$, i.e.,
$$
p^{(r+1)} = \prox_{-\gamma \nu} \left( p^{(r)} \right).
$$
The improved convergence properties came at a cost. While the minimization with respect to $x$ and $y$
can be separately computed in \eqref{sep} using 
$
\langle p^{(r)}, (A | -I) \begin{pmatrix}x\\y \end{pmatrix} \rangle 
= 
\langle \begin{pmatrix}A^\tT\\-I \end{pmatrix} p^{(r)}, \begin{pmatrix}x\\y \end{pmatrix} \rangle 
$,
this is no longer possible for the augmented Lagrangian.
A remedy is to alternate the minimization with respect to $x$ and $y$ which leads to
\begin{align}
x^{(r+1)} &\in \argmin_{x \in \R^d} L_\gamma(x,y^{(r)},p^{(r)}), \label{first_iter}\\
y^{(r+1)} &= \argmin_{y \in \R^m} L_\gamma(x^{(r+1)},y,p^{(r)}), \label{second_iter}\\
p^{(r+1)} &= p^{(r)} + \gamma (A x^{(r+1)} - y^{(r+1)}). \nonumber
\end{align}
This is the {\it alternating direction method of multipliers (ADMM)}\index{alternating direction method of multipliers}
which dates back to \cite{Ga83,GM76,GM75}.
\begin{algorithm}[H]
\caption{Alternating Direction Method of Multipliers (ADMM) \label{alg:ADMM}}
\begin{algorithmic}
\STATE \textbf{Initialization:} $y^{(0)}\in \mathbb R^m$, $p^{(0)} \in \mathbb R^m$
\STATE \textbf{Iterations:} For $r = 0,1,\ldots$
\STATE
$
\begin{array}{lcl}
x^{(r+1)}&\in& \argmin_{x \in \mathbb R^d} \left\{ g(x) +  \frac{\gamma}{2} \|\frac{1}{\gamma} p^{(r)} + Ax - y^{(r)}\|_2^2  \right\}\\
y^{(r+1)}&=& \argmin_{y \in \mathbb R^m} \left\{ h(y) + \frac{\gamma}{2} \|\frac{1}{\gamma} p^{(r)} + Ax^{(r+1)} - y\|_2^2\right\} = \prox_{\frac{1}{\gamma}h}(\frac{1}{\gamma} p^{(r)} + Ax^{(r+1)})\\
p^{(r+1)}&=& p^{(r)} + \gamma(A x^{(r+1)} -  y^{(r+1)} ) 
\end{array}
$
\end{algorithmic}
\end{algorithm}
Setting
$b^{(r)} := p^{(r)}/\gamma$
we obtain the following (scaled) ADMM:
\begin{algorithm}[H]
\caption{Alternating Direction Method of Multipliers (scaled ADMM) \label{alg:ADMMscaled}}
\begin{algorithmic}
\STATE \textbf{Initialization:} $y^{(0)}\in \mathbb R^m$, $b^{(0)} \in \mathbb R^m$
\STATE \textbf{Iterations:} For $r = 0,1,\ldots$
\STATE
$
\begin{array}{lcl}
x^{(r+1)}&\in& \argmin_{x \in \mathbb R^d} \left\{ g(x) +  \frac{\gamma}{2} \|b^{(r)} + Ax - y^{(r)}\|_2^2  \right\} \\
y^{(r+1)}&=& \argmin_{y \in \mathbb R^m} \left\{ h(y) + \frac{\gamma}{2} \|b^{(r)} + Ax^{(r+1)} - y\|_2^2\right\} = \prox_{\frac{1}{\gamma} h}(b^{(r)} + Ax^{(r+1)})\\
b^{(r+1)}&=& b^{(r)} + A x^{(r+1)} -  y^{(r+1)}
\end{array}
$
\end{algorithmic}
\end{algorithm}
A good overview on the ADMM algorithm and its applications is given in \cite{BPCPE11}, where in particular the
important issue of choosing the parameter $\gamma >0$ is addressed.
The ADMM can be considered for more general problems
\begin{equation} \label{general_admm_problem}
\argmin_{x \in \mathbb R^d, y \in \mathbb R^m} \left\{ g(x) + h(y)  \quad {\rm s.t.} \quad A x + By = c\right\}.
\end{equation}
Convergence of the ADMM under various assumptions was proved, e.g., in \cite{GM76,HY98,LM79,Ts91}.
We will see that for our problem \eqref{problem_pdhg_1} the convergence follows by the relation of the ADMM to the so-called
Douglas-Rachford splitting algorithm which convergence can be shown using averaged operators.
Few bounds on the global convergence rate of the algorithm can be found in \cite{EB90} (linear convergence for linear programs depending on a variety of quantities),
\cite{HL2012} (linear convergence for sufficiently small step size) and
on the local behaviour of a specific variation of the ADMM during the course of iteration for quadratic programs in \cite{Bol2014}.
\begin{theorem} [Convergence of ADMM]\label{thm:convADMM}
Let $g \in \Gamma_0(\R^d)$, $h \in \Gamma_0(\R^m)$ 
and $A \in \R^{m,d}$. Assume that the Lagrangian \eqref{Lagrangian} has a saddle point.
Then,  for $r\rightarrow \infty$, the sequence $\gamma \left( b^{(r)} \right)_r$ converges 
to a solution of the dual problem.
If in addition the first step \eqref{first_iter} in the ADMM algorithm has a unique solution, then
$ \left( x^{(r)} \right)_r$ converges to a solution of the primal problem.
\end{theorem}

There exist different modifications of the ADMM algorithm presented above:
\begin{itemize}
	\item[-] {\it inexact computation} of the first step
\eqref{first_iter} \cite{CT94,EB92} such that it might be handled by an iterative method,
	\item[-] {\it variable parameter and metric strategies} \cite{BPCPE11,He2002,HY98,He2000,Kontogiorgis1998} where the fixed parameter $\gamma$ 
can vary in each step, or the quadratic term $(\gamma / 2) \|Ax - y\|_2^2$ within the augmented Lagrangian \eqref{Lagrangian_augmented} is replaced 
by the more general proximal operator based on \eqref{precond} such that the ADMM updates \eqref{first_iter} and \eqref{second_iter} receive the form
	\begin{eqnarray*}
		x^{(r+1)}&\in& \argmin_{x \in \mathbb R^d} \left\{ g(x) +  \frac{1}{2} \|b^{(r)} + Ax - y^{(r)}\|_{Q_r}^2  \right\}, \\
y^{(r+1)}&=& \argmin_{y \in \mathbb R^m} \left\{ h(y) + \frac{1}{2} \|b^{(r)} + Ax^{(r+1)} - y\|_{Q_r}^2\right\},
	\end{eqnarray*}
respectively, with symmetric, positive definite matrices $Q_r$. The variable parameter strategies might mitigate the performance dependency on the 
initial chosen fixed parameter \cite{BPCPE11,He2000,Kontogiorgis1998,Wang2001} and include monotone conditions \cite{HY98,Kontogiorgis1998} or more flexible non-monotone rules \cite{BPCPE11,He2002,He2000}.
\end{itemize}
\paragraph{ADMM from the Perspective of Variational Inequalities} 
The ADMM algorithm presented above from the perspective of Lagrangian functions has been also studied extensively 
in the area of {\it variational inequalities (VIs)}\index{Variational Inequality}, see, e.g., \cite{Ga83,He2002,Ts91}. 
A VI problem consists of finding for a mapping $F: \R^l \rightarrow \R^l$ a vector $\hat{z} \in \R^l$ such that
\begin{equation} \label{VI}
\left\langle z - \hat{z}, F(\hat{z})\right\rangle \geq 0, \qquad \forall z \in \R^l .
\end{equation}
In the following, we consider the minimization problem \eqref{problem_pdhg_1}, i.e.,
\begin{equation*}
\argmin_{x \in \mathbb R^d, y \in \mathbb R^m} \left\{ g(x) + h(y)  \quad {\rm s.t.} \quad A x = y \right\}, 
\end{equation*}
for $g \in \Gamma_0(\R^d)$, $h \in \Gamma_0(\R^m)$.
The discussion can be extended to the more general problem \eqref{general_admm_problem} \cite{He2002,Ts91}. 
Considering the Lagrangian \eqref{Lagrangian} and its optimality conditions, solving \eqref{problem_pdhg_1} 
is equivalent to find a triple $\hat{z} = ((\hat{x},\hat{y}),\hat{p}) \in \R^{dm,m}$ such that \eqref{VI} holds with
\[
z = \begin{pmatrix} x\\y\\p \end{pmatrix}, \qquad F(z) = \begin{pmatrix}\partial g(x) + A^\tT p \\ \partial h(y) - p \\ Ax - y \end{pmatrix} ,
\]
where $\partial g$ and $\partial h$ have to be understood as any element of the corresponding subdifferential for simplicity. 
Note that $\partial g$ and $\partial h$ 
are maximal monotone operators \cite{BC11}.
A VI problem of this form can be solved by ADMM as proposed by Gabay \cite{Ga83} and Gabay and Mercier \cite{GM76}: for a given triple $(x^{(r)},y^{(r)},p^{(r)})$ 
generate new iterates $(x^{(r+1)},y^{(r+1)},p^{(r+1)})$ by
\begin{itemize}
	\item[i)] find $x^{(r+1)}$ such that
		\begin{equation} \label{VI:x}
		\langle x - x^{(r+1)}, \partial g(x^{(r+1)}) + A^\tT (p^{(r)} + \gamma(Ax^{(r+1)} - y^{(r)})) \rangle \geq 0, \quad \forall x \in \R^d,
		\end{equation}
	\item[ii)] find $y^{(r+1)}$ such that
		\begin{equation} \label{VI:y}
		\langle y - y^{(r+1)}, \partial h(y^{(r+1)}) - (p^{(r)} + \gamma(Ax^{(r+1)} - y^{(r+1)})) \rangle \geq 0, \quad \forall y \in \R^m,
		\end{equation}
		\item[iii)] update $p^{(r+1)}$ via
		\[ p^{(r+1)} = p^{(r)} + \gamma(Ax^{(r+1)} - y^{(r+1)}), \]
\end{itemize}
where $\gamma > 0$ is a fixed penalty parameter. To corroborate the equivalence of the iteration scheme 
above to ADMM in Algorithm \ref{alg:ADMM}, note that \eqref{VI} reduces to $\left\langle \hat{z}, F(\hat{z})\right\rangle \geq 0$ for $z = 2 \hat{z}$. 
On the other hand, \eqref{VI} is equal to $\left\langle \hat{z}, F(\hat{z})\right\rangle \leq 0$ when $z = - \hat{z}$. The both cases transform \eqref{VI} 
to find a solution $\hat{z}$ of a system of equations $F(\hat{z}) = 0$. Thus, the VI sub-problems \eqref{VI:x} and \eqref{VI:y} can be reduced to find a pair $(x^{(r+1)},y^{(r+1)})$ with
\begin{eqnarray*}
	\partial g(x^{(r+1)}) + A^\tT (p^{(r)} + \gamma(Ax^{(r+1)} - y^{(r)})) & = & 0, \\
	\partial h(y^{(r+1)}) - (p^{(r)} + \gamma(Ax^{(r+1)} - y^{(r+1)})) & = & 0 .
\end{eqnarray*}
The both equations correspond to optimality conditions of the minimization sub-problems \eqref{first_iter} 
and \eqref{second_iter} of the ADMM algorithm, respectively. The theoretical properties of ADMM from the perspective of VI problems were studied extensively 
and a good reference overview can be found in \cite{He2002}.
%
%
\paragraph{Relation to Douglas-Rachford Splitting} 
Finally we want to point out the relation of the ADMM to the Douglas-Rachford splitting algorithm
applied to the dual problem, see \cite{EB92,Es09,Ga83,Se11}. We consider again the problem \eqref{basic_split}, i.e.,
$$
\argmin_{x \in \R^d} \{g(x) + h(x) \},
$$
where we assume this time only $g,h \in \Gamma_0(\R^d)$ and that $f$ or $g$ is continuous at a point in $\dom g \cap \dom h$.
Fermat's rule and subdifferential calculus implies that $\hat x$ is a minimizer if and only if
\begin{equation} 
0\in \partial g (\hat{x}) + \partial h (\hat{x}) 
\quad \Leftrightarrow \quad 
\exists \hat \xi \in \eta \partial g (\hat x) \; \mbox{such that} \; \hat x =  \prox_{\eta h} (\hat x - \hat \xi) \label{derivation_drs3}.
\end{equation}
The basic idea for finding  such minimizer is to set up 
a 'nice' operator  $T:\mathbb R^d \rightarrow \mathbb R^d$ by
\begin{equation} \label{op_drs}
T := \prox_{\eta h}(2 \prox_{\eta g} - I) - \prox_{\eta g} + I,
\end{equation}
which fixed points $\hat t$ are related to the minimizers as follows:
setting $\hat x := \prox_{\eta g}(\hat t)$, i.e.,
$\hat t \in \hat x + \eta \partial g (\hat x)$ and 
$\hat \xi := \hat t - \hat x \in \eta \partial g(\hat x)$
we see that
\begin{eqnarray*}
\hat t &=& T(\hat t) \; = \; \prox_{\eta h}(2 \hat x - \hat t)- \hat x + \hat t,\\
\hat \xi + \hat x &=& \prox_{\eta h} (\hat x - \hat \xi) + \hat \xi,\\
\hat x &=& \prox_{\eta h} (\hat x - \hat \xi),
\end{eqnarray*}
which coincides with \eqref{derivation_drs3}.
By the proof of the next theorem, the operator $T$ is firmly nonexpansive such that by Theorem \ref{thm:ernte}  
a fixed point of $T$ can be found by Picard iterations.
This gives rise to the 
following {\it Douglas-Rachford splitting algorithm (DRS)}.\index{algorithm!Douglas-Rachford splitting}\index{Douglas-Rachford splitting algorithm|see{algorithm}}  
%
\begin{algorithm}[H]
\caption{Douglas-Rachford Splitting Algorithm (DRS)\label{a:DRS}}
\begin{algorithmic}
\STATE \textbf{Initialization:} $x^{(0)}, t^{(0)} \in \R^d$, $\eta > 0$
\STATE \textbf{Iterations:} For $r = 0,1,\ldots$
\STATE
$
\begin{array}{lcl}
t^{(r+1)}&=& \prox_{\eta h}(2 x^{(r)}-t^{(r)}) + t^{(r)} - x^{(r)}, \nonumber \\
 x^{(r+1)}&=& \prox_{\eta g}(t^{(r+1)}).
\end{array}
$
\end{algorithmic}
\end{algorithm}
%
The following theorem verifies the convergence of the DRS algorithm. For a recent convergence result see also \cite{HY2012}.
%
\begin{theorem}[Convergence of Douglas-Rachford Splitting Algorithm] \label{DRS} 
Let $g,h\in \Gamma_0(\R^d)$ where one of the functions is continuous at a point in $\dom g \cap \dom h$. 
Assume that a solution of $\argmin_{x \in \R^d} \{g(x) + h(x) \}$ exists.
Then, for any initial  $t^{(0)}, x^{(0)} \in \R^d$
and any $\eta > 0$, the DRS sequence $(t^{(r)})_{r}$ converges to a fixed point $\hat{t}$ of $T$ in \eqref{op_drs}
and
$(x^{(r)})_{r}$  to a solution of the minimization problem.
\end{theorem}
\begin{proof} It remains to show that $T$ is firmly nonexpansive.
We have for $R_1 := 2 \prox_{\eta g} -I$ and $R_2 := 2 \prox_{\eta h} -I$ that
\begin{eqnarray*}
 2 T &=& 2 \prox_{\eta h} (2 \prox_{\eta g}-I)  - (2 \prox_{\eta g} -I) + I 
\; = \; 
R_{2} \circ R_{1} + I,\\
T &=& \tfrac{1}{2} I + \tfrac{1}{2}R_{2} \circ R_{1}.
\end{eqnarray*}
The operators $R_{i}$, $i=1,2$ are nonexpansive since $\prox_{\eta g}$ and $\prox_{\eta h}$ are firmly nonexpansive. 
Hence $R_{2} \circ R_{1}$ is nonexpansive and we are done.
\end{proof}
%
The relation of the ADMM algorithm and DRS algorithm applied to the Fenchel dual problem \eqref{fenchel_D}, i.e.,
\begin{equation} \label{drs_dual}
\begin{array}{lcl}
 t^{(r+1)}&=& \prox_{\eta g^*\circ (-A^\tT)}(2 p^{(r)}-t^{(r)}) + t^{(r)} - p^{(r)},  \\
 p^{(r+1)}&=& \prox_{\eta h^*}(t^{(r+1)}),
\end{array}
\end{equation}
is given by the following theorem, see \cite{EB92,Ga83}.
\begin{theorem}[Relation between ADMM and DRS] \label{DRS_equiv} 
The ADMM sequences $\left( b^{(r)} \right)_r$ and $\left( y^{(r)} \right)_r$ 
are related with the sequences \eqref{drs_dual} generated by the DRS algorithm applied to the dual problem
by $\eta = \gamma$ and
\begin{equation}\label{identi}
\begin{array}{lcl}
t^{(r)} &=& \eta (b^{(r)}+y^{(r)}), \\
p^{(r)} &=& \eta b^{(r)}.
\end{array} 
\end{equation}
\end{theorem}
\begin{proof}
First, we show that 
\begin{equation} \label{fundament_1}
\hat{p} =\argmin_{p \in \R^m} \left\{ \frac{\eta}{2} \|A p-q\|_2^2 + g(p) \right\}
\quad \Rightarrow \quad
\eta ( A \hat{p}-q ) =  \prox_{\eta g^*\circ (-A^\tT)}(- \eta q)
\end{equation}
holds true. The left-hand side of \eqref{fundament_1} is equivalent to
$$
0 \in \eta A^\tT (A\hat{p}-q)+\partial g(\hat{p})
\quad \Leftrightarrow \quad
 \hat{p} \in \partial g^*\big( - \eta A^\tT(A\hat{p}-q) \big).
$$
Applying $-\eta A$ on both sides and using the chain rule implies
\begin{eqnarray*}
 -\eta A \hat{p} \in - \eta A \partial g^* \big( -\eta A^\tT (A \hat{p}-q) \big)
\; = \;
\eta \, \partial \big( g^*\circ (-A^\tT) \big) \big( \eta(A\hat{p}-q) \big).
\end{eqnarray*}
Adding $-\eta q$ we get
\begin{eqnarray*}
 -\eta q \in \big( I + \eta \, \partial (g^*\circ (-A^\tT)) \big) \big( \eta(A \hat{p}-q) \big),
\end{eqnarray*}
which is equivalent to the right-hand side of \eqref{fundament_1} by the definition of the resolvent
 (see Remark \ref{rem:resolvent}).

Secondly, applying \eqref{fundament_1} to  the first ADMM step
with $\gamma = \eta$ and $q := y^{(r)} - b^{(r)}$ yields
\begin{eqnarray*}
 \eta (b^{(r)}+A x^{(r+1)}-y^{(r)}) = \prox_{\eta g^*\circ (-A^\tT) }(\eta (b^{(r)}-y^{(r)})).
\end{eqnarray*}
Assume that the ADMM and DRS iterates have
the identification \eqref{identi} up to some $r \in \mathbb N$.
Using this induction hypothesis it follows that
\begin{eqnarray}
 \eta (b^{(r)}+A x^{(r+1)}) = \prox_{\eta g^*\circ (-A^\tT)}(\underbrace{\eta (b^{(r)}-y^{(r)})}_{2p^{(r)}-t^{(r)}})
+\underbrace{\eta y^{(r)}}_{t^{(r)}-p^{(r)}} \stackrel{\eqref{drs_dual}}{=} t^{(r+1)} \label{pf_DRS_3}.
\end{eqnarray}
By definition of $b^{(r+1)}$  we see that $\eta (b^{(r+1)}+y^{(r+1)})=t^{(r+1)}$.
Next we apply \eqref{fundament_1} in
the second ADMM step where we replace $g$ by $h$ and $A$ by $-I$ and use
 $q := -b^{(r)} - A x^{(r+1)}$. 
Together with \eqref{pf_DRS_3} this gives
\begin{eqnarray} 
\eta(-y^{(r+1)} + b^{(r)}+A x^{(r+1)}) = \prox_{\eta h^*}(\underbrace{\eta (b^{(r)}+A x^{(r+1)})}_{t^{(r+1)}})\stackrel{\eqref{drs_dual}}{=}p^{(r+1)}. \label{asd}
\end{eqnarray}
Using again the definition of $b^{(r+1)}$ we obtain $\eta b^{(r+1)} =p^{(r+1)}$ which completes the proof.
\end{proof}
%
\subsection{Primal Dual Hybrid Gradient Algorithms} \label{subsec:pdhg}
%
The first ADMM step \eqref{first_iter} requires in general the solution of a linear system of equations.
This can be avoided by modifying this step using the  Taylor expansion 
at $x^{(r)}$:
$$
\frac{\gamma}{2} \| \frac{1}{\gamma} p^{(r)} + Ax - y^{(r)} \|_2^2 \approx {\rm const} + 
\gamma \langle A^\tT(A x^{(r)} - y^{(r)} + \frac{1}{\gamma} p^{(r)}), x \rangle
+ \frac{\gamma}{2} (x-x^{(r)})^\tT A^\tT A (x- x^{(r)}),
$$
approximating $A^\tT A \approx \frac{1}{\gamma \tau} I$,
setting $\gamma := \sigma$ and using $p^{(r)}/\sigma$ 
instead of $A x^{(r)} - y^{(r)} + p^{(r)}/\sigma$ we obtain
(note that $p^{(r+1)}/\sigma=  p^{(r)}/\sigma + A x^{(r+1)} -  y^{(r+1)}$):
\begin{equation} \label{pdhg}
\begin{array}{lclcl} 
x^{(r+1)}&=& \argmin_{x \in \mathbb R^d} \left\{ g(x) + \frac{1}{2 \tau} \|x - \left( x^{(r)} - \tau A^\tT {p}^{(r)} \right)\|_2^2  \right\}
&=& \prox_{\tau g} \left( x^{(r)} -  \tau A^\tT {p}^{(r)} \right),\\[0.5ex]
y^{(r+1)}&=& \argmin_{y \in \mathbb R^n} \left\{ h(y) + \frac{\sigma}{2} \|\frac1\sigma p^{(r)} + Ax^{(r+1)} - y\|_2^2\right\}
&=&
\prox_{\frac1\sigma h} \left(\frac1\sigma p^{(r)} + Ax^{(r+1)} \right)
,\\[0.5ex]
p^{(r+1)}&=& p^{(r)} + \sigma(A x^{(r+1)} -  y^{(r+1)}). 
\end{array}
\end{equation}
The above algorithm can be deduced in another way by the Arrow-Hurwicz method: 
we alternate the minimization in the primal and dual problems \eqref{L_primal} and \eqref{L_dual} and add quadratic terms. 
The resulting sequences
\begin{align}
 x^{(r+1)}&= \argmin_{x \in \mathbb R^d} \left\{ g(x) +  \langle p^{(r)}, Ax \rangle  + \frac{1}{2\tau} \|x-x^{(r)}\|_2^2\right\}, \nonumber\\
&= \prox_{\tau g} (x^{(r)} - \tau A^\tT p^{(r)})
\label{xxnew1} \\
 p^{(r+1)}&= \argmin_{p \in \mathbb R^m} \left\{ h^*(p) - \langle p, Ax^{(r+1)} \rangle + \frac{1}{2\sigma} \|p-p^{(r)}\|_2^2\right\},\nonumber \\
&= \prox_{\sigma h^*} (p^{(r)} + \sigma A x^{(r+1)})
\label{xxnew2}
 \end{align}
coincide with those in \eqref{pdhg} which can be seen as follows: For $x^{(r)}$ the relation is straightforward.
From the last equation we obtain 
\begin{eqnarray*}
   p^{(r)} + \sigma A x^{(r+1)} &\in& p^{(r+1)} + \sigma \partial h^* (p^{(r+1)}),\\
  \frac{1}{\sigma} (p^{(r)} - p^{(r+1)}) + Ax^{(r+1)} &\in& \partial h^* (p^{(r+1)}),  
\end{eqnarray*}
and using that $p \in \partial h(x) \Leftrightarrow x \in \partial h^*(p)$ further
$$
p^{(r+1)} \in \partial h \big( \underbrace{\frac{1}{\sigma} (p^{(r)} - p^{(r+1)}) + Ax^{(r+1)}}_{y^{(r+1)}} \big).
$$
Setting
$$
y^{(r+1)} := \frac{1}{\sigma} (p^{(r)} - p^{(r+1)}) + Ax^{(r+1)},
$$
we get
\begin{equation} \label{pdhg_help_2}
p^{(r+1)} = p^{(r)} + \sigma (Ax^{(r+1)} - y^{(r+1)} )
\end{equation} 
and $p^{(r+1)} \in \partial h (y^{(r+1)})$ which can be rewritten as
\begin{eqnarray*}
y^{(r+1)} + \frac{1}{\sigma} p^{(r+1)} &\in& y^{(r+1)} + \frac{1}{\sigma} \partial h( y^{(r+1)}),\\
\frac{1}{\sigma} p^{(r)} + A x^{(r+1)} &\in& y^{(r+1)} + \frac{1}{\sigma} \partial h( y^{(r+1)}),\\
y^{(r+1)} &=& \prox_{\frac{1}{\sigma}  h} \left( \frac{1}{\sigma} p^{(r)} + A x^{(r+1)}\right).
\end{eqnarray*}
There are several modifications of the basic linearized ADMM which improve its convergence properties as
\begin{itemize}
\item[-] the predictor corrector proximal multiplier method\index{predictor corrector proximal multiplier method} \cite{CT94},
\item[-] the primal dual hybrid gradient method (PDHG)\index{primal dual hybrid gradient method (PDHG)} \cite{ZC08} 
with convergence proof in \cite{BR2012},
\item[-] primal dual hybrid gradient method with extrapolation of the primal or dual variable \cite{CP11,PCCB09},
 a preconditioned version \cite{CP11a} and a generalization \cite{con2013}, Douglas-Rachford-type algorithms \cite{BH2013,BH2014} for solving inclusion
equations, see also \cite{CP2012,Vu2011}, 
as well an extension allowing the operator $A$ to be non-linear \cite{Valkonen2013}.
\end{itemize}
A good overview on primal-dual methods is also given in \cite{KP14}.
Here is the algorithm proposed by Chambolle, Cremers and Pock \cite{CP11,PCCB09}.
%
\begin{algorithm}[H]
\caption{Primal Dual Hybrid Gradient Method with Extrapolation of Dual Variable (PDHGMp) \label{alg:PDHGMp}}
\begin{algorithmic}
\STATE \textbf{Initialization:} $y^{(0)}, b^{(0)}= b^{(-1)}\in \mathbb R^m$, $\tau,\sigma > 0$ with $\tau \sigma < 1/\|A\|_2^2$ and $\theta \in (0,1]$
\STATE \textbf{Iterations:} For $r = 0,1,\ldots$
\STATE
$
\begin{array}{lcl}
x^{(r+1)} &=& \argmin_{x \in \mathbb R^d} \left\{ g(x) +  \frac{1}{2\tau} \|x - (x^{(r)} - \tau\sigma A^\tT \bar b^{(r)} )\|_2^2  \right\}\\
y^{(r+1)} &=& \argmin_{y \in \mathbb R^m} \left\{ h(y) + \frac{\sigma}{2} \|b^{(r)} + Ax^{(r+1)} - y\|_2^2\right\}\\
b^{(r+1)} &=& b^{(r)} + A x^{(r+1)} -  y^{(r+1)}.  \\ 
\bar b^{(r+1)}&=&  b^{(r+1)} + \theta (b^{(r+1)} - b^{(r)} )
\end{array}
$
\end{algorithmic}
\end{algorithm}
Note that the new first updating step can be also deduced by applying the so-called  
{\it inexact Uzawa algorithm}\index{inexact Uzawa algorithm} to the first ADMM step (see Section \ref{subsec:proximal_admm}). 
Furthermore, it can be directly seen that for $A$ being the identity and  $\theta = 1$ and $\gamma = \sigma = \frac1\tau$, 
the PDHGMp algorithm corresponds to the ADMM as well as the Douglas-Rachford splitting algorithm as proposed in Section \ref{subsec:admm}.
The following theorem and convergence proof are based on \cite{CP11}.
%
\begin{theorem} \label{chambolle_pock}
Let $g\in \Gamma_0(\R^d)$, $h \in \Gamma_0(\R^m)$ and $\theta \in (0,1]$.
Let $\tau, \sigma > 0$ fulfill 
\begin{equation} \label{cond_pdhg}
\tau \sigma < 1/\|A\|_2^2.
\end{equation} 
Suppose that the Lagrangian
$L(x,p) := g(x) - h^*(p) + \langle Ax,p \rangle$
has a saddle point $(x^*, p^*)$.
Then the sequence $\{(x^{(r)},p^{(r)})\}_r$
produced by PDGHMp converges to a saddle point of the Lagrangian.
\end{theorem}
%
\begin{proof} 
We restrict the proof to the case $\theta = 1$.
For arbitrary $\bar x \in \R^d, \bar p \in \R^m$ consider according to \eqref{xxnew1} and \eqref{xxnew2} the iterations 
\begin{eqnarray*}
x^{(r+1)}&=& (I + \tau \partial g)^{-1} \left( x^{(r)} - \tau A^\tT \bar p  \right),\\
p^{(r+1)}&=& (I + \sigma \partial h^*)^{-1} \left( p^{(r)} + \sigma A \bar x \right),
 \end{eqnarray*}
 i.e.,
$$
\frac{x^{(r)} - x^{(r+1)}}{\tau} - A^\tT \bar p \in   \partial g \left( x^{(r+1)}  \right),\quad
\frac{ p^{(r)} - p^{(r+1)}}{\sigma} + A \bar x \in   \partial h^* \left( p^{(r+1)}  \right).
$$
 By definition of the subdifferential we obtain for all $x \in \R^d$ and all $p \in \R^m$ that
 \begin{eqnarray*}
g(x) &\ge& g(x^{(r+1)}) + \frac{1}{\tau} \langle x^{(r)}- x^{(r+1)}, x - x^{(r+1)} \rangle - \langle A^\tT \bar p,  x - x^{(r+1)} \rangle ,\\
h^*(p) &\ge& h^*(p^{(r+1)}) + \frac{1}{\sigma} \langle p^{(r)}- p^{(r+1)}, p - p^{(r+1)} \rangle + \langle p - p^{(r+1)}, A \bar x \rangle
 \end{eqnarray*}
 and by adding the equations
 \begin{eqnarray*}
 0 &\ge& g(x^{(r+1)}) - h^*(p) - \left( g(x) - h^*(p^{(r+1)}) \right) 
 - \langle A^\tT \bar p,  x - x^{(r+1)} \rangle + \langle p - p^{(r+1)}, A \bar x \rangle\\
 &&
 +\; \frac{1}{\tau} \langle x^{(r)}- x^{(r+1)}, x - x^{(r+1)} \rangle 
 +\frac{1}{\sigma} \langle p^{(r)}- p^{(r+1)}, p - p^{(r+1)} \rangle.
 \end{eqnarray*}
 By
 $$
 \langle x^{(r)}- x^{(r+1)}, x - x^{(r+1)} \rangle
 =
 \frac12 \left( \|x^{(r)}- x^{(r+1)} \|_2^2 + \|x - x^{(r+1)} \|_2^2 - \|x - x^{(r)} \|_2^2\right)
 $$
 this can be rewritten as
 \begin{eqnarray*}
 &&
 \frac{1}{2\tau} \|x - x^{(r)} \|_2^2 + \frac{1}{2\sigma} \|p - p^{(r)} \|_2^2\\
 &\ge& 
 \frac{1}{2\tau} \|x^{(r)} - x^{(r+1)} \|_2^2 + \frac{1}{2\tau} \|x - x^{(r+1)} \|_2^2
 +
 \frac{1}{2\sigma} \|p^{(r)} - p^{(r+1)} \|_2^2 + \frac{1}{2\sigma} \|p - p^{(r+1)} \|_2^2
 \\
 && 
 + \; \left( g(x^{(r+1)}) - h^*(p) + \langle p, A x^{(r+1)} \rangle \right) 
 - \left( g(x) - h^*(p^{(r+1)}) + \langle p^{(r+1)}, A x \rangle \right) \\
 &&
 - \langle p, A x^{(r+1)} \rangle 
 + \langle p^{(r+1)}, A x \rangle
  - \langle \bar p,  A(x - x^{(r+1)}) \rangle 
  + \langle p - p^{(r+1)}, A \bar x \rangle
  \\
&=&
 \frac{1}{2\tau} \|x^{(r)} - x^{(r+1)} \|_2^2 + \frac{1}{2\tau} \|x - x^{(r+1)} \|_2^2 +
 \frac{1}{2\sigma} \|p^{(r)} - p^{(r+1)} \|_2^2 + \frac{1}{2\sigma} \|p - p^{(r+1)} \|_2^2
 \\
 && + \; 
 \left( g(x^{(r+1)}) - h^*(p) + \langle p, A x^{(r+1)} \rangle \right)   - 
 \left( g(x) - h^*(p^{(r+1)}) + \langle p^{(r+1)}, A x \rangle \right) \\
 &&
  + \; \langle p^{(r+1)} - p,  A(x^{(r+1)} - \bar x) \rangle 
  - \langle  p^{(r+1)} - \bar p, A (x^{(r+1)} - x) \rangle. 
 \end{eqnarray*}
 For any saddle point $(x^*, p^*)$ we have that
 $L(x^*,p) \le L(x^*, p^*) \le L(x,p^*)$ for all $x,p$ 
 so that in particular
 $0 \le L(x^{(r+1)},p^*) - L(x^*,p^{(r+1)})$.
 Thus, using $(x,p) := (x^*,p^*)$ in the above inequality, we get
 \begin{eqnarray*}
 &&\frac{1}{2\tau} \|x^* - x^{(r)} \|_2^2 + \frac{1}{2\sigma} \|p^* - p^{(r)} \|_2^2
 \\
 &\ge& 
 \frac{1}{2\tau} \|x^{(r)} - x^{(r+1)} \|_2^2 + \frac{1}{2\tau} \|x^* - x^{(r+1)} \|_2^2 +
 \frac{1}{2\sigma} \|p^{(r)} - p^{(r+1)} \|_2^2 + \frac{1}{2\sigma} \|p^* - p^{(r+1)} \|_2^2
 \\
 && +
 \langle p^{(r+1)} - p^*,  A(x^{(r+1)} - \bar x) \rangle 
  - \langle  p^{(r+1)} - \bar p, A (x^{(r+1)} - x^*) \rangle. 
 \end{eqnarray*}
 In the algorithm we use $\bar x := x^{(r+1)}$ and $\bar p := 2 p^{(r)} - p^{(r-1)}$.
 Note that $\bar p = p^{(r+1)}$ would be the better choice, but this is impossible
 if we want to keep on an explicit algorithm.
 For these values the above inequality further simplifies to
 \begin{eqnarray*}
&& \frac{1}{2\tau} \|x^* - x^{(r)} \|_2^2 + \frac{1}{2\sigma} \|p^* - p^{(r)} \|_2^2
 \\
 &\ge& 
 \frac{1}{2\tau} \|x^{(r)} - x^{(r+1)} \|_2^2 + \frac{1}{2\tau} \|x^* - x^{(r+1)} \|_2^2 +
 \frac{1}{2\sigma} \|p^{(r)} - p^{(r+1)} \|_2^2 + \frac{1}{2\sigma} \|p^* - p^{(r+1)} \|_2^2
 \\
 && 
 + \; \langle  p^{(r+1)} - 2p^{(r)} + p^{(r-1)}, A (x^* - x^{(r+1)}) \rangle. 
 \end{eqnarray*}
 We estimate the last summand using Cauchy-Schwarz's inequality as follows:
 \begin{eqnarray*}
 &&\langle  p^{(r+1)} - p^{(r)} -(p^{(r)}- p^{(r-1)}), A (x^* - x^{(r+1)}) \rangle
 \\
 &=&
 \langle  p^{(r+1)} - p^{(r)},A (x^* - x^{(r+1)}) \rangle
 -
 \langle  p^{(r)} - p^{(r-1)},A (x^* - x^{(r)}) \rangle
 \\
 &&
 - \;
 \langle  p^{(r)} - p^{(r-1)},A (x^{(r)} - x^{(r+1)}) \rangle
  \\
 &\ge&
 \langle  p^{(r+1)} - p^{(r)},A (x^* - x^{(r+1)} ) \rangle
 -
 \langle  p^{(r)} - p^{(r-1)},A ( x^* - x^{(r)}) \rangle
 \\
 &&- \;
 \|A\|_2 \|x^{(r+1)} - x^{(r)} \|_2 \, \|p^{(r)} - p^{(r-1)}\|_2.
 \end{eqnarray*}
 Since
 \begin{equation} \label{nett}
 2uv \le \alpha u^2 + \frac{1}{\alpha} v^2, \quad \alpha >0,
 \end{equation}
 we obtain
 \begin{eqnarray*}
 \|A\|_2 \|x^{(r+1)} - x^{(r)} \|_2 \, \|p^{(r)} - p^{(r-1)}\|_2
 &\le&
 \frac{\|A\|_2}{2} \left( \alpha \|x^{(r+1)} - x^{(r)} \|_2^2 \, + \, \frac{1}{\alpha} \|p^{(r)} - p^{(r-1)}\|_2^2 \right)
 \\
 &=&
 \frac{\|A\|_2 \alpha \tau}{2 \tau}  \|x^{(r+1)} - x^{(r)} \|_2^2 \, + \, \frac{\|A \|_2 \sigma}{2 \alpha \sigma} \|p^{(r)} - p^{(r-1)}\|_2^2 .
 \end{eqnarray*}
With $\alpha := \sqrt{\sigma/\tau}$ the relation
$$\|A\|_2 \alpha \tau = \frac{\|A \|_2 \sigma}{ \alpha} = \|A \|_2 \sqrt{\sigma \tau} < 1$$
holds true. Thus, we get 
\begin{eqnarray}
&& \frac{1}{2\tau} \|x^* - x^{(r)} \|_2^2 + \frac{1}{2\sigma} \|p^* - p^{(r)} \|_2^2 \nonumber
 \\
 &\ge& 
 \frac{1}{2\tau} \|x^* - x^{(r+1)} \|_2^2 + \frac{1}{2\sigma} \|p^* - p^{(r+1)} \|_2^2 \nonumber
 \\
 &&+ \;
 \frac{1}{2\tau} (1-\|A \|_2 \sqrt{\sigma \tau})  \|x^{(r+1)} - x^{(r)} \|_2^2
 +
 \frac{1}{2\sigma} \|p^{(r+1)} - p^{(r)} \|_2^2
 -
 \frac{ \|A \|_2 \sqrt{\sigma \tau} }{2\sigma} \|p^{(r)} - p^{(r-1)} \|_2^2 \nonumber
 \\
 &&+ \;
  \langle  p^{(r+1)} - p^{(r)},A (x^* - x^{(r+1)}) \rangle
 -
 \langle  p^{(r)} - p^{(r-1)},A (x^* - x^{(r)} ) \rangle. \label{zwischenresultat}
 \end{eqnarray}
 Summing up these inequalities from $r=0$ to $N-1$ and regarding that
 $p^{(0)} = p^{(-1)}$, we obtain
 \begin{eqnarray*}
&& \frac{1}{2\tau} \|x^* - x^{(0)} \|_2^2 + \frac{1}{2\sigma} \|p^* - p^{(0)} \|_2^2
 \\
 &\ge& 
 \frac{1}{2\tau} \|x^* - x^{(N)} \|_2^2 + \frac{1}{2\sigma} \|p^* - p^{(N)} \|_2^2 
 \\
 &&+ \;
 (1-\|A \|_2 \sqrt{\sigma \tau})  
 \left( \frac{1}{2\tau} \sum_{r=1}^N \|x^{(r)} - x^{(r-1)} \|_2^2
 +
 \frac{1}{2\sigma} \sum_{r=1}^{N-1} \|p^{(r)} - p^{(r-1)} \|_2^2
 \right)
 \\
 &&+ \;
 \frac{1}{2\sigma} \|p^{(N)} - p^{(N-1)} \|_2^2
 +
  \langle  p^{(N)} - p^{(N-1)},A (x^* - x^{(N)} ) \rangle
 \end{eqnarray*}
 By
 $$
  \langle  p^{(N)} - p^{(N-1)},A (x^{(N)} - x^*) \rangle
  \le 
  \frac{1}{2\sigma} \| p^{(N)} - p^{(N-1)} \|_2^2 
  +
  \frac{\|A\|_2^2 \sigma \tau}{2\tau} \| x^{(N)} - x^* \|_2^2 
 $$
 this can be further estimated as
 \begin{eqnarray}
 &&\frac{1}{2\tau} \|x^* - x^{(0)} \|_2^2 + \frac{1}{2\sigma} \|p^* - p^{(0)} \|_2^2 \nonumber
 \\
  &\ge& 
 \frac{1}{2\tau} (1 - \|A\|_2^2 \sigma \tau) \|x^* - x^{(N)} \|_2^2 + \frac{1}{2\sigma} \|p^* - p^{(N)} \|_2^2 \nonumber
 \\
 &&+ \;
 (1-\|A \|_2 \sqrt{\sigma \tau})  
 \left( \frac{1}{2\tau} \sum_{k=1}^N \|x^{(r)} - x^{(r-1)} \|_2^2+
 \frac{1}{2\sigma} \sum_{k=1}^{N-1} \|p^{(r)} - p^{(r-1)} \|_2^2
 \right). \label{fast_geschafft}
  \end{eqnarray}
 By \eqref{fast_geschafft} we conclude that the sequence $\{ ( x^{(n)},p^{(n)} ) \}_n$ is bounded.
 Thus, there exists a convergent subsequence $\{ ( x^{(n_j)},p^{(n_j)} ) \}_j$ which convergenes
 to some point $(\hat x,\hat p)$ as $j \rightarrow \infty$.
 Further, we see by \eqref{fast_geschafft} that
 $$
 \lim_{n \rightarrow \infty} \left( x^{(n)} - x^{(n-1)} \right) = 0, \quad
 \lim_{n \rightarrow \infty} \left(p^{(n)} - p^{(n-1)}  \right)= 0.
 $$
 Consequently,
 $$\lim_{j \rightarrow \infty} \left(x^{(n_j -1)} - \hat x \right) = 0,
 \quad
 \lim_{j \rightarrow \infty} \left(p^{(n_j - 1)} - \hat p  \right)= 0
 $$
 holds true.
 Let $T$ denote the iteration operator of the PDHGMp cycles,
 i.e., $T(x^{(r)},p^{(r)}) = (x^{(r+1)},p^{(r+1)})$.
 Since $T$ is the concatenation of affine operators and proximation operators,
 it is continuous.
 Now we have that
 $T\left(x^{(n_j-1)}, p^{(n_j-1)}\right) = \left(x^{(n_j)}, p^{(n_j)}\right)$
 and taking the limits for $j \rightarrow \infty$ we see that
 $T(\hat x,\hat p) = (\hat x,\hat p)$ so that $(\hat x,\hat p)$ 
 is a fixed point of the iteration
 and
 thus a saddle point of $L$.
 Using this particular saddle point in \eqref{zwischenresultat} and summing up from 
 $r=n_j$ to $N-1$, $N > n_j$ we obtain
 \begin{eqnarray*}
&& \frac{1}{2\tau} \|\hat x - x^{(n_j)} \|_2^2 + \frac{1}{2\sigma} \|\hat p - p^{(n_j)} \|_2^2 
 \\
 &\ge& 
 \frac{1}{2\tau} \|\hat x - x^{(N)} \|_2^2 + \frac{1}{2\sigma} \|\hat p - p^{(N)} \|_2^2 
 \\
 &&+ \;
 (1-\|A \|_2 \sqrt{\sigma \tau}) \left( \frac{1}{2\tau} \sum_{r = n_j}^{N-1} \|x^{(r+1)} - x^{(r)} \|_2^2
 +
 \frac{1}{2\sigma} \sum_{r = n_j+1}^{N-1} \|p^{(r)} - p^{(r-1)} \|^2 \right)
  \\
 &&+ \;
 \frac{1}{2\sigma} \|p^{(N)} - p^{(N-1)}\|^2 - \frac{\|A \|_2 \sqrt{\sigma \tau}}{2 \sigma} \|p^{(n_j)} - p^{(n_j-1)}\|_2^2 
 \\
 &&+ \;
  \langle  p^{(N)} - p^{(N-1)},A (\hat x - x^{(N)}) \rangle
 -
 \langle  p^{(n_j)} - p^{(n_j-1)},A (\hat x - x^{(n_j)}) \rangle 
 \end{eqnarray*}
 and further
 \begin{eqnarray*}
&& \frac{1}{2\tau} \|\hat x - x^{(n_j)} \|_2^2 + \frac{1}{2\sigma} \|\hat p - p^{(n_j)} \|_2^2 
 \\
 &\ge& 
 \frac{1}{2\tau} \|\hat x - x^{(N)} \|_2^2 + \frac{1}{2\sigma} \|\hat p - p^{(N)} \|_2^2 
 \\
  &&+ \;
 \frac{1}{2\sigma} \|p^{(N)} - p^{(N-1)}\|_2^2 - \frac{\|A \|_2 \sqrt{\sigma \tau}}{2 \sigma} \|p^{(n_j)} - p^{(n_j-1)}\|_2^2 
 \\
 &&+ \;
  \langle  p^{(N)} - p^{(N-1)},A (\hat x - x^{(N)}) \rangle
 -
 \langle  p^{(n_j)} - p^{(n_j-1)},A (\hat x - x^{(n_j)}) \rangle 
\end{eqnarray*}
 For $j \rightarrow \infty$ this implies that $(x^{(N)},p^{(N)})$ converges also to $(\hat x, \hat y)$
 and we are done. 
 \end{proof}
\subsection{Proximal ADMM} \label{subsec:proximal_admm}
To avoid the computation of a linear system of equations in the first ADMM step \eqref{first_iter}, 
an alternative to  the linearized ADMM offers the proximal ADMM algorithm \cite{He2002,ZBO09} that 
can be interpreted as a preconditioned variant of ADMM. In this variant the minimization step \eqref{first_iter} 
is replaced by a proximal-like iteration based on the general proximal operator \eqref{precond},
\begin{equation} \label{proximal_admm}
x^{(r+1)} = \argmin_{x \in \R^d} \{ L_\gamma(x,y^{(r)},p^{(r)}) + \frac{1}{2} \| x - x^{(r)}\|^2_R \} 
\end{equation}
with a symmetric, positive definite matrix $R \in \R^{d,d}$. 
The introduction of $R$ provides an additional flexibility to cancel out linear operators which might be difficult to invert. 
In addition the modified minimization problem is strictly convex inducing a unique minimizer. 
In the same manner the second ADMM step \eqref{second_iter} can also be extended by a proximal term $(1/2) \| y - y^{(r)}\|^2_S$ 
with a symmetric, positive definite matrix $S \in \R^{m,m}$ \cite{ZBO09}. 
The convergence analysis of the proximal ADMM was provided in \cite{ZBO09} 
and the algorithm can be also classified as an inexact Uzawa method. 
A generalization, where the matrices $R$ and $S$ can non-monotonically vary in each iteration step, 
was analyzed in \cite{He2002}, additionally allowing an inexact computation of minimization problems.

In case of the PDHGMp algorithm, it was mentioned that the first updating step can be deduced 
by applying the inexact Uzawa algorithm to the first ADMM step. 
Using the proximal ADMM, it is straightforward to see that the first updating step of the PDHGMp algorithm with $\theta = 1$ 
corresponds to \eqref{proximal_admm} in case of $R = \frac1{\tau} I - \sigma A^\tT A$ with $0 < \tau < 1 / \| \sigma A^\tT A \|$, see \cite{CP11,Esser2010}. 
Further relations of the (proximal) ADMM to primal dual hybrid methods discussed above can be found in \cite{Esser2010}.
\subsection{Bregman Methods} \label{subsec:bregman}
Bregman methods became very popular in image processing 
by a series papers of Osher and co-workers, see, e.g.,
\cite{GO09,OBGXY05}. Many of these methods can be interpreted as ADMM methods and its linearized versions.
In the following we briefly sketch these relations.
 
The PPA is a special case of the Bregman PPA.
Let
$\varphi: \R^d \rightarrow \R \cup \{+ \infty\}$ by a convex function.
Then the {\it Bregman distance} $D_\varphi^{p}: \R^d \times \R^d \rightarrow \R$ is given by
$$
D_\varphi^{p} (x,y) = \varphi(x) - \varphi(y) - \langle p,x-y \rangle
$$
with $p \in \partial \varphi(y)$, $y \in \dom f$. If $\partial \varphi(y)$ contains only one element, we just write $D_\varphi$.
If $\varphi$ is smooth, then the Bregman distance can be interpreted as substracting the first order Taylor expansion 
of $\varphi(x)$ at $y$.
\begin{example} {\rm (Special Bregman Distances)}\\
1. The Bregman distance corresponding to  $\varphi (x):= \frac12 \| x \|^2_2$ is given by
$D_\varphi (x,y) = \frac12 \|x-y \|^2_2$.
\\
2. For the negative  Shannon entropy $\varphi (x) := \langle 1_d, x \log x \rangle$, $x > 0$ we obtain
the (discrete) $I$-divergence or generalized Kullback-Leibler entropy
$D_\varphi (x,y) = x \log \frac{x}{y} - x + y$.
\end{example}
For $f \in \Gamma_0(\mathbb R^d)$ we consider the generalized proximal problem
$$
\argmin_{y \in \R^d} \left\{ \frac{1}{\gamma} D_\varphi^{p} (x,y)  + f(y) \right\}.
$$
The Bregman Proximal Point Algorithm (BPPA) for solving this problem reads as follows:
\begin{algorithm}[H]
\caption{Bregman Proximal Point Algorithm (BPPA) \label{a:BPP} }
\begin{algorithmic}
\STATE \textbf{Initialization:} $x^{(0)} \in \R^d$, $p^{(0)} \in \partial \varphi(x^{(0)})$, $\gamma >0$
\STATE \textbf{Iterations:} For $r = 0,1,\ldots$
\STATE
$
\begin{array}{l}	
x^{(r+1)} = \argmin_{y \in \R^d} \left\{ \frac{1}{\gamma} D_\varphi^{p^{(r)}} (y,x^{(r)})  + f(y) \right\}\\
p^{(r+1)} \in \partial \varphi(x^{(r+1)} )
\end{array}
$
\end{algorithmic}
\end{algorithm}
The BPPA converges for any initialization $x^{(0)}$ to a minimizer of $f$ 
if $f\in \Gamma_0(\mathbb R^d)$ attains its minimum and $\varphi$ is finite, lower semi-continuous and strictly convex.
For  convergence proofs we refer, e.g., to \cite{Kiw97_1,Kiw97}.
We are interested again in the problem \eqref{problem_pdhg_1}, i.e.,
$$
\argmin_{x \in \mathbb R^d, y \in \mathbb R^m} \left\{ \Phi(x,y)  \quad {\rm s.t.} \quad A x = y \right\}, \quad \Phi(x,y) := g(x) + h(y).
$$
We consider the BPP algorithm for the objective function $f(x,y):=\frac{1}{2} \|Ax-y\|^2$ 
with the Bregman distance
\[
 D^{(p_x^{(r)},p_y^{(r)}) }_\Phi \left( (x,y),(x^{(r)},y^{(r)}) \right)= \Phi(x,y) - \Phi(x^{(r)},y^{(r)})-\lb p_x^{(r)},x-x^{(r)}\rb -\lb p_y^{(r)},y-y^{(r)}\rb,
\]
where
$\big( p_x^{(r)}, p_y^{(r)} \big)\in \partial \Phi(x^{(r)},y^{(r)})$. 
This results in 
\begin{align}
 (x^{(r+1)},y^{(r+1)})&=\argmin_{x \in \mathbb R^d, y \in \mathbb R^m} \big\{ \frac{1}{\gamma} D^{(p_x^{(r)},p_y^{(r)})}_{\Phi} \left( (x,y),(x^{(r)},y^{(r)}) \right)
 + \frac{1}{2} \|Ax -y\|^2 \big\}, \nonumber\\
 p^{(r+1)}_x&= p_x^{(r)}-\gamma A^{*}(A x^{(r+1)}-y^{(r+1)}), \label{sp2}\\
 p^{(r+1)}_y &= p_y^{(r)}+\gamma (A x^{(r+1)}-y^{(r+1)}), \label{sp3}
\end{align}
where we have used that the first equation  implies
\begin{align*}
 0&\in \frac{1}{\gamma}\partial \left( \Phi(x^{(r+1)},y^{(r+1)})-\big( p_x^{(r)},p_y^{(r)} \big) \right) + \big( 
 A^{*}(A x^{(r+1)}-y^{(r+1)}),- (A x^{(r+1)}-y^{(r+1)})\big), \\
0&\in \partial \Phi(x^{(r+1)},y^{(r+1)}) - \big( p_x^{(r+1)}, p_y^{(r+1)} \big),
\end{align*}
so that $\big( p_x^{(r)}, p_y^{(r)} \big)\in \partial \Phi(x^{(r)},y^{(r)})$.
From \eqref{sp2} and \eqref{sp3} we see by induction that $p_x^{(r)} = -A^* p_y^{(r)}$.
Setting $p^{(r)} = p^{(r)}_y$
and regarding that 
\begin{align*}
 & \frac{1}{\gamma} D^{p^{(r)}}_\Phi \left((x,y),(x^{(r)},y^{(r)}) \right) + \frac{1}{2 } \|Ax - y\|_2^2 \\
&= {\rm const} + \frac{1}{\gamma} \left( \Phi(x,y) + \langle  A^* p^{(r)} ,x \rangle - \langle  p^{(r)} ,y \rangle \right) + \frac{1}{2 } \|Ax - y\|_2^2\\
& = {\rm const} + \frac{1}{\gamma} \left( \Phi(x,y) + \frac{\gamma}{2} \|\frac{p^{(r)} }{\gamma} + Ax -y\|_2^2  \right)
\end{align*}
we obtain the following split Bregman method, see \cite{GO09}:\\
\begin{algorithm}[H]
\caption{Split Bregman Algorithm \label{a:split_breg} }
\begin{algorithmic}
\STATE \textbf{Initialization:} $x^{(0)} \in \R^d$, $p^{(0)}$, $\gamma >0$
\STATE \textbf{Iterations:} For $r = 0,1,\ldots$
\STATE
$
\begin{array}{ll}
(x^{(r+1)},y^{(r+1)})  &= \argmin\limits_{x \in \R^d,y \in \R^m} \left\{ \Phi(x,y) + \frac{\gamma}{2} \| \frac{p^{(r)}}{\gamma} + A x - y\|_2^2 \right\}\\
p^{(r+1)} &= p^{(r)} + \gamma(A x^{(r+1)} - y^{(r+1)})
\end{array}
$
\end{algorithmic}
\end{algorithm}
Obviously, this is exactly the form of the augmented Lagrangian method in \eqref{ALM}.
%
\section{Iterative Regularization for Ill-posed Problems}
So far we have discussed the use of splitting methods for the numerical solution of well-posed variational problems, 
which arise in a discrete setting and in particular for the standard approach of Tikhonov-type regularization in inverse problems in imaging. 
The latter is based on minimizing a weighted sum of a data fidelity and a regularization functional, and can be more generally analyzed 
in Banach spaces, cf. \cite{schuster2012regularization}. 
However, such approaches have several disadvantages, 
in particular it has been shown that they lead to unnecessary bias in solutions, e.g., 
a contrast loss in the case of total variation regularization, cf. \cite{osher2005iterative,burger2013guide}. 
A successful alternative to overcome this issue is iterative regularization, 
which directly applies iterative methods to solve the constrained variational problem
\begin{equation}
	\argmin_{x \in {\cal X}} \left\{ g(x) \quad {\rm s.t.} \quad A x = f \right\}. \label{eq:minimumnorm}
\end{equation}
Here $A:{\cal X} \rightarrow {\cal Y}$ is a bounded linear operator between Banach spaces (also nonlinear versions can be considered, 
cf. \cite{bachmayr2009iterative,kaltenbacher2008iterative}) and $f$ are given data. 
In the well-posed case, \eqref{eq:minimumnorm} can be rephrased as the saddle-point problem
\begin{equation}
	\min_{x \in {\cal X}} \sup_q \left( g(x) - \langle q, A x - f\rangle \right) \label{eq:saddlepointip}
\end{equation}

The major issue compared to the discrete setting is that for many prominent examples the operator $A$ 
does not have a closed range (and hence a discontinuous pseudo-inverse), 
which makes \eqref{eq:minimumnorm} ill-posed. From the optimization point of view, this raises two major issues: 
\begin{itemize}
\item {\em Emptyness of the constraint set:} 
In the practically relevant case of noisy measurements one has to expect that $f$ is not in the range of $A$, 
i.e., the constraint cannot be satisfied exactly. Reformulated in the constrained optimization view, 
the standard paradigm in iterative regularization is to construct an iteration slowly increasing the functional 
$g$ while decreasing the error in the constraint. 
\item {\em Nonexistence of saddle points:} 
Even if the data or an idealized version $A x^*$ to be approximated are in the range of $A$, 
the existence of a saddle point $(x^*,q^*)$ of \eqref{eq:saddlepointip} is not guaranteed. 
The optimality condition for the latter would yield
\begin{equation}
	A^* q^* \in \partial g(x^*),
\end{equation}
which is indeed an abstract smoothness condition on the subdifferential of $g$ at $x^*$ if $A$ and consequently $A^*$ are smoothing operators, 
it is known as source condition in the field of inverse problems, cf. \cite{burger2013guide}.  
\end{itemize}
Due to the above reasons the use of iterative methods for solving respectively approximating \eqref{eq:minimumnorm} 
has a different flavour than iterative methods for well-posed problems. 
The key idea is to employ the algorithm as an iterative regularization method, cf. \cite{kaltenbacher2008iterative}, 
where appropriate stopping in dependence on the noise, i.e. a distance between $Ax^*$ and $f$, 
needs to be performed in order to obtain a suitable approximation. 
The notion to be used is called semiconvergence, i.e., if $\delta > 0$ denotes a measure for the data error (noise level) and $\hat r(\delta)$ 
is the stopping index of the iteration in dependence on $\delta$, then we look for convergence
\begin{equation}
	x^{(\hat r(\delta))} \rightarrow x^* \qquad \mbox{as~} \delta \rightarrow 0,
\end{equation}
in a suitable topology. The minimal ingredient in the convergence analysis is the convergence $x^{(r)} \rightarrow x^*$, 
which already needs different approaches as discussed above. For iterative methods working on primal variables one can at least use the existence of \eqref{eq:minimumnorm} in this case, 
while real primal-dual iterations still suffer from the potential nonexistence of solutions of the saddle point problem \eqref{eq:saddlepointip}. 

A well-understood  iterative method  is the Bregman iteration 
\begin{equation}
	x^{(r+1)} \in \argmin_{x \in {\cal X}} \left( \frac{\mu}2 \Vert A x - f \Vert^2 + D_g^{p^{(r)}}(x,x^{(r)}) \right),
\end{equation}
with $p^{(r)} \in \partial g(x^{(r)})$, which has been analyzed as an iterative method in \cite{osher2005iterative}, 
respectively for nonlinear $A$ in \cite{bachmayr2009iterative}. 
Note that again with $p^{(r)} = A^* q^{(r)}$ the Bregman iteration is equivalent to the augmented Lagrangian method 
for the saddle-point problem \eqref{eq:saddlepointip}. With such iterative regularization methods superior results 
compared to standard variational methods can be computed for inverse and imaging problems, in particular bias can be eliminated, cf. \cite{burger2013guide}.

The key properties are the {\em decrease of the data fidelity} 
\begin{equation}
	\Vert A x^{(r+1)} - f \Vert^2 \leq \Vert A x^{(r)} - f \Vert^2, 
\end{equation}
for all $r$ and the {\em decrease of the Bregman distance to the clean solution}
\begin{equation}
	D_g^{p^{(r+1)}}(x^*,x^{(r+1)}) \leq D_g^{p^{(r)}}(x^*,x^{(r)}) 
\end{equation}
for those $r$ such that 
$$ \Vert A x^{(r)} - f \Vert \geq \Vert A x^* -f \Vert = \delta. $$
Together with a more detailed analysis of the difference between consecutive Bregman distances, 
this can be used to prove semiconvergence results in appropriate weak topologies, cf. \cite{osher2005iterative,schuster2012regularization}. 
In \cite{bachmayr2009iterative} further variants approximating the quadratic terms, such as the linearized Bregman iteration, are analyzed, 
however with further restrictions on $g$. For all other iterative methods discussed above, 
a convergence analysis in the case of ill-posed problems is completely open and appears to be a valuable task for future research. 
Note that in the realization of the Bregman iteration, a well-posed but complex variational problem needs to be solved in each step. 
By additional operator splitting in an iterative regularization method one could dramatically reduce the computational effort. 

If the source condition is satisfied, i.e. if there exists a saddle-point $(x^*,q^*)$, one can further exploit the decrease of {\em dual distances}
\begin{equation}
	\Vert q^{(r+1)} - q^* \Vert \leq \Vert q^{(r)} - q^* \Vert
\end{equation}
to obtain a quantitative estimate on the convergence speed, we refer to \cite{Brune2010,burger2007error} for a further discussion.

\section{Applications}
So far we have focused on technical aspects of first order algorithms whose (further) development has been heavily forced by practical applications. 
In this section we give a rough overview of the use of first order algorithms in practice. 
We start with applications from classical imaging tasks such as computer vision and image analysis and proceed to applications in natural and life sciences. 
From the area of biomedical imaging, we will present the {\it Positron Emission Tomography (PET)}\index{Positron Emission Tomography (PET)} and {\it Spectral X-ray CT}\index{Spectral CT} 
in more detail and show some results reconstructed with first order algorithms.

At the beginning it is worth to emphasize that many algorithms based on proximal operators such as proximal point algorithm, 
proximal forward-backward splitting, ADMM, or Douglas-Rachford splitting have been introduced in the 1970s, cf. \cite{GM76,LM79,Ro76a,Ro76}. 
However these algorithms have found a broad application in the last two decades, 
mainly caused by the technological progress. 
Due to the ability of distributed convex optimization with ADMM related algorithms, 
these algorithms seem to be qualified for 'big data' analysis and large-scale applications in applied statistics and machine learning, e.g., in areas as artificial intelligence, 
internet applications, computational biology and medicine, finance, network analysis, or logistics \cite{BPCPE11,PB2013}. 
Another boost for the popularity of first order splitting algorithms was the 
increasing use of sparsity-promoting regularizers based on $\ell_1$- or $L_1$-type penalties \cite{GO09,ROF92}, 
in particular in combination with inverse problems considering non-linear image formation models \cite{bachmayr2009iterative,Valkonen2013} 
and/or statistically derived (inverse) problems \cite{burger2013guide}. The latter problems lead to non-quadratic fidelity terms 
which result from the non-Gaussian structure of the noise model. The overview given in the following mainly concentrates on the latter mentioned applications.

The most classical application of first order splitting algorithms is 
image analysis such as denoising, where these methods were originally pushed by the Rudin, Osher, and Fatemi (ROF) model \cite{ROF92}. 
This model and its variants are frequently used as prototypes for total variation methods in imaging 
to illustrate the applicability of proposed algorithms in case of non-smooth cost functions, cf. \cite{CP11,Es09,Esser2010,GO09,OBGXY05,Se11,ZC08}. 
Since the standard $L_2$ fidelity term is not appropriate for non-Gaussian noise models, modifications of the ROF problem have been considered in the past 
and were solved using splitting algorithms to denoise images perturbed by non-Gaussian noise, cf. \cite{Bioucas-Dias2010,CP11,FB09,Sawatzky2013,Steidl2010}. 
Due to the close relation of total variation techniques to image segmentation \cite{burger2013guide,PCCB09}, 
first order algorithms have been also applied in this field of applications (cf. \cite{CP11a,CP11,Goldstein2010,PCCB09}). 
Other image analysis tasks where proximal based algorithms have been applied successfully are deblurring and zooming (cf. \cite{BR2012,CP11,Esser2010,Setzer2010}), 
inpainting \cite{CP11}, stereo and motion estimation \cite{CEFPP12,CP11,YSS07}, and segmentation \cite{BYT09,LKYBCS09,PCCB09,KSS14}.

Due to increasing noise level in modern biomedical applications, 
the requirement on statistical image reconstruction methods has been risen recently and the proximal methods 
have found access to many applied areas of biomedical imaging. 
Among the enormous amount of applications from the last two decades, 
we only give the following selection and further links to the literature:
\begin{itemize}
\item \emph{X-ray CT:} 
Recently statistical image reconstruction methods have received increasing attention in X-ray CT due to increasing noise level 
encountered in modern CT applications such as sparse/limited-view CT and low-dose imaging, cf., e.g., \cite{Vandeghinste2011,Wang2008,Wang2006}, 
or K-edge imaging where the concentrations of K-edge materials are inherently low, see, e.g., \cite{Schirra2013,Schirra2014,Schlomka2008}. 
In particular, first order splitting methods have received strong attention due to the ability to handle non-standard noise models and sparsity-promoting regularizers efficiently. 
Beside the classical fan-beam and cone-beam X-ray CT (see, e.g., \cite{Anthoine2011a,Chartrand2013,Choi2010,Jia2010,Nien2014,Ramani2012,Sidky2012,Vandeghinste2011}), 
the algorithms have also found applications in emerging techniques such as spectral CT, see Section \ref{sec:spectralCT} and \cite{Gao2011,Sawatzky2014,Xu2014} 
or phase contrast CT \cite{Cong2012,Nilchian2013,Xu2014a}.
\item \emph{Magnetic resonance imaging (MRI):} 
Image reconstruction in MRI is mainly achieved 
by inverting the Fourier transform which can be performed efficiently and robustly if a sufficient number 
of Fourier coefficients is measured. However, this is not the case in special applications such as 
fast MRI protocols, cf., e.g., \cite{Lustig2007,Pruessmann1999}, where the Fourier space is undersampled 
so that the Nyquist criterion is violated and Fourier reconstructions exhibit aliasing artifacts. 
Thus, compressed sensing theory have found the way into MRI by exploiting sparsity-promoting variational approaches, see, e.g., \cite{Benning2014,Huang2011,Ma2008,Ramani2011}. 
Furthermore, in advanced MRI applications such as velocity-encoded MRI or diffusion MRI, the measurements can be modeled more accurately by non-linear operators 
and splitting algorithms provide the ability to handle the increased reconstruction complexity efficiently \cite{Valkonen2013}.
\item \emph{Emission tomography:} 
Emission tomography techniques used in nuclear medicine such as 
\emph{positron emission tomography (PET)} and \emph{single photon emission computed tomography (SPECT)} \cite{Wernick2004} 
are classical examples for inverse problems in biomedical imaging where statistical modeling of the reconstruction problem is essential 
due to Poisson statistics of the data. In addition, in cases where short time or low tracer dose measurements are available (e.g., using cardiac 
and/or respiratory gating \cite{Buther2009}) or tracer with a short radioactive half-life are used (e.g., radioactive water H$_2$$^{15}$O \cite{Schafers2002}), 
the measurements suffer from inherently high noise level and thus a variety of first order splitting algorithms has been utilized in emission tomography, 
see, e.g., \cite{Anthoine2011a,Benning2010,Mehranian2013,Muller2011,Pustelnik2010,Sawatzky2013}.
\item \emph{Optical microscopy:} 
In modern light microscopy techniques such as \emph{stimulated emission depletion (STED)} or \emph{4Pi-confocal fluorescence microscopy} 
\cite{Hell2003,Hell2007} resolutions beyond the diffraction barrier can be achieved allowing imaging at nanoscales. 
However, by reaching the diffraction limit of light, measurements suffer from blurring effects and Poisson noise with low photon count rates \cite{Dey2006,Schrader1998}, 
in particular in live imaging and in high resolution imaging at nanoscopic scales. Thus, regularized (blind) deconvolution addressing appropriate Poisson noise is quite beneficial 
and proximal algorithms have been applied to achieve this goal, cf., e.g., \cite{Brune2010,Frick2013,Sawatzky2013,Setzer2010}.
\item \emph{Other modalities:} It is quite natural that first order splitting algorithms have found a broad usage in biomedical imaging, 
in particular in such applications where the measurements are highly perturbed by noise and thus regularization with probably a proper statistical modeling are essential as, e.g., 
in optical tomography \cite{Abascal2011,Freiberger2010}, medical ultrasound imaging \cite{Sawatzky2013a}, hybrid photo-/optoacoustic tomography \cite{Gao2012,Wang2012}, 
or electron tomography \cite{Goris2012}.
\end{itemize}



\subsection{Positron Emission Tomography (PET)} \label{sec:PET}

PET is a biomedical imaging technique visualizing biochemical and physiological processes such as glucose metabolism, blood flow, or receptor concentrations, see, e.g., \cite{Wernick2004}. 
This modality is mainly applied in nuclear medicine and the data acquisition is based on weak radioactively marked pharmaceuticals (so-called tracers), 
which are injected into the blood circulation. Then bindings dependent on the choice of the tracer to the molecules are studied. Since the used markers are radio-isotopes, 
they decay by emitting a positron which annihilates almost immediately with an electron. The resulting emission of two photons is detected and, due to the radioactive decay, 
the measured data can be modeled as an inhomogeneous Poisson process with a mean given by the X-ray transform of the spatial tracer distribution (cf., e.g., \cite{Natterer2001,Vardi1985}). 
Note that, up to notation, the X-ray transform coincides with the more popular Radon transform in the two dimensional case \cite{Natterer2001}. Thus, the underlying reconstruction problem can be modeled as
\begin{equation}
	\label{eq:PETReconProblem}
	\sum_{m = 1}^M \bigl( (Ku)_m - f_m \log((Ku)_m) \bigr) + \alpha R(u) \rightarrow \min_{u \geq 0}, \qquad \alpha > 0,
\end{equation}
where $M$ is the number of measurements, $f$ are the given data, and $K$ is the system matrix which describes the full properties of the PET data acquisition.

To solve \eqref{eq:PETReconProblem}, algorithms discussed above can be applied and several of them have been already studied for PET recently. In the following, 
we will give a (certainly incomplete) performance discussion of different first order splitting algorithms on synthetic PET data and highlight the strengths 
and weaknesses of them which could be carried over to many other imaging applications. For the study below, the total variation (TV) 
was applied as regularization energy $R$ in \eqref{eq:PETReconProblem} and the following algorithms and parameter settings were used for the performance evaluation:
\begin{itemize}
	\item \emph{\textbf{FB-EM-TV:}} The FB-EM-TV algorithm \cite{Sawatzky2013} represents an instance of the proximal 
forward-backward (FB) splitting algorithm discussed in Section \ref{sec:FBalg} using a variable metric strategy \eqref{eq:VarMetricStrategy}. 
The preconditioned matrices $Q^{(r)}$ in \eqref{eq:VarMetricStrategy} are chosen in a way that the gradient descent step corresponds to an expectation-maximization (EM) reconstruction step. 
The EM algorithm is a classically applied (iterative) reconstruction method in emission tomography \cite{Natterer2001,Vardi1985}. 
The TV proximal problem was solved by an adopted variant of the modified Arrow-Hurwicz method proposed in \cite{CP11} since it was shown to be the most efficient method 
for TV penalized weighted least-squares denoising problems in \cite{Sawatzky2014a}. Furthermore, a warm starting strategy was used to initialize the dual variables 
within the TV proximal problem and the inner iteration sequence was stopped if the relative error of primal and dual optimality conditions was below an error tolerance $\delta$, 
i.e., using the notations from \cite{CP11}, if
	\begin{equation} \label{eq:StopAHMOD}
		\max\{d^{(r)},p^{(r)}\} \leq \delta
	\end{equation}
	with
	\begin{eqnarray*}
		d^{(r)} & = & \| (y^{(r)} - y^{(r-1)}) / \sigma_{r-1} + \nabla (x^{(r)} - x^{(r-1)})\| \, / \, \|\nabla x^{(r)}\| , \\
		p^{(r)} & = & \| x^{(r)} - x^{(r-1)} \| \, / \, \| x^{(r)} \| .
	\end{eqnarray*}
	The damping parameter $\eta^{(r)}$ in \eqref{eq:VarMetricStrategy} was set to $\eta^{(r)} = 1$ as indicated in \cite{Sawatzky2013}.
	\item \emph{\textbf{FB-EM-TV-Nes83:}} A modified version of FB-EM-TV described above using the acceleration strategy proposed by Nesterov in \cite{Ne83}. 
This modification can be seen as a variant of FISTA \cite{BT09} with a variable metric strategy \eqref{eq:VarMetricStrategy}. Here, $\eta^{(r)}$ in \eqref{eq:VarMetricStrategy} 
was chosen fixed (i.e. $\eta^{(r)} = \eta$) but has to be adopted to the pre-defined inner accuracy threshold $\delta$ \eqref{eq:StopAHMOD} to guarantee the convergence 
of the algorithm and it was to be done manually.
	\item \emph{\textbf{CP-E:}} The fully explicit variant of the Chambolle-Pock's primal-dual algorithm \cite{CP11} (cf. Section \ref{subsec:pdhg}) studied 
for PET reconstruction problems in \cite{Anthoine2011} (see \emph{CP2TV} in \cite{Anthoine2011}). The dual step size $\sigma$ was set manually and the primal one corresponding to \cite{CP11} 
as $\tau \sigma (\|\nabla\|^2 + \|K\|^2) = 1$, where $\|K\|$ was pre-estimated using the Power method.
	\item \emph{\textbf{Precond-CP-E:}} The CP-E algorithm described above but using the diagonal preconditioning strategy proposed in \cite{CP11a} with $\alpha=1$ in \cite[Lemma 2]{CP11a}.
	\item \emph{\textbf{CP-SI:}} The semi-implicit variant of the Chambolle-Pock's primal-dual algorithm \cite{CP11} (cf. Section \ref{subsec:pdhg}) 
studied for PET reconstruction problems in \cite{Anthoine2011} (see \emph{CP1TV} in \cite{Anthoine2011}). The difference to CP-E is that a 
TV proximal problem has to be solved in each iteration step. This was performed as in case of FB-EM-TV method. Furthermore, the dual step size $\sigma$ was set manually 
and the primal one corresponding to \cite{CP11} as $\tau \sigma \|K\|^2 = 1$, where $\|K\|$ was pre-estimated using the Power method. 
	\item \emph{\textbf{Precond-CP-SI:}} The CP-SI algorithm described above but using the diagonal preconditioning strategy proposed in \cite{CP11a} 
with $\alpha=1$ in \cite[Lemma 2]{CP11a}.
	\item \emph{\textbf{PIDSplit+:}} An ADMM based algorithm (cf. Section \ref{subsec:admm}) that has been discussed for Poisson deblurring problems 
of the form \eqref{eq:PETReconProblem} in \cite{Setzer2010}. However, in case of PET reconstruction problems, the solution of a linear system of equations of the form
	\begin{equation}
		\label{eq:PIDSplitLSE}
		(I + K^\tT K + \nabla^\tT\nabla) u^{(r+1)} = z^{(r)}
	\end{equation}
has to be computed in a different way in contrast to deblurring problems. This was done by running two preconditioned conjugate gradient (PCG) iterations 
with warm starting and cone filter preconditioning whose effectiveness has been validated in \cite{Ramani2012} for X-ray CT reconstruction problems. 
The cone filter was constructed as described in \cite{Fessler1997,Fessler1999} and diagonalized by the discrete cosine transform (DCT-II) supposing Neumann boundary conditions.
The PIDSplit+ algorithm described above can be accomplished by a strategy of adaptive augmented 
Lagrangian parameters $\gamma$ in \eqref{Lagrangian_augmented} as proposed for the PIDSplit+ algorithm in \cite{BPCPE11,Teuber2012}. 
The motivation behind this strategy 
is to mitigate the performance dependency on the initial chosen fixed parameter that may strongly influence the speed of convergence of ADMM based algorithms.
\end{itemize}

All algorithms were implemented in MATLAB and executed on a machine with 4 CPU cores, each 2.83 GHz, and 7.73 GB physical memory, running a 64 bit Linux system and MATLAB 2013b. 
The built-in multi-threading in MATLAB was disabled such that all computations were limited to a single thread. The algorithms were evaluated on a simple object 
(image size $256 \times 256$) and the synthetic 2D PET measurements were obtained via a Monte-Carlo simulation with $257$ radial and $256$ angular samples, 
using one million simulated events (see Figure \ref{fig:PETdata}). Due to sustainable image and measurement dimensions, the system matrix $K$ was pre-computed 
for all reconstruction runs. To evaluate the performance of algorithms described above, the following procedure was applied. First, since $K$ is injective 
and thus an unique minimizer of \eqref{eq:PETReconProblem} is guaranteed \cite{Sawatzky2013}, we can run a well performing method for a very long time 
to compute a "ground truth" solution $u^*_\alpha$ for a fixed $\alpha$. To this end, we have run the Precond-CP-E algorithm for 100,000 iterations 
for following reasons: (1) all iteration steps can be solved exactly such that the solution cannot be influenced by inexact computations (see discussion below); 
(2) due to preconditioning strategy, no free parameters are available those may influence the speed of convergence negatively such that $u^*_\alpha$ is expected 
to be of high accuracy after 100,000 iterations. Having $u^*_\alpha$, each algorithm was applied until the relative error
\begin{equation} \label{eq:relErr}
	\| u^{(r)}_\alpha - u^*_\alpha \| / \| u^*_\alpha \|
\end{equation}
was below a pre-defined threshold $\epsilon$ (or a maximum number of iterations adopted for each algorithm individually was reached). 
The "ground truth" solutions for three different values of $\alpha$ are shown in Figure \ref{fig:PETgroundtruth}.

\begin{figure}%
	\centering
		\includegraphics[width=.32\textwidth]{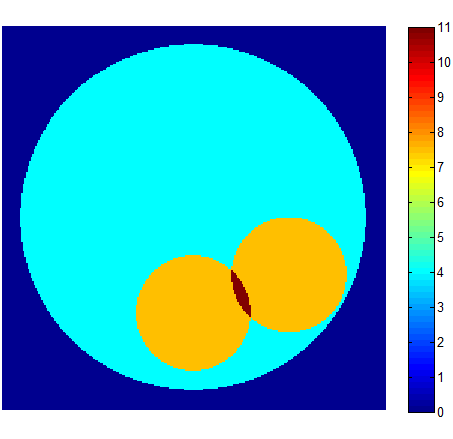}%
		\hfil
		\includegraphics[width=.32\textwidth]{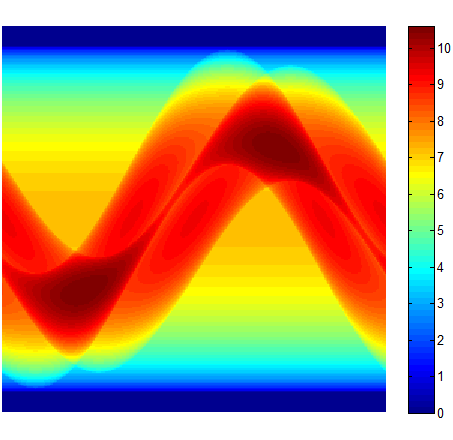}%
		\hfil
		\includegraphics[width=.32\textwidth]{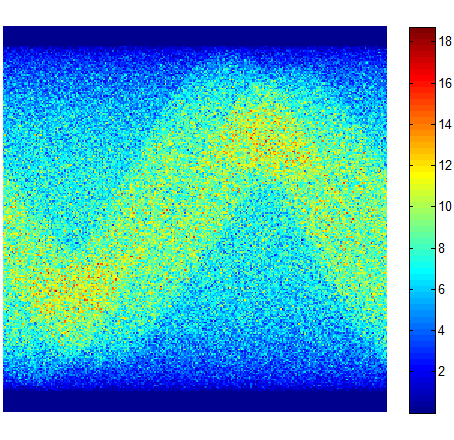}%
	\caption{Synthetic 2D PET data. \emph{Left:} Exact object. \emph{Middle:} Exact Radon data. \emph{Right:} Simulated PET measurements via a Monte-Carlo simulation using one million events.}%
	\label{fig:PETdata}%
\end{figure}

\begin{figure}%
	\centering
		\includegraphics[width=.32\textwidth]{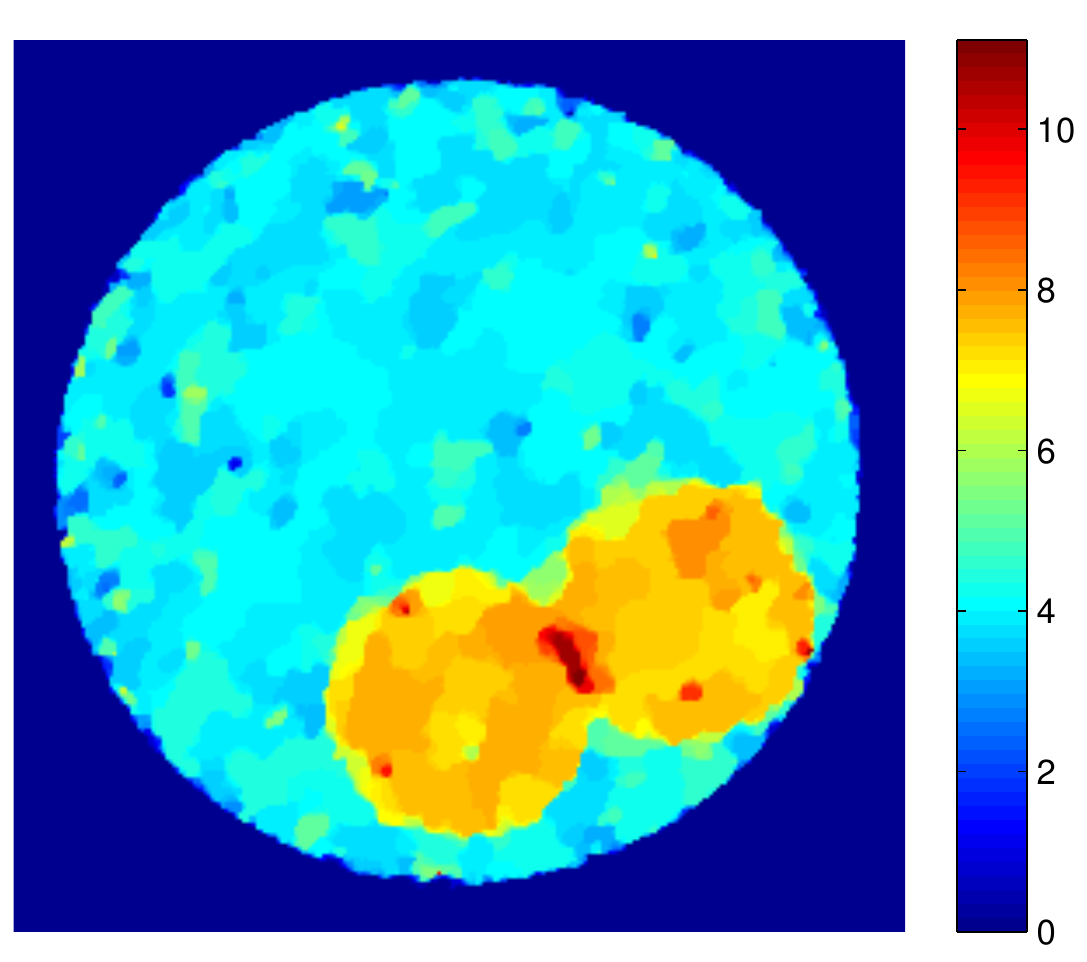}%
		\hfil
		\includegraphics[width=.32\textwidth]{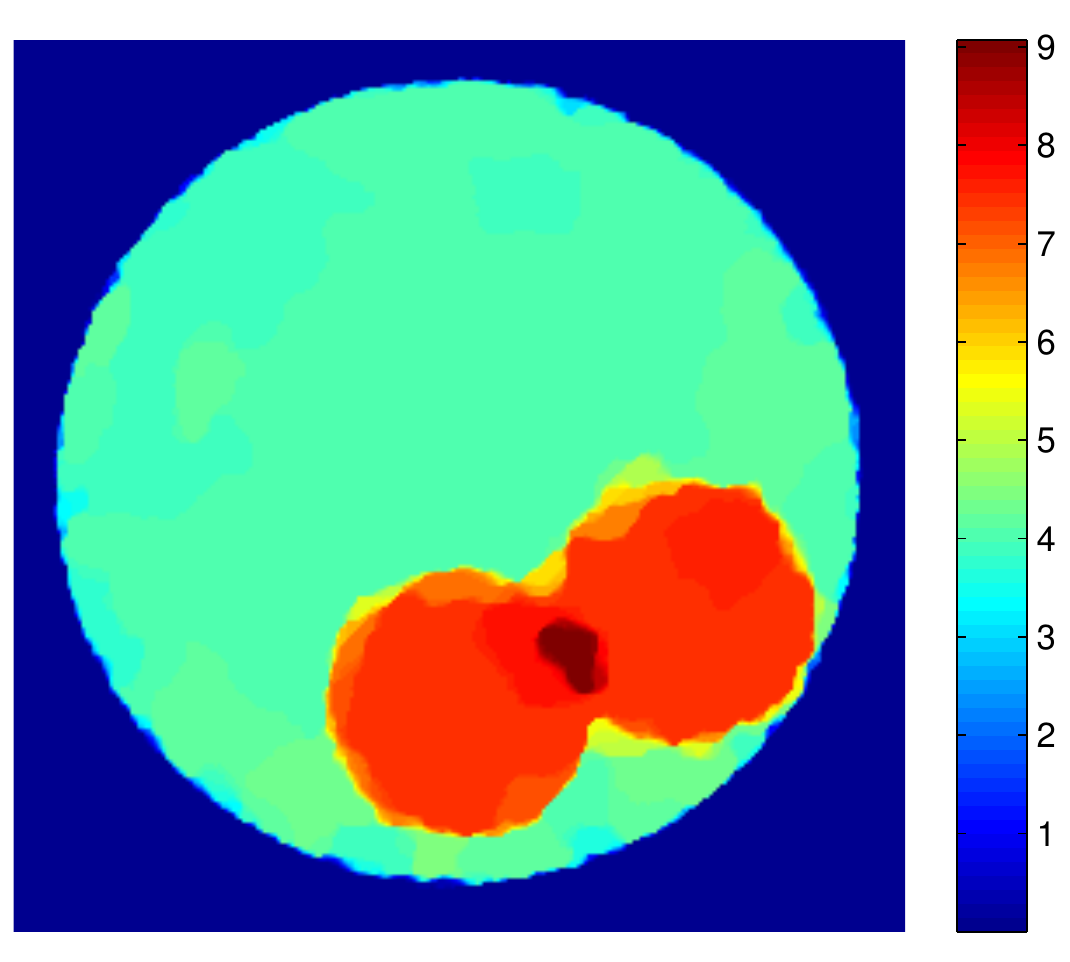}%
		\hfil
		\includegraphics[width=.32\textwidth]{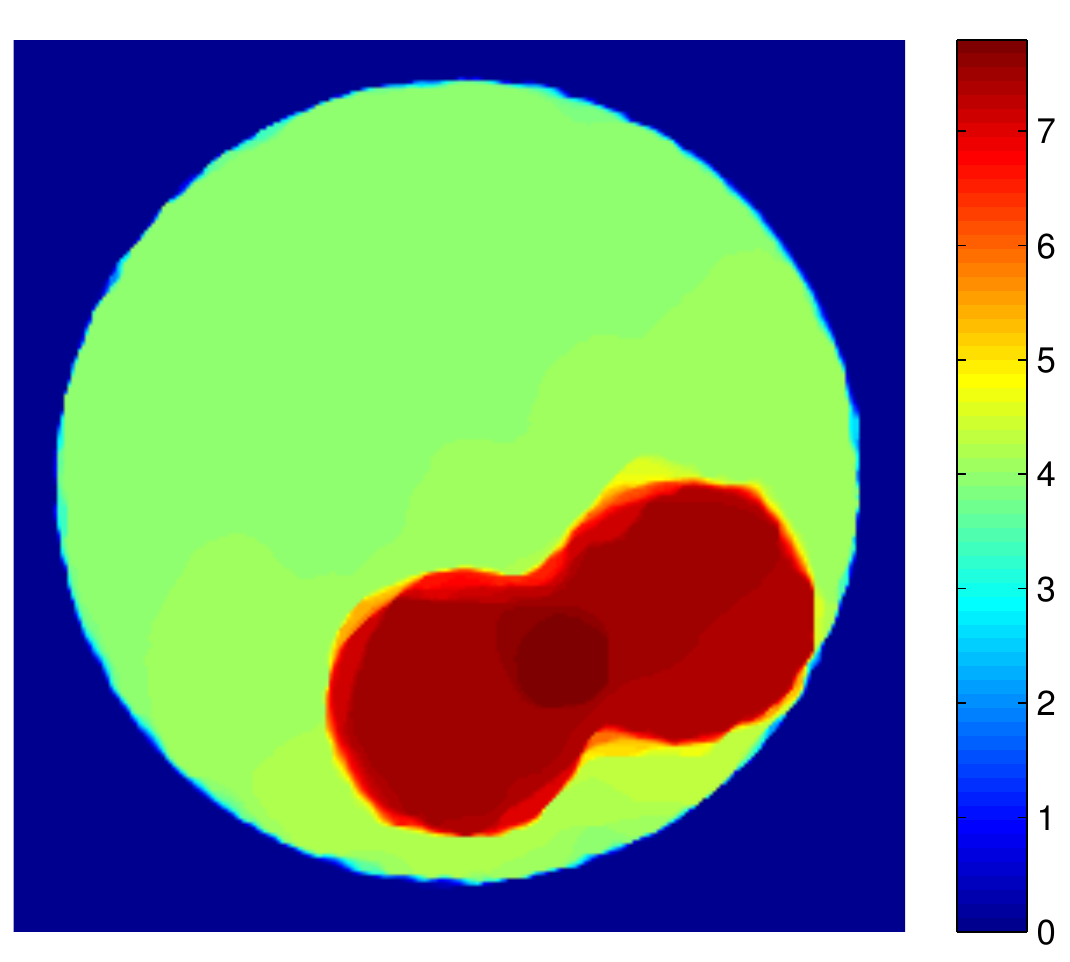}%
	\caption{The "ground truth" solutions for regularization parameter values $\alpha = 0.04$ (left), $\alpha = 0.08$ (middle), and $\alpha = 0.20$ (right).}%
	\label{fig:PETgroundtruth}%
\end{figure}

Figures \ref{fig:FBEMTV} - \ref{fig:ADMM} show the performance evaluation of algorithms plotting the propagation of the relative error \eqref{eq:relErr} 
in dependency on the number of iterations and CPU time in seconds. 
Since all algorithms have a specific set of unspecified parameters, different values of them are plotted to give a representative overall impression. 
The reason for showing the performance both in dependency on the number of iterations and CPU time is twofold: (1) in the presented case 
where the PET system matrix $K$ is pre-computed and thus available explicitly, the evaluation of forward and backward projections is nearly negligible 
and TV relevant computations have the most contribution to the run time such that the CPU time will be a good indicator for algorithm's performance; 
(2) in practically relevant cases where the forward and backward projections have to be computed in each iteration step implicitly and in general are computationally consuming, 
the number of iterations and thus the number of projection evaluations will be the crucial factor for algorithm's efficiency. In the following, we individually discuss 
the behavior of algorithms observed for the regularization parameter $\alpha = 0.08$ \eqref{eq:PETReconProblem} with the "ground truth" solution shown in Figure \ref{fig:PETgroundtruth}:
\begin{itemize}
	\item \emph{\textbf{FB-EM-TV(-Nes83):}} The evaluation of FB-EM-TV based algorithms is shown in Figure \ref{fig:FBEMTV}. The major observation for any $\delta$ in \eqref{eq:StopAHMOD} 
is that the inexact computations of TV proximal problems lead to a restrictive approximation of the "ground truth" solution where the approximation accuracy stagnates after 
a specific number of iterations depending on $\delta$. In addition, it can also be observed that the relative error \eqref{eq:relErr} becomes better with more accurate TV proximal solutions 
(i.e. smaller $\delta$) indicating that a decreasing sequence $\delta^{(r)}$ should be used to converge against the solution of \eqref{eq:PETReconProblem} (see, e.g., \cite{Schmidt2011,Villa2013} 
for convergence analysis of inexact proximal gradient algorithms). However, as indicated in \cite{Machart2012} 
and is shown in Figure \ref{fig:FBEMTV}, the choice of $\delta$ provides a trade-off between the approximation accuracy and computational 
cost such that the convergence rates proved in \cite{Schmidt2011,Villa2013} might be computationally not optimal. 
Another observation concerns the accelerated modification FB-EM-TV-Nes83. In Figure \ref{fig:FBEMTV} we can observe that the performance of FB-EM-TV can actually 
be improved by FB-EM-TV-Nes83 regarding the number of iterations but only for smaller values of $\delta$. One reason might be that using FB-EM-TV-Nes83 
we have seen in our experiments that the gradient descent parameter $0 < \eta \leq 1$ \cite{Sawatzky2013} in \eqref{eq:VarMetricStrategy} has to be chosen smaller 
with increased TV proximal accuracy (i.e. smaller $\delta$). Since in such cases the effective regularization parameter value in each TV proximal problem is $\eta \alpha$, 
a decreasing $\eta$ will result in poorer denoising properties increasing the inexactness of TV proximal operator. 
Recently, an (accelerated) inexact variable metric proximal gradient method was analyzed in \cite{CPR2013} providing a theoretical view on such a type of methods.

\begin{figure}%
	\centering
		\includegraphics[width=.49\textwidth]{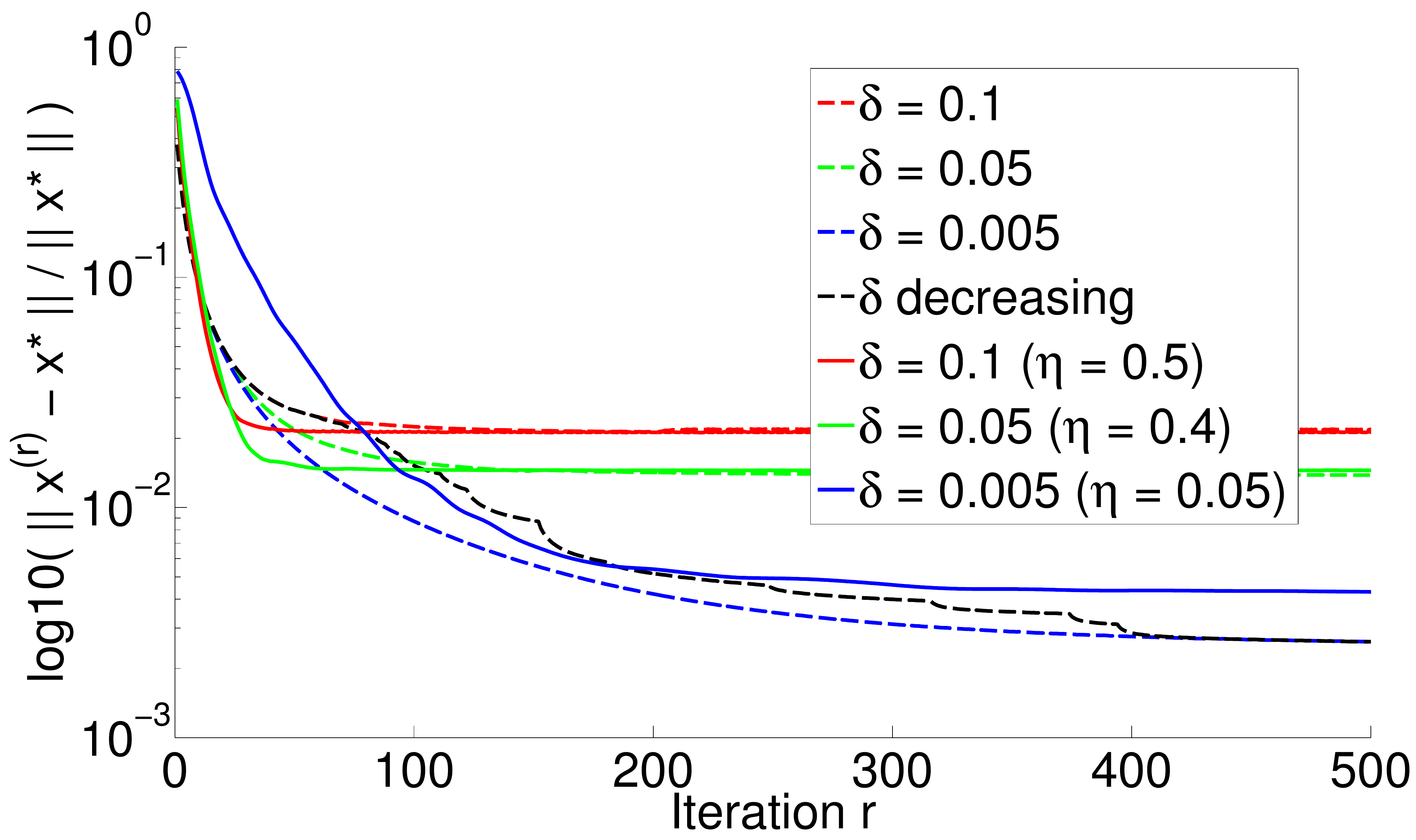}%
		\hfil
		\includegraphics[width=.49\textwidth]{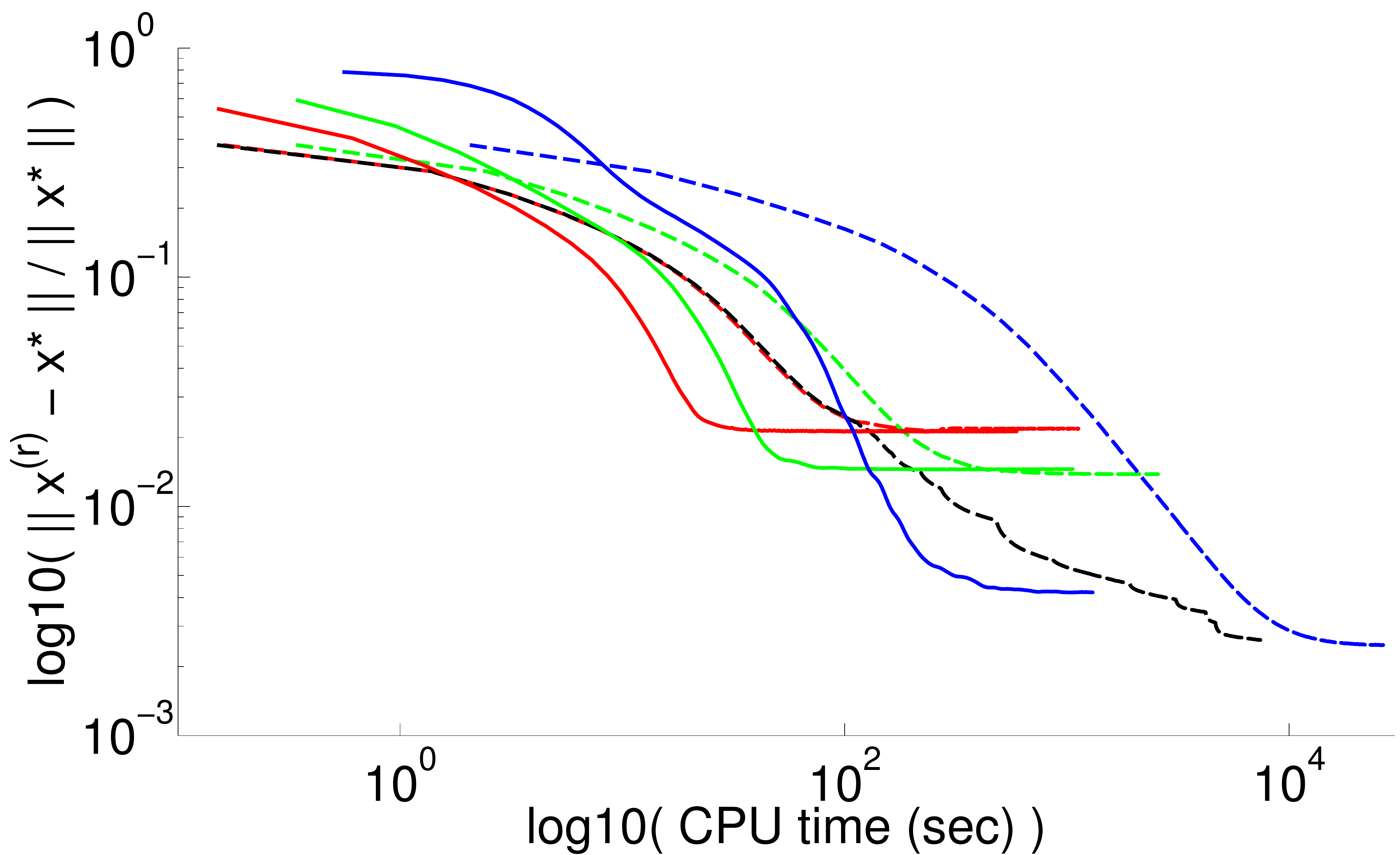}%
	\caption{Performance of FB-EM-TV (dashed lines) and FB-EM-TV-Nes83 (solid lines) for different accuracy thresholds $\delta$ \eqref{eq:StopAHMOD} within the TV proximal step. 
Evaluation of relative error \eqref{eq:relErr} 
is shown as a function of the number of iterations (left) and CPU time in seconds (right).}%
	\label{fig:FBEMTV}%
\end{figure}

	\item \emph{\textbf{(Precond-)CP-E:}} In Figure \ref{fig:CPE}, the algorithms CP-E and Precond-CP-E are evaluated. In contrast to FB-EM-TV-(Nes83), 
the approximated solution cannot be influenced by inexact computations such that a decaying behavior of relative error can be observed. 
The single parameter that affects the convergence rate is the dual steplength $\sigma$ and we observe in Figure \ref{fig:CPEa} 
that some values yield a fast initial convergence (see, e.g., $\sigma = 0.05$ and $\sigma = 0.1$), but are less suited to achieve fast asymptotic convergence and vice versa 
(see, e.g., $\sigma = 0.3$ and $\sigma = 0.5$). However, the plots in Figure \ref{fig:CPEb} indicate that $\sigma \in [0.2,0.3]$ may provide an acceptable trade-off between initial 
and asymptotic convergence in terms of the number of iterations and CPU time. 
Regarding the latter mentioned aspect we note that in case of CP-E the more natural setting of $\sigma$ would be $\sigma = \sqrt{\|\nabla\|^2 + \|K\|^2}$ what is approximately $0.29$ in our experiments providing acceptable trade-off between initial and asymptotic convergence. Finally, no acceleration was observed in case of Precond-CP-E algorithm due to the regular structure of linear operators $\nabla$ and $K$ in our experiments such that the performance is comparable to CP-E with $\sigma = 0.5$ (see Figure \ref{fig:CPEa}). 
	
\begin{figure}%
	\centering
	\begin{subfigure}{\textwidth}
		\includegraphics[width=.49\textwidth]{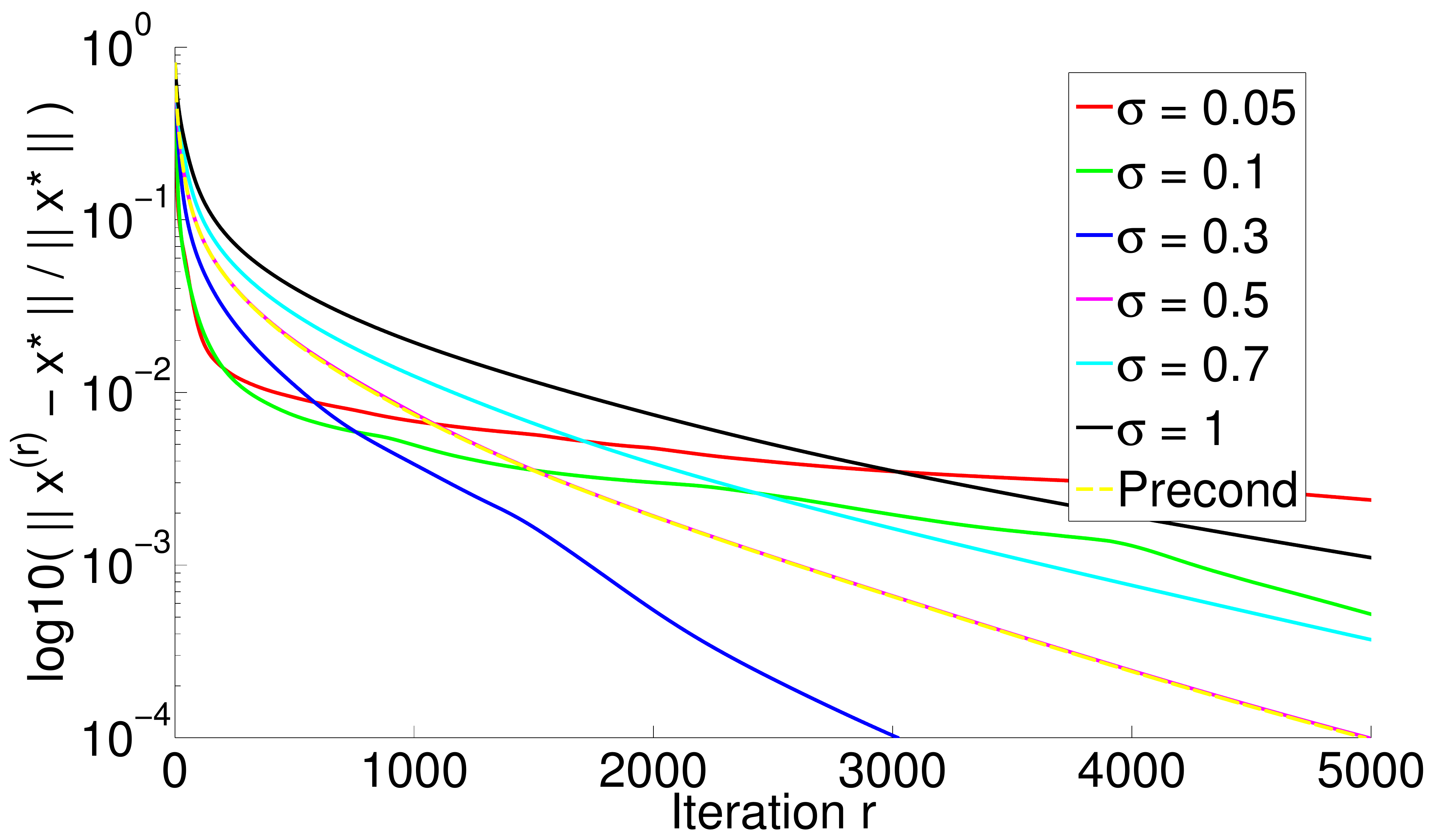}%
		\hfil
		\includegraphics[width=.49\textwidth]{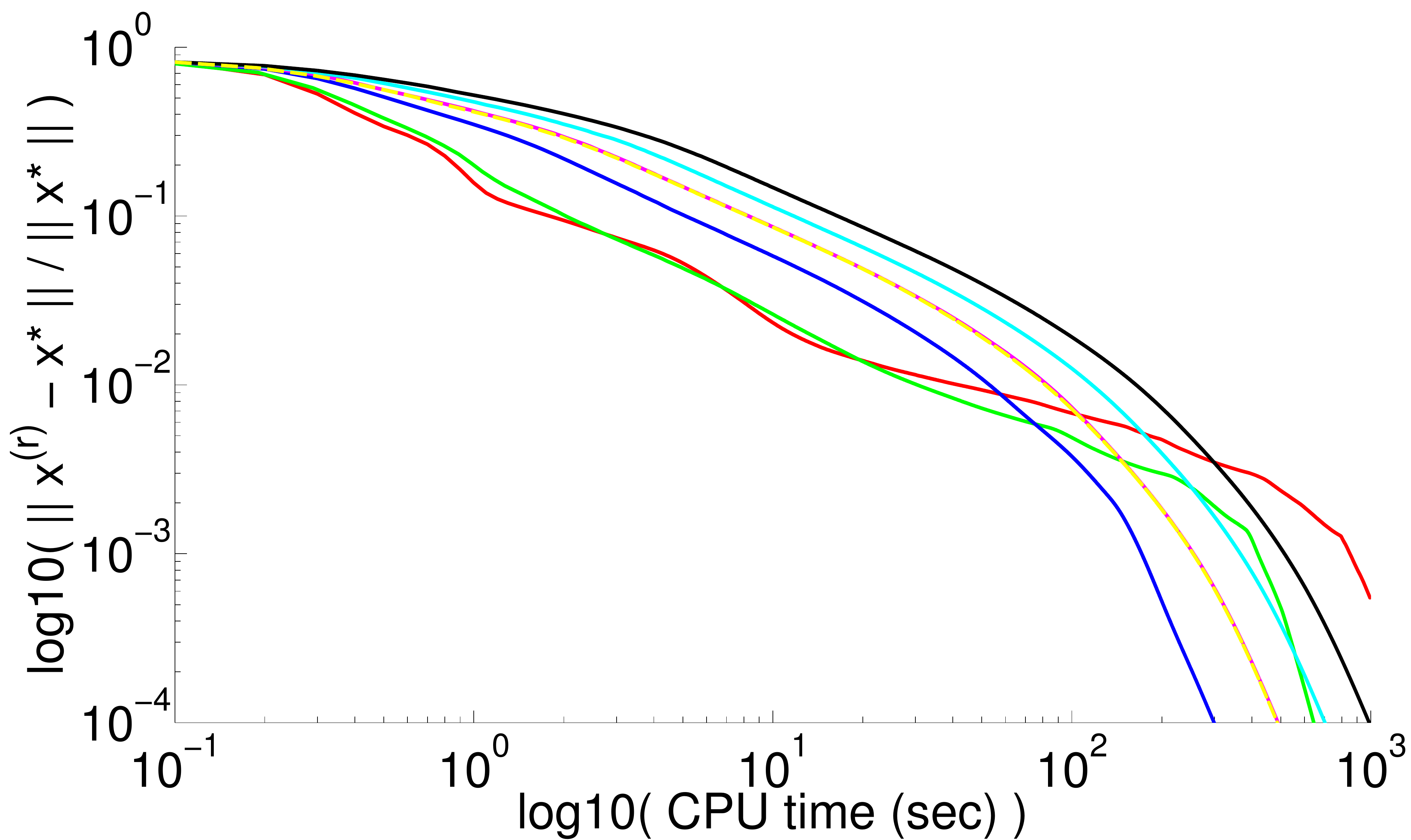}%
		\caption{Evaluation of relative error for fixed dual step sizes $\sigma$.}
		\label{fig:CPEa}%
	\end{subfigure}
	\\
	\begin{subfigure}{\textwidth}
		\includegraphics[width=.49\textwidth]{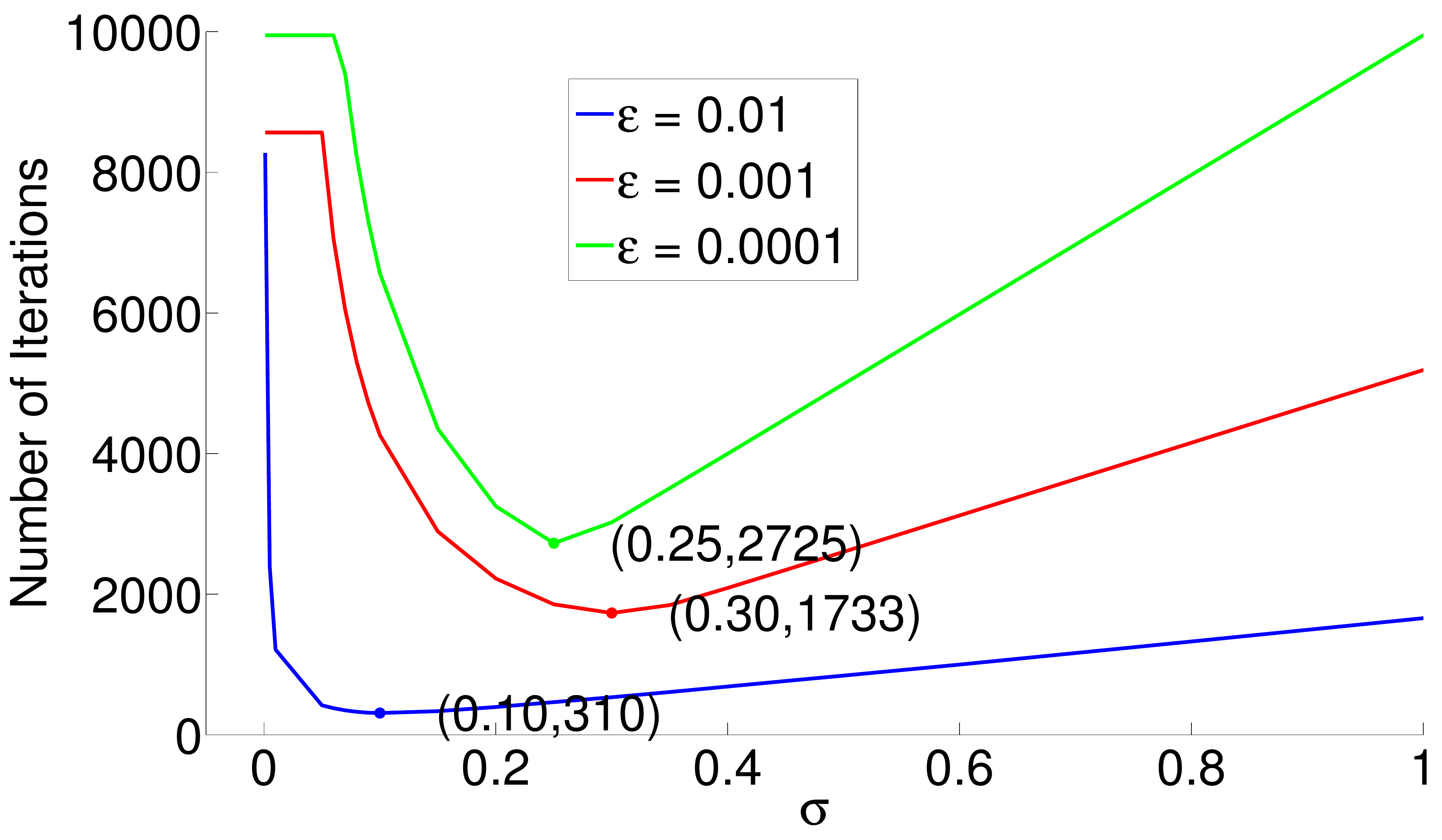}%
		\hfil
		\includegraphics[width=.49\textwidth]{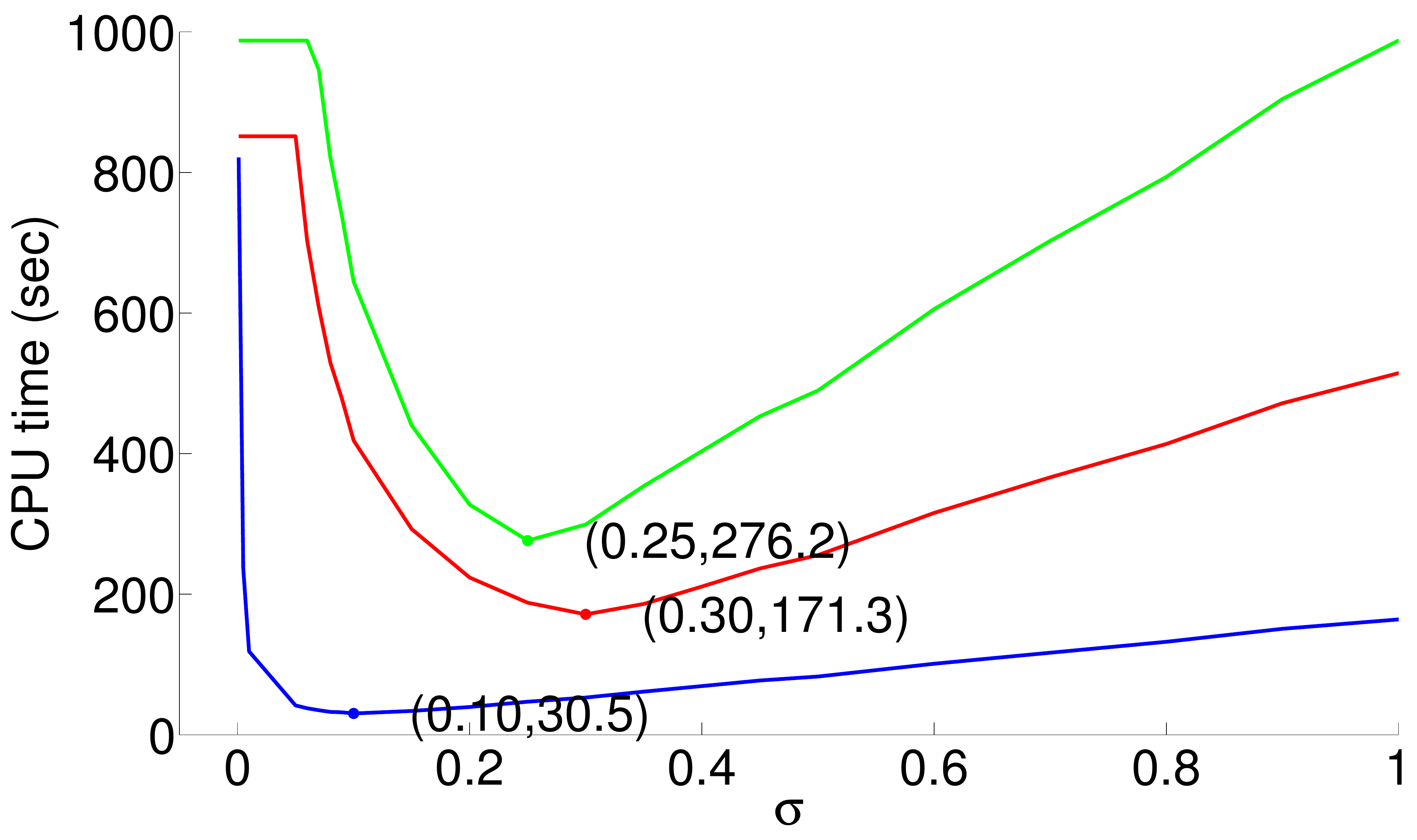}%
		\caption{Performance to get the relative error below the threshold $\epsilon$ as a function of dual step size $\sigma$.}
		\label{fig:CPEb}%
	\end{subfigure}
	\caption{Performance of Precond-CP-E (dashed lines in (\subref{fig:CPEa})) and CP-E (solid lines) for different dual step sizes $\sigma$. (\subref{fig:CPEa}) Evaluation 
of relative error as a function of the number of iterations (left) 
and CPU time in seconds (right). (\subref{fig:CPEb}) Required number of iterations (left) 
and CPU time (right) to get the relative error below a pre-defined threshold $\epsilon$ as a function of $\sigma$.}%
	\label{fig:CPE}%
\end{figure}

	\item \emph{\textbf{(Precond-)CP-SI:}} In Figures \ref{fig:CPSI} and \ref{fig:PrecondSPSI}, 
the evaluation of CP-SI and Precond-CP-SI is presented. Since a TV proximal operator has to be approximated in each iteration step, 
the same observations can be made as in case of FB-EM-TV that depending on $\delta$ the relative error stagnates after a specific number 
of iterations and that the choice of $\delta$ provides a trade-off between approximation accuracy and computational time (see Figure \ref{fig:CPSI} for Precond-CP-SI). 
In addition, since the performance of CP-SI not only depends on $\delta$ but also on the dual steplength $\sigma$, the evaluation of CP-SI for different values of $\sigma$ 
and two stopping values $\delta$ is shown in Figure \ref{fig:PrecondSPSI}. The main observation is that for smaller $\sigma$ a better initial convergence can be achieved in terms of the number 
of iterations but results in less efficient performance regarding the CPU time. The reason is that the effective regularization parameter within the TV proximal problem is $\tau\alpha$ 
(see \eqref{eq:VarMetricStrategy}) with $\tau = (\sigma \|K\|^2)^{-1}$ and a decreasing $\sigma$ leads to an increasing TV denoising effort. Thus, in practically relevant cases, 
$\sigma$ should be chosen optimally in a way balancing the required number of iterations and TV proximal computation.

\begin{figure}%
	\centering
		\includegraphics[width=.49\textwidth]{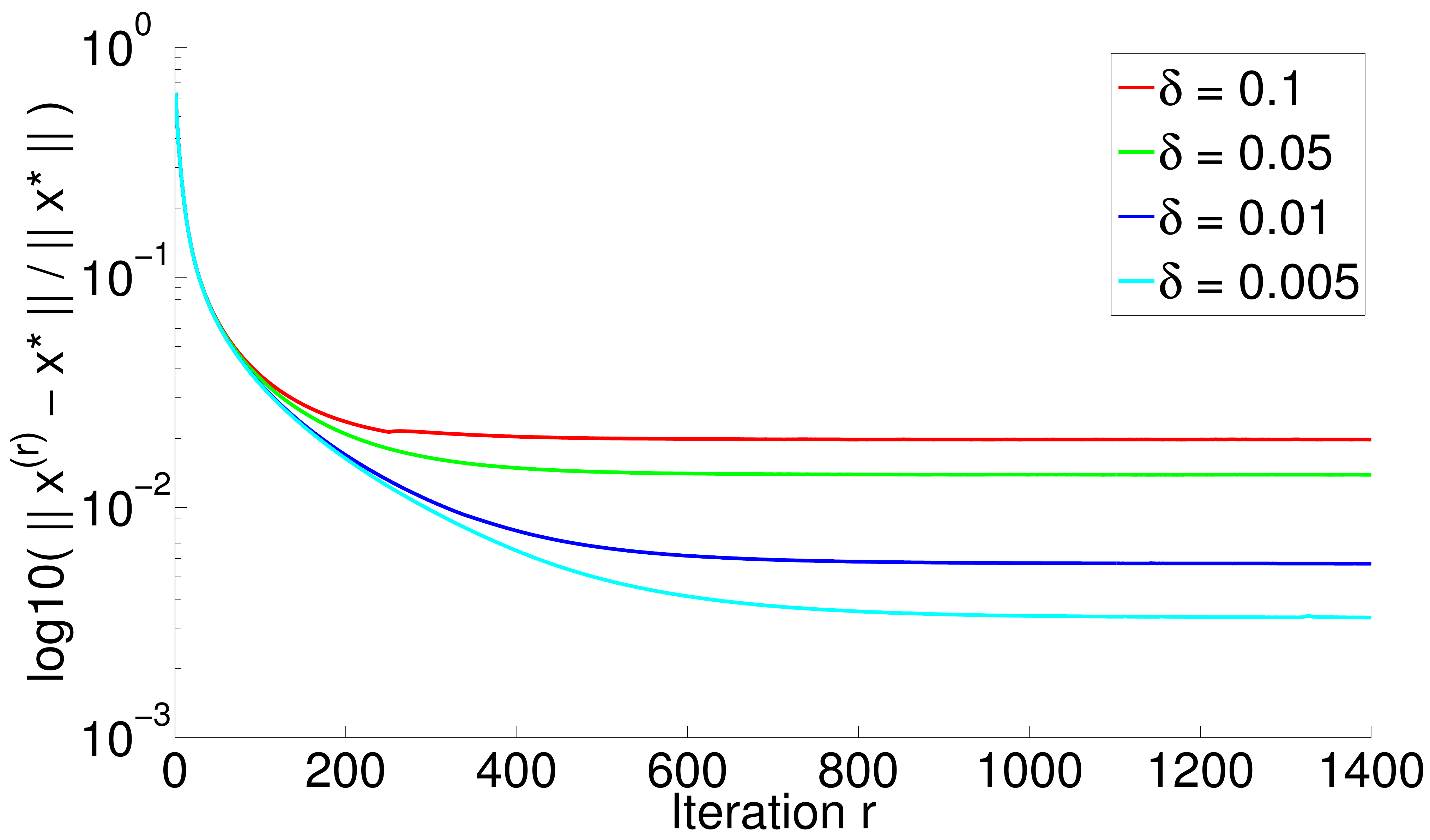}%
		\hfil
		\includegraphics[width=.49\textwidth]{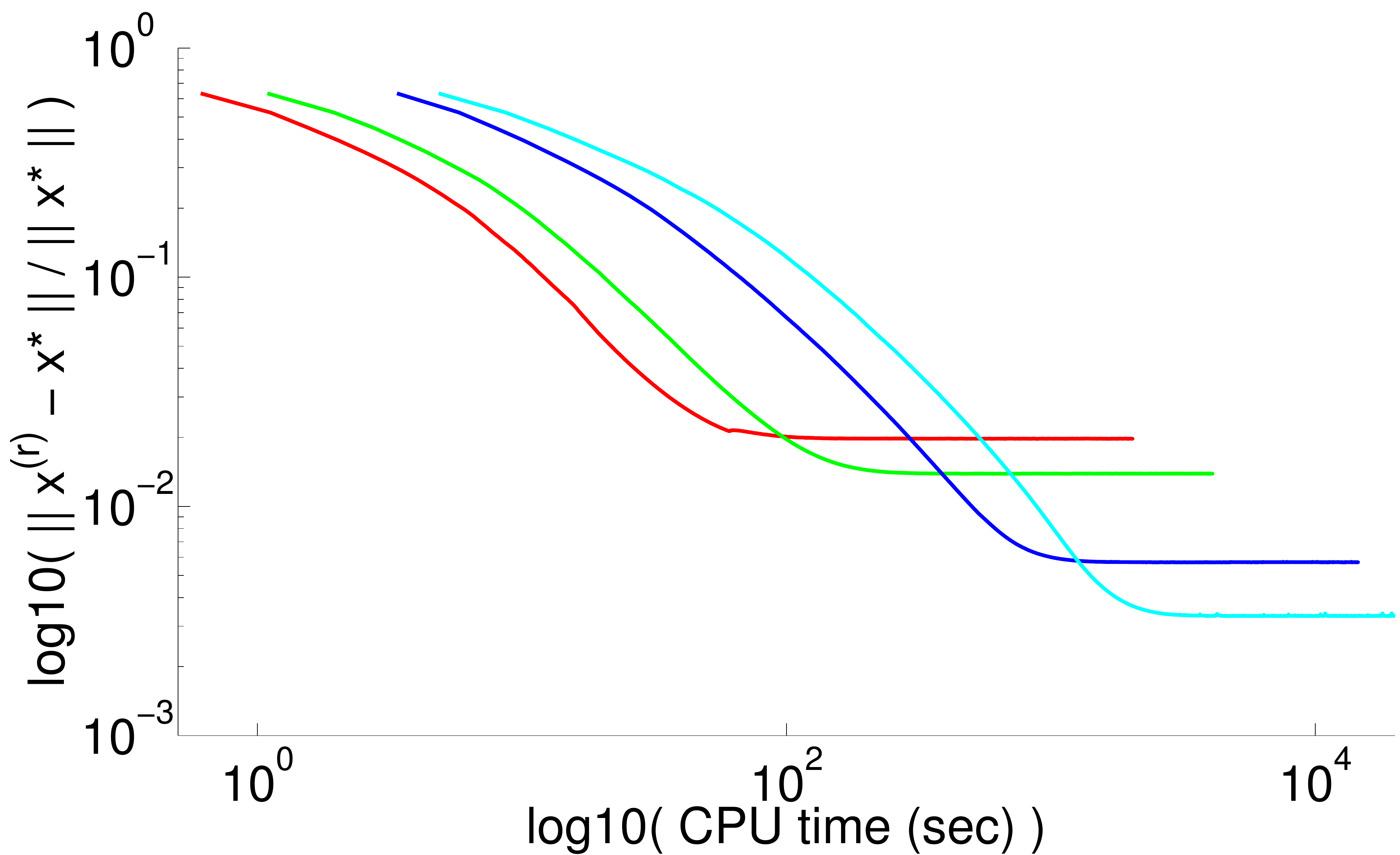}%
	\caption{Performance of Precond-CP-SI for different accuracy thresholds $\delta$ \eqref{eq:StopAHMOD} within the TV proximal step. 
Evaluation of relative error as a function of number of iterations (left) and CPU time in seconds (right).}%
	\label{fig:CPSI}%
\end{figure}

\begin{figure}%
	\centering
	\begin{subfigure}{\textwidth}
		\includegraphics[width=.49\textwidth]{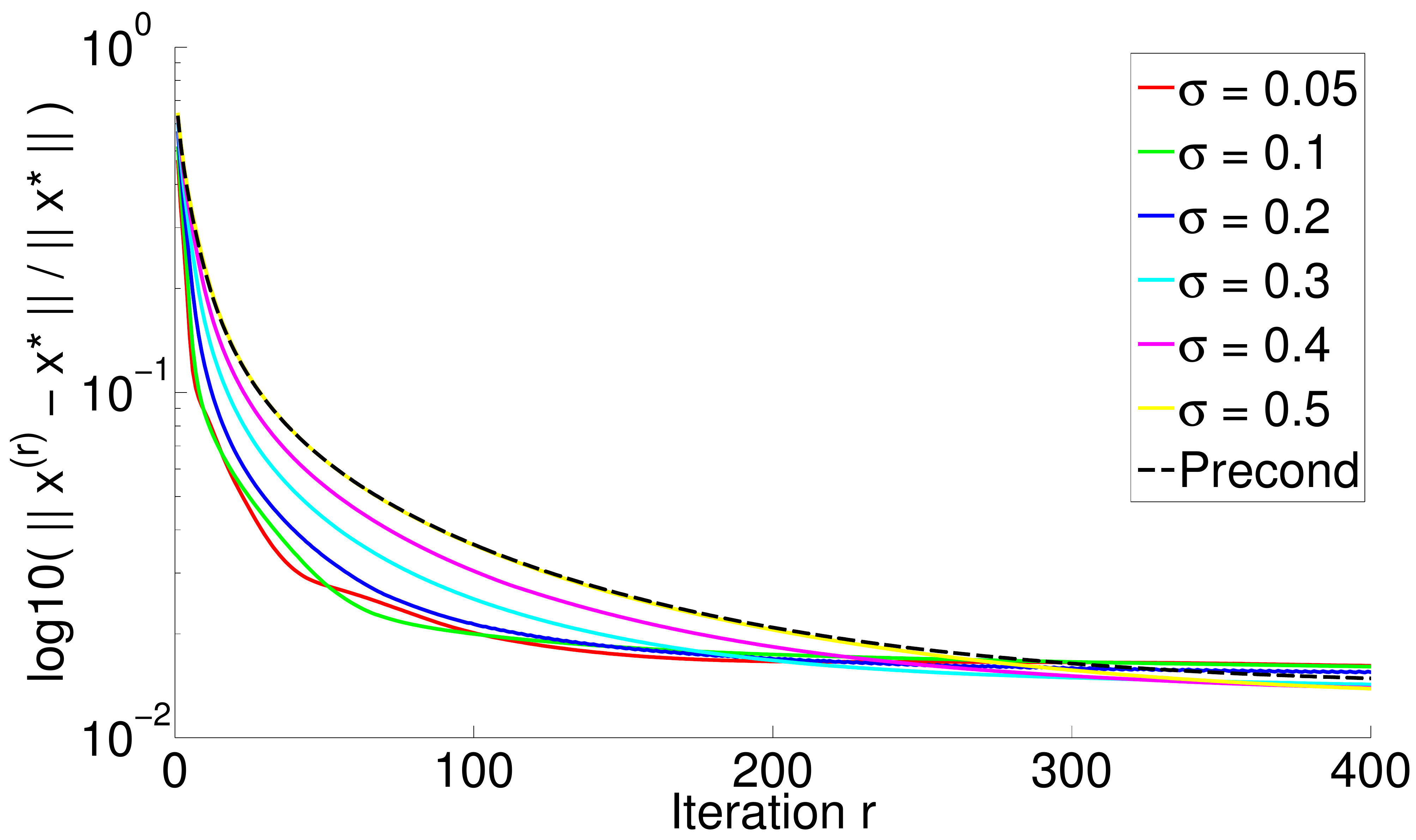}%
		\hfil
		\includegraphics[width=.49\textwidth]{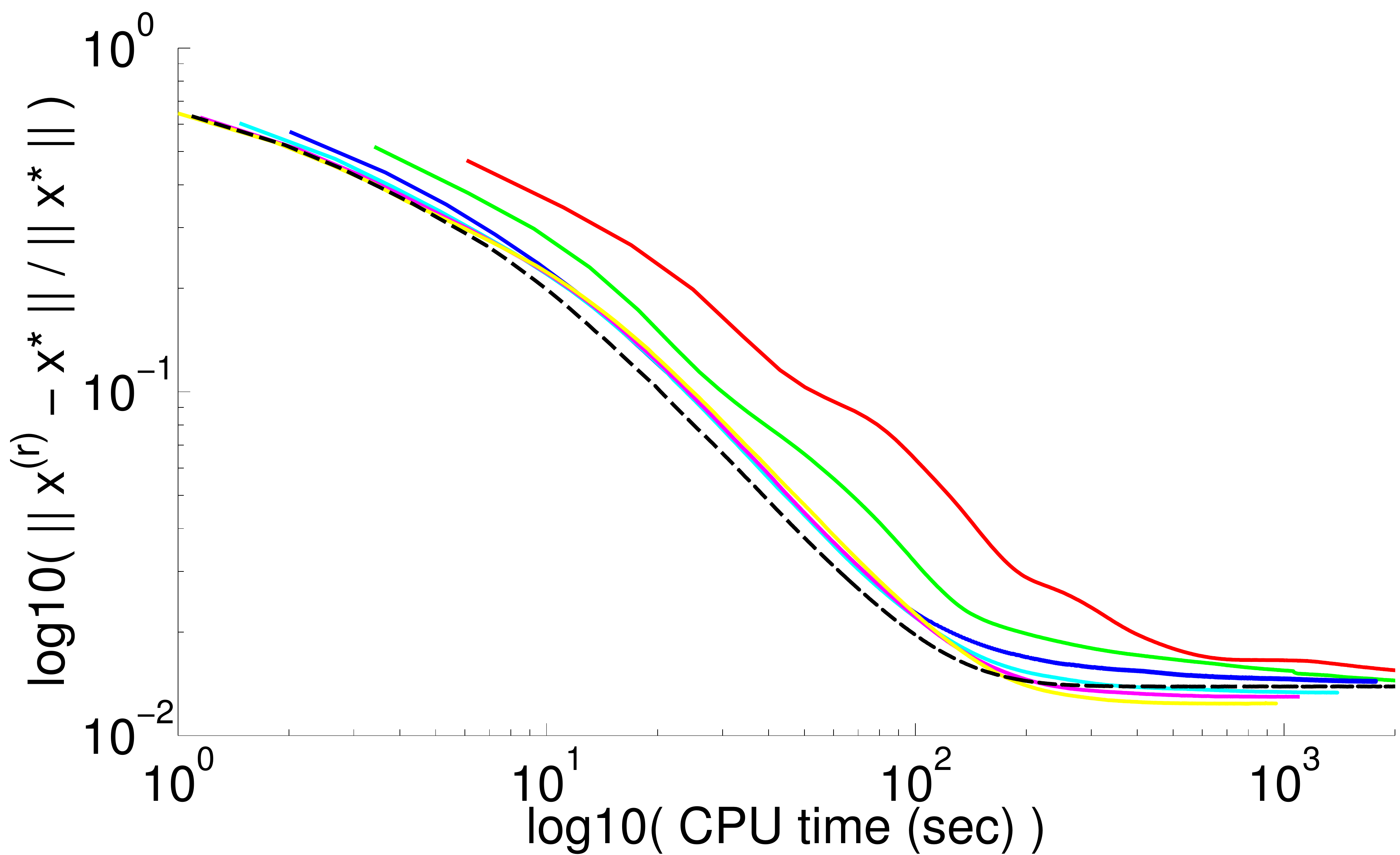}%
		\caption{Performance with TV proximal accuracy $\delta = 0.05$ \eqref{eq:StopAHMOD}.}
		\label{fig:PrecondSPSIa}
	\end{subfigure}
	\\
	\begin{subfigure}{\textwidth}
		\includegraphics[width=.49\textwidth]{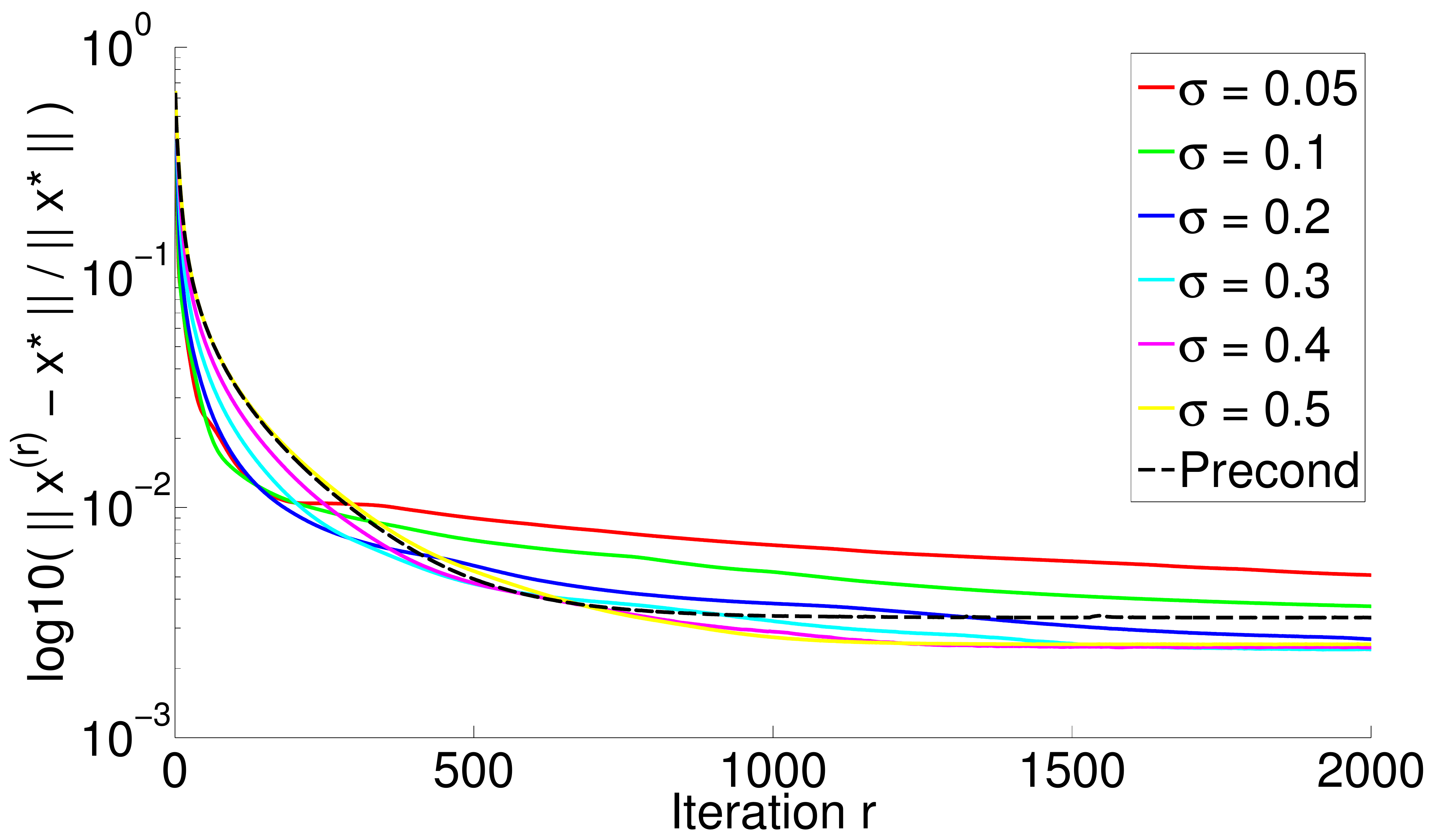}%
		\hfil
		\includegraphics[width=.49\textwidth]{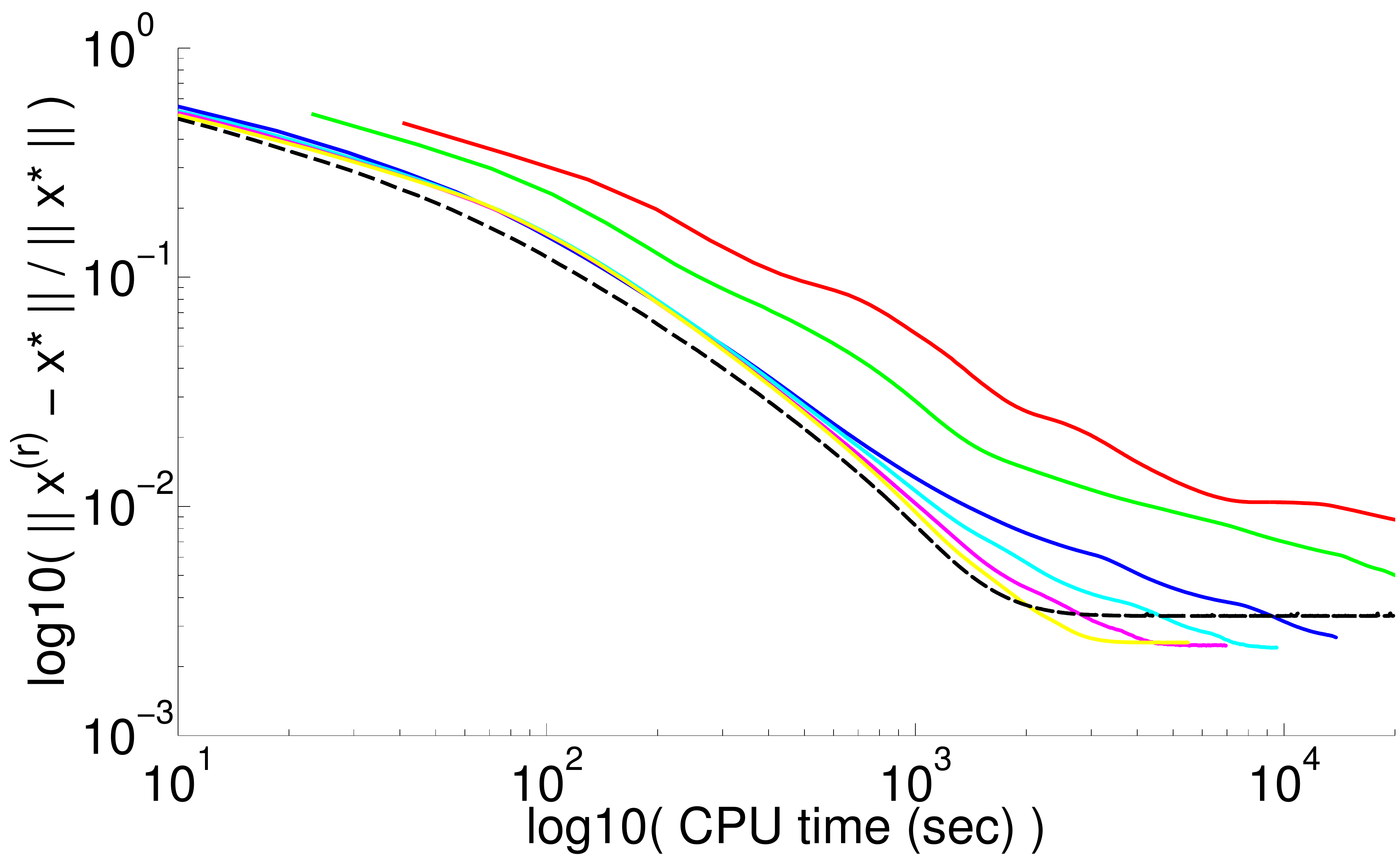}%
		\caption{Performance with TV proximal accuracy $\delta = 0.005$ \eqref{eq:StopAHMOD}.}
		\label{fig:PrecondSPSIb}
	\end{subfigure}
	\caption{Performance of Precond-CP-SI (dashed lines) and CP-SI (solid lines) for different dual step sizes $\sigma$.
Evaluation of relative error as a function of number of iterations (left) and CPU time in seconds (right) for accuracy thresholds $\delta = 0.05$ (\subref{fig:PrecondSPSIa}) 
and $\delta = 0.005$ (\subref{fig:PrecondSPSIb}) within the TV proximal problem.}%
	\label{fig:PrecondSPSI}%
\end{figure}

	\item \emph{\textbf{PIDSplit+:}} In Figure \ref{fig:ADMM} the performance of PIDSplit+ is shown. 
It is well evaluated that the convergence of ADMM based algorithms is strongly dependent on the augmented Lagrangian parameter $\gamma$ \eqref{Lagrangian_augmented} 
and that some values yield a fast initial convergence but are less suited to achieve a fast asymptotic convergence and vice versa. This behavior can also be observed 
in Figure \ref{fig:ADMM} (see $\gamma = 30$ in upper row).


\begin{figure}
	\centering
	\includegraphics[width=.49\textwidth]{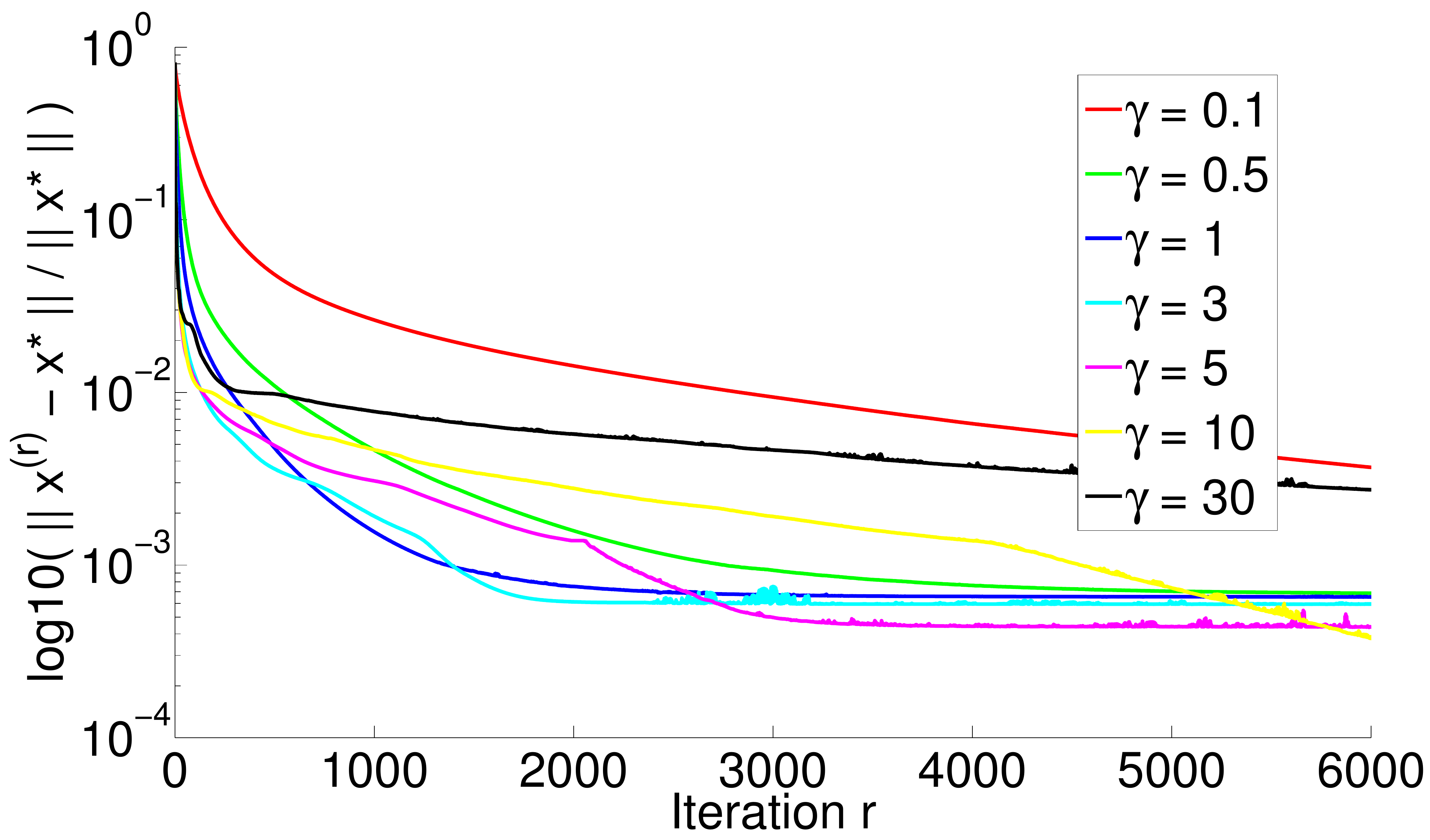}%
	\hfil
	\includegraphics[width=.49\textwidth]{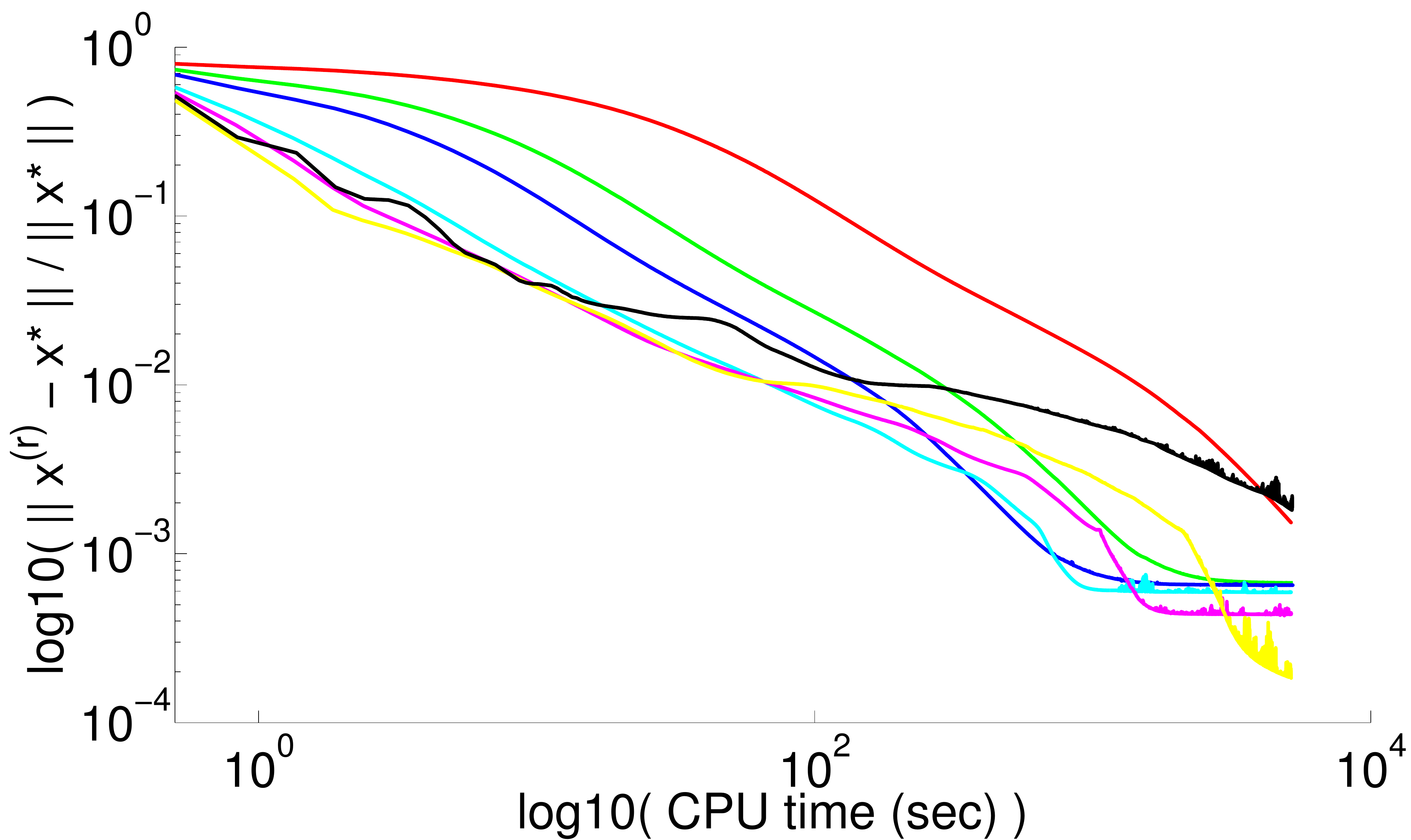}%
	\caption{Performance of PIDSplit+ for fixed augmented Lagrangian penalty parameters $\gamma$ \eqref{Lagrangian_augmented}. Evaluation of relative error as a function of number of iterations (left) and CPU time in seconds (right).}%
	\label{fig:ADMM}%
\end{figure}

\end{itemize}

Finally, to get a feeling how the algorithms perform against each other, the required CPU time and number of projection evaluations 
to get the relative error \eqref{eq:relErr} below a pre-defined threshold are shown in Table \ref{tab:PET1} for two different values of $\epsilon$. 
The following observations can be made:
\begin{itemize}
	\item The FB-EM-TV based algorithms are competitive in terms of required number of projection evaluations but have a higher CPU time due to the computation of TV proximal operators, 
in particular the CPU time strongly grows with decreasing $\epsilon$ since TV proximal problems have to be approximated with increased accuracy. However, in our experiments, a fixed $\delta$ 
was used in each TV denoising step and thus the performance can be improved utilizing the fact that a rough accuracy is sufficient at the beginning of the iteration sequence without 
influencing the performance regarding the number of projector evaluations negatively (cf. Figure \ref{fig:FBEMTV}). Thus, a proper strategy to iteratively decrease $\delta$ in \eqref{eq:StopAHMOD}  can strongly improve the performance of FB-EM-TV based algorithms.
	\item The CP-E algorithm is optimal in our experiments in terms of CPU time since the TV regularization is computed by the shrinkage operator 
and thus is simply to evaluate. However, this algorithm needs almost the highest number of projection evaluations that will result in a slow algorithm in practically relevant cases 
where the projector evaluations are highly computationally expansive.
	\item The PIDSplit+ algorithm is slightly poorer in terms of CPU time than CP-E but required a smaller number of projector evaluations. 
However, we remind that this performance may probably be improved since two PCG iterations were used in our experiments and thus 
two forward and backward projector evaluations are required in each iteration step of PIDSplit+ method. Thus, if only one PCG step is used, 
the CPU time and number of projector evaluations can be decreased leading to a better performing algorithm. However, in the latter case, 
the total number of iteration steps might be increased since a poorer approximation of \eqref{eq:PIDSplitLSE} will be performed if only one PCG step is used. 
Another opportunity to improve the performance of PIDSplit+ algorithm is to use the proximal ADMM strategy described in Section \ref{subsec:proximal_admm}, 
namely, to remove $K^\tT K$ from \eqref{eq:PIDSplitLSE}. That will result in only a single evaluation of forward and backward projectors in each iteration step 
but may lead to an increased number of total number of algorithm iterations.
\end{itemize}

\begin{table}
	\caption{Performance evaluation of algorithms described above for $\alpha = 0.08$ (see $u^*_\alpha$ in Figure \ref{fig:PETgroundtruth} (middle)). 
The table displays the CPU time in seconds and required number of forward and backward projector evaluations ($K/K^\tT$) to get the relative error
\eqref{eq:relErr} below the error tolerance $\epsilon$. For each algorithm the best performance regarding the CPU time and $K/K^\tT$ evaluations are shown where $\shortparallel$ means that the value coincides with the value directly above.}
	\centering
		\begin{tabular}{ll*{2}{c}p{1em}*{2}{c}}
			\hline\noalign{\smallskip}
				& & \multicolumn{2}{c}{$\epsilon = 0.05$} & & \multicolumn{2}{c}{$\epsilon = 0.005$} \\
				\noalign{\smallskip}\cline{3-4}\cline{6-7}\noalign{\smallskip}
				& & $K/K^\tT$ & CPU & & $K/K^\tT$ & CPU \\
			\noalign{\smallskip}\hline\noalign{\smallskip}
			FB-EM-TV & (best $K/K^\tT$) \hspace{0.5em} & 20 & 40.55 & & \textbf{168} & 4999.71 \\ \noalign{\smallskip}
			\quad $\shortparallel$ & (best CPU) & $\shortparallel$ & $\shortparallel$ & & 230 & 3415.74 \\ \noalign{\smallskip}
			FB-EM-TV-Nes83 & (best $K/K^\tT$) & \textbf{15} & 14.68 & & 231 & 308.74 \\ \noalign{\smallskip}
			\quad $\shortparallel$ & (best CPU) & $\shortparallel$ & $\shortparallel$ & & $\shortparallel$ & $\shortparallel$ \\ \noalign{\smallskip}
			CP-E & (best $K/K^\tT$) & 48 & \textbf{4.79} & & 696 & \textbf{69.86} \\ \noalign{\smallskip}
			\quad $\shortparallel$ & (best CPU) & $\shortparallel$ & $\shortparallel$ & & $\shortparallel$ & $\shortparallel$ \\ \noalign{\smallskip}
			CP-SI & (best $K/K^\tT$) & 22 & 198.07 & & 456 & 1427.71 \\ \noalign{\smallskip}
			\quad $\shortparallel$ & (best CPU) & 25 & 23.73 & & 780 & 1284.56 \\ \noalign{\smallskip}
			PIDSplit+ & (best $K/K^\tT$) & 30 & 7.51 & & 698 & 179.77 \\ \noalign{\smallskip}
			\quad $\shortparallel$ & (best CPU) & $\shortparallel$ & $\shortparallel$ & & $\shortparallel$ & $\shortparallel$ 
\\		\hline
		\end{tabular}
	\label{tab:PET1}
\end{table}

Finally, to study the algorithm's stability regarding the choice of regularization parameter $\alpha$, we have run the algorithms for two additional values of $\alpha$ 
using the parameters shown the best performance in Table \ref{tab:PET1}. The additional penalty parameters include a slightly under-smoothed and over-smoothed result 
respectively as shown in Figure \ref{fig:PETgroundtruth} and the evaluation results are shown in Tables \ref{tab:PET2} and \ref{tab:PET3}. In the following we describe the major observations:

\begin{itemize}
	\item The FB-EM-TV method has the best efficiency in terms of projector evaluations, independently from the penalty parameter $\alpha$, but has the disadvantage 
of solving a TV proximal problem in each iteration step which get harder to solve with increasing smoothing level (i.e. larger $\alpha$) leading to a negative computational time. 
The latter observation holds also for the CP-SI algorithm. In case of a rough approximation accuracy (see Table \ref{tab:PET2}), the FB-EM-TV-Nes83 scheme is able to improve 
the overall performance, respectively at least the computational time for higher accuracy in Table \ref{tab:PET3}, but here the damping parameter $\eta$ in \eqref{eq:VarMetricStrategy} 
has to be chosen carefully to ensure the convergence (cf. Table \ref{tab:PET2} and \ref{tab:PET3} in case of $\alpha = 0.2$). Additionally based on Table \ref{tab:PET1}, 
a proper choice of $\eta$ is not only dependent on $\alpha$ but also on the inner accuracy of TV proximal problems.
	\item In contrast to FB-EM-TV and CP-SI, the remaining algorithms provide a superior computational time due to the solution of TV related steps 
by the shrinkage formula but show a strongly increased requirements on projector evaluations across all penalty parameters $\alpha$. In addition, the performance 
of these algorithms is strongly dependent on the proper setting of free parameters ($\sigma$ in case of CP-E and $\gamma$ in PIDSplit+) which unfortunately 
are able to achieve only a fast initial convergence or a fast asymptotic convergence. Thus different parameter settings of $\sigma$ and $\gamma$ were used in Tables \ref{tab:PET2} and \ref{tab:PET3}.
\end{itemize}

\begin{table}
	\caption{Performance evaluation for different values of $\alpha$ (see Figure \ref{fig:PETgroundtruth}). The table displays the CPU time in seconds 
and required number of forward and backward projector evaluations ($K/K^\tT$) to get the relative error \eqref{eq:relErr} below the error tolerance $\epsilon = 0.05$. 
For each $\alpha$, the algorithms were run using the following parameters: FB-EM-TV ($\delta = 0.1$), FB-EM-TV-Nes83 ($\delta = 0.1$, $\eta = 0.5$), CP-E ($\sigma = 0.07$), 
CP-SI ($\delta = 0.1$, $\sigma = 0.05$), PIDSplit+ ($\gamma = 10$), which were chosen based on the "best" performance 
regarding $K/K^\tT$ for $\epsilon = 0.05$ in Table \ref{tab:PET1}.} 
	\centering
		\begin{tabular}{l*{2}{c}p{1em}*{2}{c}p{1em}*{2}{c}}
			\hline\noalign{\smallskip}
				& \multicolumn{2}{c}{$\alpha = 0.04$} & & \multicolumn{2}{c}{$\alpha = 0.08$} & & \multicolumn{2}{c}{$\alpha = 0.2$} \\
				\noalign{\smallskip}\cline{2-3}\cline{5-6}\cline{8-9}\noalign{\smallskip}
				& $K/K^\tT$ & CPU & & $K/K^\tT$ & CPU & & $K/K^\tT$ & CPU \\
			\noalign{\smallskip}\hline\noalign{\smallskip}
			FB-EM-TV & 28 & 16.53 & & 20 & 40.55 & & \textbf{19} & 105.37 \\ \noalign{\smallskip}
			FB-EM-TV-Nes83 & \textbf{17} & \textbf{5.26} & & \textbf{15} & 14.68 & & - & - \\ \noalign{\smallskip}
			CP-E & 61 & 6.02 & & 48 & \textbf{4.79} & & 51 & \textbf{5.09} \\ \noalign{\smallskip}
			CP-SI & - & - & & 25 & 23.73 & & 21 & 133.86 \\ \noalign{\smallskip}
			PIDSplit+ & 32 & 8.08 & & 30 & 7.51 & & 38 & 9.7 \\ \noalign{\smallskip}
		\hline
		\end{tabular}
	\label{tab:PET2}
\end{table}

\begin{table}
	\caption{Performance evaluation for different values of $\alpha$ (see Figure \ref{fig:PETgroundtruth}) as in Table \ref{tab:PET2} 
but for $\epsilon = 0.005$ and using the following parameters: FB-EM-TV ($\delta = 0.005$), FB-EM-TV-Nes83 ($\delta = 0.005$, $\eta = 0.05$), CP-E ($\sigma = 0.2$), 
CP-SI ($\delta = 0.005$, $\sigma = 0.3$), PIDSplit+ ($\gamma = 3$).} 
	\centering
		\begin{tabular}{l*{2}{c}p{1em}*{2}{c}p{1em}*{2}{c}}
			\hline\noalign{\smallskip}
				& \multicolumn{2}{c}{$\alpha = 0.04$} & & \multicolumn{2}{c}{$\alpha = 0.08$} & & \multicolumn{2}{c}{$\alpha = 0.2$} \\
				\noalign{\smallskip}\cline{2-3}\cline{5-6}\cline{8-9}\noalign{\smallskip}
				& $K/K^\tT$ & CPU & & $K/K^\tT$ & CPU & & $K/K^\tT$ & CPU \\
			\noalign{\smallskip}\hline\noalign{\smallskip}
			FB-EM-TV & \textbf{276} & 2452.14 & & \textbf{168} & 4999.71 & & \textbf{175} & 12612.7 \\ \noalign{\smallskip}
			FB-EM-TV-Nes83 & 512 & 222.98 & & 231 & 308.74 & & - & - \\ \noalign{\smallskip}
			CP-E & 962 & \textbf{94.57} & & 696 & \textbf{69.86} & & 658 & \textbf{65.42} \\ \noalign{\smallskip}
			CP-SI & 565 & 1117.12 & & 456 & 1427.71 & & 561 & 7470.18 \\ \noalign{\smallskip}
			PIDSplit+ & 932 & 239.35 & & 698 & 179.77 & & 610 & 158.94 \\ \noalign{\smallskip}
		\hline
		\end{tabular}
	\label{tab:PET3}
\end{table}

\subsection{Spectral X-Ray CT} \label{sec:spectralCT}
Conventional X-ray CT is based on recording changes in the X-ray intensity due to attenuation of X-ray beams traversing the scanned object and has been applied 
in clinical practice for decades. However, the transmitted X-rays carry more information than just intensity changes since the attenuation 
of an X-ray depends strongly on its energy \cite{Alvarez1976,Knoll2000}. It is well understood that the transmitted energy spectrum contains valuable information about 
the structure and material composition of the imaged object and can be utilized to better distinguish different types of absorbing material, such as varying tissue types 
or contrast agents. But the detectors employing in traditional CT systems provide an integral measure of absorption over the transmitted energy spectrum 
and thus eliminate spectral information \cite{Cammin2012,Schirra2014}. 
Even so, spectral information can be obtained by using different input spectra \cite{Flohr2006,Zou2008} or using the concept of dual-layer (integrating) detectors \cite{Carmi2005}.
This has been limited the practical usefulness of energy-resolving imaging,
 also referred to as spectral CT, to dual energy systems. Recent advances in detector technology towards 
binned photon-counting detectors have enabled a new generation of detectors that can measure and analyze incident photons individually \cite{Cammin2012} 
providing the availability of more than two spectral measurements. This development has led to a new imaging method named K-edge imaging \cite{Kruger1977} 
that can be used to selectively and quantitatively image contrast agents loaded with K-edge materials \cite{Feuerlein2008,Pan2012}. For a compact overview 
on technical and practical aspects of spectral CT we refer to \cite{Cammin2012,Schirra2014}.


Two strategies have been proposed to reconstruct material specific images from spectral CT projection data and we refer to \cite{Schirra2014} for a compact overview. 
Either of them is a projection-based material decomposition with a subsequent image reconstruction. This means that in the first step, estimates of material-decomposed 
sinograms are computed from the energy-resolved measurements, and in the second step, material images are reconstructed from the decomposed material sinograms. 
A possible decomposition method to estimate the material sinograms $f_l$, $l= 1,\ldots,L$, from the acquired data is a maximum-likelihood estimator assuming 
a Poisson noise distribution \cite{Roessl2007}, where $L$ is the number of materials considered. An accepted noise model for line integrals $f_l$ 
is a multivariate Gaussian distribution \cite{Schirra2014,Schirra2013} leading to a penalized weighted least squares (PWLS) estimator to reconstruct material images $u_l$:
\begin{equation} \label{eq:PWLSSpectralCT}
	\frac12 \| f - (I_L \otimes K) u \|^2_{\Sigma^{-1}} + \alpha R(u) \rightarrow \min_u , \quad \alpha > 0,
\end{equation}
where $f = (f_1^\tT, \ldots, f_L^\tT)^\tT$, $u = (u_1^\tT, \ldots, u_L^\tT)^\tT$, $I_L$ denotes the $L \times L$ identity matrix, $\otimes$ represents the Kronecker product, 
and $K$ is the forward projection operator. The given block matrix $\Sigma$ is the covariance matrix representing the (multivariate) Gaussian distribution, 
where the off-diagonal block elements describe the inter-sinogram correlations, and can be estimated, e.g., using the inverse of the Fisher information matrix \cite{Roessl2009,Schirra2013}. 
Since $f$ is computed from a common set of measurements, the correlation of the decomposed data is very high and thus a significant improvement can in fact be expected intuitively by exploiting the fully populated covariance matrix $\Sigma$ in \eqref{eq:PWLSSpectralCT}.
In the following, we exemplary show reconstruction results on spectral CT data where \eqref{eq:PWLSSpectralCT} was solved by a proximal ADMM algorithm with a material 
independent total variation penalty function $R$ as discussed in \cite{Sawatzky2014}. For a discussion why ADMM based methods are more preferable for PWLS problems in X-ray CT than, e.g., gradient descent based techniques,
 we refer to \cite{Ramani2012}.


Figures \ref{fig:SpectralCT} and \ref{fig:SpectralCT_Yb} show an example for a statistical image reconstruction method applied to K-edge imaging.
A numerical phantom as shown in Figure \ref{fig:phantomSpectralCT} was employed in a spectral CT simulation study assuming 
a photon-counting detector. Using an analytical spectral attenuation model, spectral measurements were computed. 
The assumed X-ray source spectrum and detector response function of the photon-counting detector were identical to those employed in a prototype spectral CT scanner described in \cite{Schlomka2008}. 
The scan parameters were set to tube voltage 130 kVp, anode current 200 $\mu$A, detector width/height 946.38/1.14 mm, number of columns 1024, source-to-isocenter/-detector distance 570/1040 mm, views per turn 1160, time per turn 1 s, and energy thresholds to 25, 46, 61, 64, 76, and 91 keV. 
The spectral data were then decomposed into 'photo-electric absorption', 'Compton effect', and 'ytterbium' by performing a maximum-likelihood estimation \cite{Roessl2007}. 
By computing the covariance matrix $\Sigma$ of the material decomposed sinograms via the Fisher information matrix \cite{Roessl2009,Schirra2013} and treating the sinograms as the mean and $\Sigma$ as the variance of a Gaussian random vector, 
noisy material sinograms were computed. Figures \ref{fig:SpectralCT} and \ref{fig:SpectralCT_Yb} show material images that were then reconstructed using the traditional filtered backprojection (upper row) 
and proximal ADMM algorithm as described in \cite{Sawatzky2014} (middle and lower row). In the latter case, two strategies were performed: (1) 
keeping only the diagonal block elements of $\Sigma$ in \eqref{eq:PWLSSpectralCT} and thus neglecting cross-correlations and decoupling the reconstruction of material images (middle row); 
(2) using the fully populated covariance matrix $\Sigma$ in \eqref{eq:PWLSSpectralCT} such that all material images have to be reconstructed jointly (lower row). The results suggest, best visible in the K-edge images in Figure \ref{fig:SpectralCT_Yb}, 
that the iterative reconstruction method, which exploits knowledge of the inter-sinogram correlations, produces images that possess a better reconstruction quality. For comparison of iterative reconstruction strategies, the regularization parameters were chosen manually so the reconstructed images possessed approximately the same variance within the region indicated by the dotted circle in Figure \ref{fig:phantomSpectralCT}. Further (preliminary) 
results that demonstrate advantages of exploiting inter-sinogram correlations on computer-simulated and experimental data in spectral CT can be found in \cite{Sawatzky2014,Zhang2012}.

\begin{figure}%
	\centering
	\includegraphics[width=.5\textwidth]{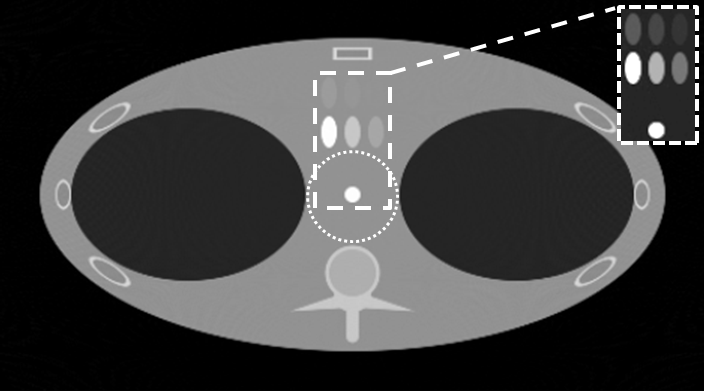}%
	\caption{Software thorax phantom comprising sternum, ribs, lungs, vertebrae, and one circle and six ellipsoids containing different concentrations 
of K-edge material ytterbium \cite{Pan2012}. The phantom was used to simulate spectral CT measurements with a six-bin photon-counting detector.}%
	\label{fig:phantomSpectralCT}%
\end{figure}

\begin{figure}%
	\centering
		\includegraphics[width=.49\textwidth]{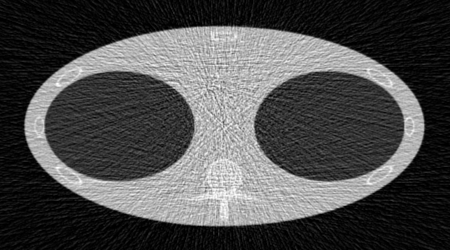}%
		\hfil
		\includegraphics[width=.49\textwidth]{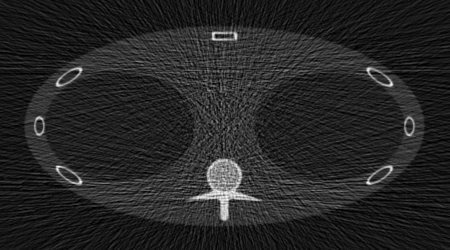}%
		\\[.1em]
		\includegraphics[width=.49\textwidth]{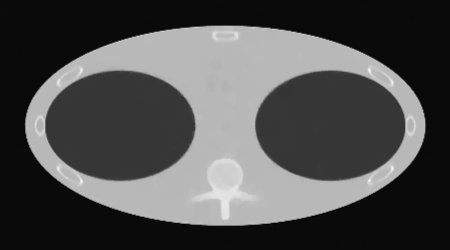}%
		\hfil
		\includegraphics[width=.49\textwidth]{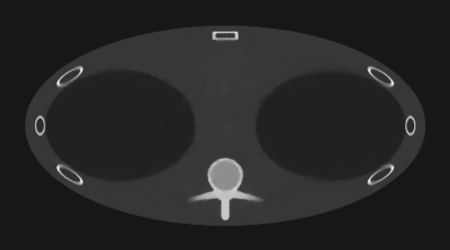}%
		\\[.1em]
		\includegraphics[width=.49\textwidth]{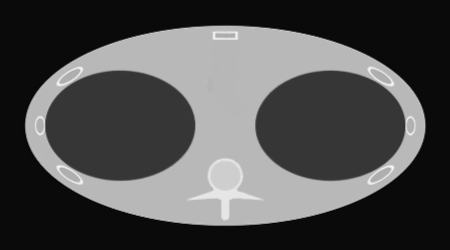}%
		\hfil
		\includegraphics[width=.49\textwidth]{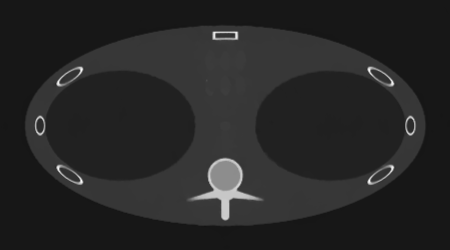}%
	\caption{Reconstructions based on the thorax phantom (see Figure \ref{fig:phantomSpectralCT}) using the traditional filtered backprojection with Shepp-Logan filter (upper row) and a proximal 
ADMM algorithm as described in \cite{Sawatzky2014} (middle and lower row). The middle row shows results based on \eqref{eq:PWLSSpectralCT} neglecting cross-correlations between the material 
decomposed sinograms and lower row using the fully populated covariance matrix $\Sigma$.
The material images show the results for the 'Compton effect' (left column) and 'photo-electric absorption' (right column). The K-edge material 'ytterbium' is shown in Figure \ref{fig:SpectralCT_Yb}.}%
	\label{fig:SpectralCT}%
\end{figure}

\begin{figure}%
	\centering
		\includegraphics[width=.49\textwidth]{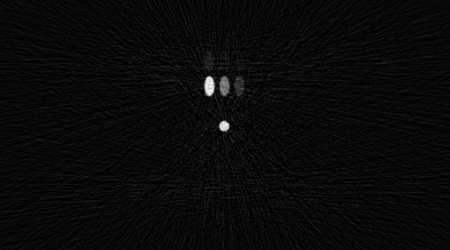}%
		\hfil
		\includegraphics[width=.49\textwidth]{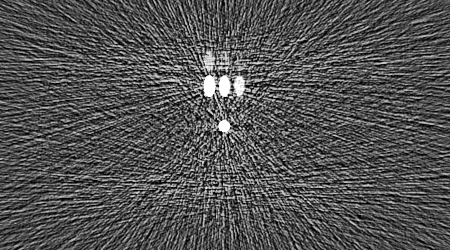}%
		\\[.1em]
		\includegraphics[width=.49\textwidth]{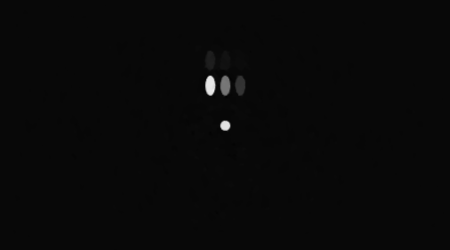}%
		\hfil
		\includegraphics[width=.49\textwidth]{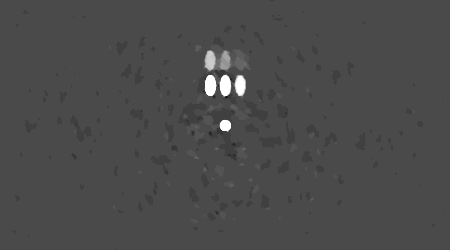}%
		\\[.1em]
		\includegraphics[width=.49\textwidth]{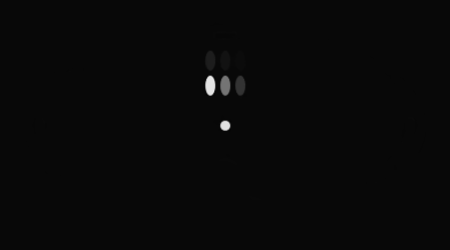}%
		\hfil
		\includegraphics[width=.49\textwidth]{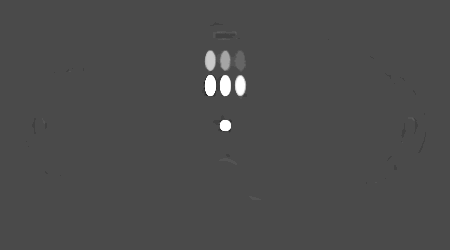}%
	\caption{Reconstructions of the K-edge material 'ytterbium' using the thorax phantom shown in Figure \ref{fig:phantomSpectralCT}. For details see Figure \ref{fig:SpectralCT}. 
To recognize the differences, the maximal intensity value of original reconstructed images shown in left column was set down in the right column.}%
	\label{fig:SpectralCT_Yb}%
\end{figure}


\paragraph{Acknowledgements} The authors thank Frank W\"ubbeling (University of M\"unster, Germany) 
for providing the Monte-Carlo simulation for the synthetic 2D PET data.
The authors also thank Thomas Koehler (Philips Technologie GmbH, Innovative Technologies, Hamburg, Germany) for comments that improved the manuscript. The work on spectral CT was performed 
when A. Sawatzky was with the Computational Bioimaging Laboratory at Washington University in St. Louis, USA, 
and was supported in part by NIH award EB009715 and funded from Philips Research North America.

\bibliographystyle{abbrv}
\bibliography{IEEEabrv,ref_algs_final}

\end{document}